\RequirePackage{fix-cm}
\documentclass{amsart}

\usepackage{graphicx}
\usepackage{mathrsfs}
\usepackage{amsfonts,amsmath}
\usepackage{color}
\usepackage{hyperref}

\renewcommand{\P}{{\mathbb P}}
\renewcommand{\L}{{\mathcal L}}
\newcommand{\Sc}{{\mathcal S}}

\newcommand{\R}{{\mathbb R}}

\newcommand{\N}{{\mathbb N}}

\newcommand{\I}{{\mathbb I}}

\newcommand \A[1]{{\bf (#1)}}

\def\i{{\mathbf{i}}}

\def\x{{\mathbf{x}}}
\def\y{{\mathbf{y}}}
\def\z{{\mathbf{z}}}
\def\p{{\mathbf{p}}}
\def\q{{\mathbf{q}}}
\def\s{{\mathbf{s}}}

\def\bxi{{\boldsymbol{\xi}}}
\def\0{{\mathbf{0}}}

\def\bzeta{{\boldsymbol{\zeta}}}

\def\btheta{{\boldsymbol{\theta}}}

\newcommand{\E}{{\mathbb E}}
\newcommand{\M}{{\mathbb M}}

\newcommand{\F}{{\mathcal F}}

\newcommand{\leftB}{{{[\![}}}
\newcommand{\rightB}{{{]\!]}}}
\renewcommand{\S}{{\mathbb S}}
\newcommand{\LAL}{\Delta^{\frac \alpha 2,A,1}}
\newcommand{\LALS}{\Delta^{\frac \alpha 2,A,1,s}}
\newcommand{\LALB}{\Delta^{\frac \alpha 2,A,1,l}}

\theoremstyle{plain}
\newtheorem{theorem}{Theorem}[section]
\newtheorem{lemma}[theorem]{Lemma}
\newtheorem{proposition}[theorem]{Proposition}

\newtheorem{definition}[theorem]{Definition}

\theoremstyle{definition}
\newtheorem{remark}{Remark}[section]

\newcommand{\mysection}{\setcounter{equation}{0} \section}

\def\nero{\textcolor{blue}}
\def\nero{\textcolor{black}}

\def\LAL#1#2{\Delta^{\frac {\alpha_{#1}} 2,A,#1,#2}}

\def\LALS#1#2{\Delta^{\frac {\alpha_{#1}} 2,A,#1,#2,s}}
\def\LALB#1#2{\Delta^{\frac {\alpha_{#1}} 2,A,#1,#2,l}}

\begin{document}

\title{ $L^p $ Estimates for  Degenerate Non-local Kolmogorov Operators
}
\author{L. Huang}

\address{Michigan State University,
East Lansing (MI). United States 
}

\email{huanglor@stt.msu.edu}

\author{ S. Menozzi}
\address{ Universit\'e d'Evry Val-d'Essonne, 23 Boulevard de France, 91037 Evry and Higher School of Economics, National Research University, Shabolovka 31, Moscow, Russian Federation}
\email{stephane.menozzi@univ-evry.fr}

\author{E. Priola}
\address{Universit\`a di Torino, Dipartimento di Matematica ``Giuseppe Peano", via Carlo Alberto 10, Torino, Italy}
\email{enrico.priola@unito.it}

\date{\today}

\begin{abstract}
%
%
%
%

Let $z = (x,y) \in \R^d \times \R^{N-d}$, with $1 \le d < N$. We prove a priori estimates of the following type :
\[
\|\Delta_{x}^{\frac \alpha 2} v \|_{L^p(\R^N)} \le
c_p
\Big \| L_{x } v  +  \sum_{i,j=1}^{N}a_{ij}z_{i}\partial_{z_{j}} v \Big \|_{L^p(\R^N)},
 \;\; 1<p<\infty,
\]
for $v \in C_0^{\infty}(\R^N)$,
where $L_x$ is a non-local operator comparable with the $\R^d $-fractional Laplacian $\Delta_{x}^{\frac \alpha 2}$ in terms of symbols, $\alpha \in (0,2)$. 
We require that when $L_x$ is replaced by the classical $\R^d$-Laplacian $\Delta_{x}$, i.e.,  in the limit local case $\alpha =2$,  the operator
$ \Delta_{x} + \sum_{i,j=1}^{N}a_{ij}z_{i}\partial_{z_{j}} $  satisfy  
a  weak type H\"ormander condition with invariance by suitable dilations. {Such} estimates were only known 
for  $\alpha =2$.
 This is
one of the first results on $L^p $ estimates for degenerate non-local operators under H\"ormander type conditions. 
We complete our result on $L^p$-regularity for $ L_{x }   +  \sum_{i,j=1}^{N}a_{ij}z_{i}\partial_{z_{j}} $  by proving 
estimates like
\begin{equation*} \label{new}
\|\Delta_{y_i}^{\frac {\alpha_i} {2}} v \|_{L^p(\R^N)} \le
c_p
 \Big \| L_{x } v  +  \sum_{i,j=1}^{N}a_{ij}z_{i}\partial_{z_{j}} v \Big \|_{L^p(\R^N)},
\end{equation*}
involving  fractional Laplacians in the degenerate directions $y_i$ (here $\alpha_i \in (0, { {1\wedge \alpha}})$ depends  on $\alpha $ and  on the numbers of commutators needed to obtain the $y_i$-direction). The last estimates   are new even in the local limit case $\alpha =2$ which is also considered. 


\end{abstract}


\maketitle

\mysection{Introduction and Setting of the problem
} \label{intro}

  
%
 We prove global $L^p$-estimates of Calder\'on-Zygmund type   for degenerate non-local Kolmogorov operators (see \eqref{OP_SPAZIALE} below) acting on regular functions defined on $\R^N$.  This is one of the first results on $L^p $ estimates for degenerate non-local operators under H\"ormander type conditions. To prove the  result we combine analytic and probabilistic techniques.  Our estimates allow
to address corresponding martingale problems or to study related parabolic Cauchy problems with $L^p$ source terms. 
%
%
%
%
%

In particular, we  consider operators which are  sums  of a fractional {like} Laplacian acting only on some (non-degenerate) variables plus a first order linear term acting on all $N$ variables  which satisfies 
a  weak type H\"ormander condition with invariance by suitable dilations (see, for instance examples \eqref{KOLM_OP}, \eqref{r44}, \eqref{r44_COUPLED} below). We obtain maximal $L^p$-regularity with respect to the Lebesgue measure both in  the non-degenerate variables and in the remaining degenerate variables. In the degenerate variables our estimates are new even in the well-studied limit local case when the fractional like Laplacian is replaced by the  Laplacian. They are also related to  well-known estimates by  Bouchut \cite{Bo02} on transport kinetic equations involving only one commutator; our  estimates  depend  on the number of commutators one needs to obtain the degenerate directions.

We stress that  $L^p$ estimates for non-local non-degenerate L\'evy type operators  have been investigated for a long time  also motivated by applications to  martingale problems. 
Significant works in that direction are for instance \cite{stro:75},  \cite{koch:89},  \cite{miku92}, 
  \cite{hoh:94},  \cite{babo07},  
 \cite{DoKi12},  \cite{zhan:13},  \cite{klk16}.
Such operators also naturally appear in  Physical applications for  the study of anomalous diffusions  (see e.g. \cite{mbb01}). 
 However, the corresponding degenerate problems have been rarely addressed and our current work can be seen as a first step towards the understanding of regularizing properties, of hypoellipticity type, for degenerate 
non-local operators satisfying a weak type H\"ormander condition.




To introduce   our setting
let $1 \le d < N$ and consider first the following non-local operator on $\R^d$:
\begin{equation}
\label{NON_DEG_OP}
L_\sigma \phi (x)=\int_{\R^d}\Big(\phi(x+\sigma y)-\phi(x)-\nabla \phi(x)\cdot \sigma y \I_{|y|\le 1}\Big) 
\nu(dy),  \;\; x \in \R^d, 
\end{equation}
where  ${\I}_{ |y| \le 1 }$ is the indicator function of the unit  ball,
  the function  $\phi\in C_0^{\infty}(\R^d) $ (i.e., $\phi : \R^d \to \R$ is infinitely differentiable with compact support)  and  the matrix $\sigma \in \R^d\otimes \R^d $ satisfies the non-degeneracy assumption:
\begin{trivlist}
\item[\A{UE}] There exists $\kappa\ge 1$ s.t. for all $x\in \R^d$,
$$ \kappa^{-1}|x|^2\le \langle \sigma\sigma^* x,x\rangle \le \kappa |x|^2,$$ 
\end{trivlist}
where  $|\cdot|$ denotes the Euclidean norm, $\langle \cdot, \cdot \rangle $ or $\cdot$ is the inner product and  ${}^* $ stands for the transpose; $\nu $ is a
 non-degenerate
symmetric
stable
L\'evy measure
of order $\alpha\in (0,2) $.  Precisely, writing
 $y=\rho s,\ (\rho,s)\in \R_+\times \S^{d-1} $, where $\S^{d-1}$ stands for the unit sphere of $\R^d$, the measure $\nu $ decomposes as:
 \begin{equation}
\label{STABLE_MEAS}
\nu(dy)=\frac{\tilde \mu(ds) d\rho}{\rho^{1+\alpha}},
 \end{equation}
where $\tilde \mu $ is a Borel finite measure on $\S^{d-1} $ which is the spherical part  of $\nu$.  If $\sigma = I_{d\times d } $ (identity matrix of $\R^d $) and $\tilde \mu$ is proportional to the surface measure then $L_{I_{d\times d}}$ coincides with the fractional Laplacian on $\R^d $ with symbol $- |\lambda|^{\alpha}$.
The L\'evy symbol  associated with $L_{\sigma}$ is given by the L\'evy-Khintchine formula
$$
\Psi(\lambda)= 
 \int_{\R^d} \Big(  e^{i \langle \sigma^* \lambda, y \rangle }  - 1 - \, { i \langle \sigma^* \lambda,y
\rangle} \, {\I}_{ |y| \le 1 } \, (y) \Big ) \nu (dy),\;\; \lambda \in \R^d
$$
(see, for instance,  \cite{sato},  \cite{Jacob1} and  \cite{appl:09}).
From  Theorem 14.10 in \cite{sato}, we know that 
\begin{equation*}
\Psi(\lambda)=- \int_{\S^{d-1}} |\langle \sigma^*\lambda,s\rangle|^\alpha \mu(ds),
\end{equation*}
where $\mu = C_{\alpha,d} \tilde{\mu}$ for a positive constant $C_{\alpha,d}$. The spherical measure $\mu $ is called the \textit{spectral measure} associated with $\nu $. We suppose that $\mu $ is non-degenerate in the sense of \cite{wata:07}:
 \begin{trivlist}
\item[\A{ND}]  There exists $\eta\ge1$  s.t. for all $\lambda\in \R^d $,
\begin{equation}
\label{EQ_ND}
\eta^{-1}|\lambda|^\alpha \le  \int_{\S^{d-1}} |\langle \lambda,s\rangle|^\alpha \mu(ds) \le \eta|\lambda|^\alpha.
\end{equation}
 \end{trivlist}
 We now introduce our (complete) \textit{Kolmogorov} operator in the following way. Let $\x=(\x_1,\cdots,\x_n)\in \R^{N}$, $n \ge 1$, where each $\x_i\in \R^{d_i} $, with $d_1 =d$,
$$
d_1 \ge d_2 \ge .... \ge d_n > 0, \;\;\; N = d_1 + d_2 + ....  + d_n.
$$
We define for a non-degenerate 
matrix $\sigma\in \R^d\otimes \R^d $  satisfying \A{UE} and $\varphi \in C_0^{\infty}(\R^{N}) $ the following operator:
\begin{equation}
\label{OP_SPAZIALE}
\L_\sigma \varphi (\x)=\langle A \x ,\nabla_\x \varphi(\x)\rangle+L_\sigma \varphi(\x), \;\;\;\; \x \in \R^N,
\end{equation}
where 
\begin {equation} \label{cia1}
L_\sigma \varphi(\x)=\int_{\R^d} \Big(\varphi(\x+B\sigma y)-\varphi(\x)- \nabla_{\x_1} \varphi(\x) \cdot \sigma y \I_{|y|\le 1}\Big)\nu(dy) =  \text{v.p.} \int_{\R^d} \Big(\varphi(\x+B\sigma y)-\varphi(\x)\Big)\nu(dy),
\end{equation}
and $B=\left(\begin{array}{c}
I_{d_1\times d}\\
0_{d_2\times d}\\
\vdots\\
0_{d_n\times d}
\end{array} \right) $ is the embedding matrix from $\R^d $ into $\R^N $.
Also, the matrix $A\in \R^{N}\otimes \R^N$ has the following structure  (cf. examples \eqref{KOLM_OP}, \eqref{r44} and \eqref{r44_COUPLED}):
\begin{eqnarray}\label{DEF_A}
A=\left( \begin{array}{ccccc}
 0_{d\times d}& \cdots& \cdots &\cdots  & 0_{d\times d_n}\\
A_{2,1} & 0_{d_2\times d_2} &\cdots &  \cdots &0_{d_2\times d_n}
\\
 0_{d_3\times d}& A_{3,2}& \ddots &  \cdots&0_{d_3\times d_n}\\
 0_{}&  \ddots   &\ddots& \ddots              & \vdots\\
 0_{d_n \times d}& \cdots & 0_{d_n\times d_{n-2}}              &A_{n,n-1}& 0_{d_n\times d_n}
 \end{array}
\right).
\end{eqnarray}
The only non-zero entries are the matrices $ \big( A_{i,i-1}\big)_{i\in \leftB 2,n\rightB} $\footnote{We use from now on the notation $\leftB \cdot,\cdot\rightB $ for integer intervals.}. We require that    $A_{i,i-1} \in \R^{d_i}\otimes \R^{d_{i-1}} $. Moreover we assume that
$A$ satisfies the following non-degeneracy condition of H\"ormander type:
\begin{trivlist}
\item[\A{H}] The $\big( A_{i,i-1}\big)_{i\in \leftB 2,n\rightB}$ are non-degenerate (i.e., each
$ A_{i,i-1}$ has rank $d_i$).
\end{trivlist}
 According to  \cite{lanc:poli:94} (see also \cite{bram:cerr:manf:96} and  \cite{dela:meno:10}) the previous conditions \A{UE} and \A{H} imply in the limit case $\alpha =2$, i.e., when $L_\sigma $ is a second order differential operator like $\frac{1}{2}{\rm Tr}(\sigma \sigma^* D^2_{\x_1})$, the well-known  H\"ormander's  hypoellipticity condition for ${\mathcal L}_{\sigma}$ \textit{involving $n-1$ commutators} starting from $\sigma D_{\x_1}$ and $\langle A \x ,\nabla_\x \cdot \rangle$  (cf. \cite{horm:67}). Precisely, our conditions on the matrix $A$ are the same as in \cite{bram:cerr:manf:96}. 
Note that in the case $\alpha =2 $    operators ${\mathcal L}_{\sigma}$ are also considered from the control theory point of view  (see \cite{BZ09}, \cite{LM16} and the references therein). 

  In our non-local  framework, assuming additionally \A{ND}, even though  no general H\"ormander theorem seems to hold (cf.  \cite{KoTa01})
the Markov semi-group associated with $ {\mathcal L}_\sigma$ admits a smooth density, see e.g. \cite{prio:zabc:09}.

We say that assumption \A{A} holds when \A{UE}, \A{ND} and \A{H} are in force. In the following, we denote by $C:=C(\A{A})\ge 1$ or $c:=c(\A{A})\ge 1$  a generic constant that might change from line to line and that depends on the parameters  of assumption \A{A}, namely  \nero{on $\alpha \in (0,2)$,} the non degeneracy constants $\kappa, \eta$ in \A{UE}, \A{ND} as well as those of \A{H} and the dimensions  $(d_i)_{i\in \leftB 1,n\rightB}$. Other specific dependences 
are explicitly specified. Let us shortly present some  examples of operators satisfying the previous assumptions \A{A} with $\sigma$ equal to identity:
\begin{trivlist}
\item[-] A \nero{basic} example is given by the following:
\begin{equation}
\label{KOLM_OP}
 {\mathcal L} = \, \Delta_{\x_1}^{\frac \alpha 2} + \x_1  \cdot \nabla_{\x_2},
\end{equation}
where $\x=(\x_1,\x_2) \in \R^{2d}$, $d\ge 1 $; this is the \textit{extension} to the non-local fractional setting of the celebrated Kolmogorov example, see \cite{kolm:34}, that inspired H\"ormander's hypoellipticity theory \cite{horm:67} in the diffusive case (in this case we have $N = 2d$, $n=2$ and $d_1 = d_2 =d$). \textcolor{black}{From the probabilistic viewpoint, ${\mathcal L} $ corresponds to the generator of the couple formed by an isotropic stable process and its integral. Such processes might appear in kinetics/Hamiltonian dynamics when considering the joint distribution of the velocity-position of stable driven particles} \textcolor{black} {(see e.g. \cite{tala:02} in the diffusive case or \cite{cush:park:klei:moro:05} for the non-local one in connection with the Richardson scaling law in turbulence)}.
  Non-local degenerate kinetic diffusion equations  appear as well as diffusion limits of linearized Boltzmann equations when the  equilibrium
distribution function is a heavy-tailed distribution (see \cite{mellet0}, \cite{mellet1} and \cite{rad}).
\item[-] 
Consider now for $d^*\in \N$:
\begin{equation} \label{r44}
{\mathcal L} = \Delta_{\x_1^1}^{\frac \alpha 2} + \Delta_{\x_1^2}^{\frac \alpha2} + 
\x_1^1 \cdot \nabla_{\x_2} + \x_2 \cdot  \nabla_{\x_3},
\end{equation}
for $ \x_1=(\x_1^1,\x_1^2)\in  \R^{2d^*}, \x=(\x_1,\x_2, \x_3)\in \R^{4d^*}$, i.e., $N = 4 d^*$,  $n =3$,  $d_1=2d^*$ and $d_2=d_3 =d^*$. If $d^*=1 $, the non-degenerate part of the above operator corresponds to the so-called cylindrical fractional Laplacian (in this case the  L\'evy measure $\nu$ is
 concentrated on $\{x_1^1=0\} \cup \{ x_1^2 =0 \}$, cf. e.g.    
 \cite{bass:chen:06}). 
\textcolor{black}{\item[-] A natural generalization of the previous example consists in considering:
\begin{equation} \label{r44_COUPLED}
{\mathcal L} = L_\sigma +  
\x_1^1 \cdot \nabla_{\x_2} + \x_2 \cdot  \nabla_{\x_3} ,\ L_\sigma \phi(x)=\sum_{i=1}^2 \text{v.p.} \int_{\R^{d^*}}\Big(\phi(x+\sigma_i z)-\phi(x) \Big)\frac{dz}{|z|^{d^*+\alpha}},
\end{equation}
  $\phi\in C_0^\infty(\R^{2d^*})$, $x\in \R^{2d^*}$, $\sigma_i \in \R^{2d*}\otimes \R^{d*}$ and $\sigma=\big( \sigma_1\ \sigma_2\big) $ verifies \A{UE}. The non-degenerate part $L_\sigma$ of ${\mathcal L}$ corresponds to  the generator of the $\R^{2d^*} $-valued  process ${\mathbf X}_t^1=\x_1+\sum_{i=^1}^2\sigma_i Z_t^i$, where the $(Z^i)_{i\in \leftB 1,2\rightB} $ are two independent isotropic stable processes of dimension $\R^{d*} $.
 The full operator ${\mathcal L} $ is  the generator of $({\mathbf X}_t^1,{\mathbf X}_t^2,{\mathbf X}_t^3)=({\mathbf X}_t^1, \int_0^t {\mathbf X}_s^{11}ds ,\int_0^t\int_0^s {\mathbf X}_u^{11} duds ) $ where for $u>0,\ {\mathbf X}^{11}_u $ denotes the first $d^*$ entries of ${\mathbf X}_u^1$. Such dynamics could  arise in finance, if we for instance assume that ${\mathbf X}^1$ models the evolution of a $2d^* $-dimensional asset, for which each component feels both noises $Z^1$ and $Z^2$ in $\R^{d^*} $ through the matrices $\sigma_1,\sigma_2$, but for which one would be interested in the distribution of the (iterated) averages of some marginals, in the framework of the associated Asian options, here $\big({\mathbf X}^2, {\mathbf X}^3 \big) $. We refer to \cite{baru:poli:vesp:01} in the diffusive setting for details.}
\end{trivlist}

In the case $N=nd,\ n>2 $, oscillator chains also naturally appear for the diffusive case in heat conduction models (see e.g. \cite{eckm:99} and \cite{dela:meno:10} for some related heat kernel estimates). The operators we consider here could also appear naturally in order to study anomalous diffusion phenomena within this framework.



To state our $L^p$ estimates, for  all $i\in \leftB 1,n \rightB $, we introduce   the orthogonal projection  $\pi_i : \R^N \to \R^{d_i}$, $\pi_i (\x) = \x_i$, $\x \in \R^N$. Then we  introduce the adjoint $B_i = \pi_i^* : \R^{d_i} \to \R^N$ (note that $B_1 = B$) and 
define 
 \begin{equation}
\label{DEF_ALPHA_I_B_I1}
 \alpha_i:=\frac{\alpha}{(1+(i-1)\alpha)}; \;\;\; \text{clearly we have} \,\; \alpha_1 = \alpha, \;\;  \;  \alpha_i\in  (0,1\wedge \alpha),\;\;\; i \in  \leftB 2,n\rightB.
 \end{equation} 
\begin{theorem}\label{THE_THM1}
Assume that \A{A} holds. Then, for $v\in C_0^\infty(\R^N) $ and $p\in (1,+\infty)$ there exists $c_p:=c_p(\A{A})$ s.t. for all $i\in \leftB 1,n\rightB $:
\begin{equation}\label{THE_CTR_LP_ELLIP}
\|\Delta_{\x_i}^{\frac {\alpha_i} {2} } v\|_{L^p(\R^N)}\le c_p \,  \|\L_\sigma v\|_{L^p(\R^N)},
\end{equation}
where in the above equation $\Delta_{\x_i}^{\frac{\alpha_i}{2}} $ denotes the $\R^{d_i}$-fractional Laplacian w.r.t. to the $\x_i$-variable, i.e.,
\begin{equation*}
\Delta_{\x_i}^{\frac{\alpha_i}{2}} v(\x) = \text{v.p.} \int_{\R^{d_i}} \big(
v (\x + B_i z) - v(\x)\big) \frac{dz}{|z|^{d_i+\alpha_i}}.
\end{equation*}
\end{theorem}
The above result still holds in the \nero{\it diffusive limit case,} i.e. when $\alpha =2 $ and $L_\sigma=\frac 12 {\rm Tr}(\sigma\sigma^* D_{\x_1}^2) $ is a  second order non-degenerate differential operator in $\x_1 $. This limit case is specifically addressed in Appendix \ref{EST_LOCAL_CASE} below. 
The theorem is  obtained as a consequence of  an $L^p$-parabolic  regularity result of independent interest (see Theorem \ref{THM_STRISCIA}).  Let us comment separately the case $i=1$ (i.e., $\Delta_{\x_1}^{\frac {\alpha_1} {2} } = \Delta_{\x_1}^{\frac {\alpha} {2} }$) and $i =2, \ldots, n$.

The case $i = 2, \ldots, n$ is new even if we consider  the local   case of $\alpha =2$. In that case,   \cite{FL04} considers similar estimates  for $p=2$ in non-isotropic fractional Sobolev spaces with respect to an invariant Gaussian measure assuming that 
 the    hypoelliptic Ornstein-Uhlenbeck operator $\mathcal{L}_{\sigma}$
admits such invariant measure. Also, in the special kinetic case \eqref{KOLM_OP}, the $L^p$-estimate for $\Delta_{\x_2}^{\frac{\alpha}{1+\alpha}} v,\ \alpha \in (0,2] $ can be derived from  Bouchut \cite{Bo02}. In this  work general estimates on the degenerate variable are derived from the transport equation $(\partial_t + \x_1 \cdot \nabla_{\x_2} )u=f$ assuming a priori regularity of $u$ in the non-degenerate variable (i.e. in the current setting  an integrability condition on $ \Delta_{\x_1}^{\frac\alpha 2}u$ or more generally on $L_\sigma u$ with $L_\sigma $ as in \eqref{NON_DEG_OP}) and integrability conditions on $f$. 

The approach we develop here, naturally based on the theory of singular integrals on homogeneous spaces (since we have dilation properties for the operator) allows to derive directly the estimates on \textit{all} the variables exploiting somehow the regularizing properties of the underlying singular kernels. On the other hand, in light of our  theorems, it seems a natural open problem to extend Bouchut's estimates  to the more general transport equation 
  $$(\partial_t + A\x \cdot \nabla_{\x} )u=f $$ 
  with  $A$ as  in \eqref{DEF_A}, 
   when more than one commutator is needed to span the space, assuming   $\Delta_{\x_1}^{\frac {\alpha} {2} } u \in L^p$.

Our estimate \eqref{THE_CTR_LP_ELLIP} in the case $i=1$ 
can be viewed as a non-local  extension of the results by  \cite{bram:cupi:lanc:prio:10}. 
Indeed, the quoted work concerns with  estimates similar to \eqref{THE_CTR_LP_ELLIP} with $i=1$ for a \textit{diffusion} operator, i.e. the limit case  $\alpha=2 $ when $\Delta_{\x_1}^{\frac \alpha 2} $ corresponds to $\Delta_{\x_1}^{} $ and $L_\sigma$  to a second order differential operator
 like $Tr(\sigma \sigma^* D^2_{\x_1})$. 
In this framework, the authors obtained estimates of the following type:
\begin{equation}
\label{cald}
\| D_{\x_1}^ 2 v\|_{L^p(\R^N)}
\le c_p\Big( \|\L_\sigma v\|_{L^p(\R^N)}+ \|v\|_{L^p(\R^N)}\Big),
\end{equation}
(here $D_{\x_1}^2 v$ is the Hessian matrix of $v$ with respect to the $\x_1$-variable).  Note that
 by
the classical Calder\'on-Zygmund theory: $\| D_{\x_1}^ 2 v\|_{L^p(\R^N)} $ $ \le C_p \| \Delta_{\x_1} v\|_{L^p(\R^N)}$. Hence \eqref{THE_CTR_LP_ELLIP} for $i=1$ and $\alpha =2$ can be reformulated as 
$$
\| D_{\x_1}^ 2 v\|_{L^p(\R^N)} \le  c_p \,  \| \L_\sigma v \|_{L^p(\R^N)}.
$$
 Note that \eqref{cald} has an extra term $\|v\|_{L^p(\R^N)}$ on the right-hand side. This is due to the fact that 
  our structure of the matrix $A$ is more restrictive than in \cite{bram:cupi:lanc:prio:10} and \cite{FL04} (cf Remark \ref{chen}). On the other hand,  by our assumptions on $A$  we can use  the theory of \textit{homogeneous spaces} with the doubling property as in \cite{coif:weis:71}.

In contrast with previous works (see in particular   \cite{bram:cerr:manf:96} and \cite{bram:cupi:lanc:prio:10}),  we do not use here an underlying Lie  group structure. Another difference is that     we are able to prove directly the important $L^2$-estimates (see Lemma \ref{LEMME_L2}).
  This is why in contrast \nero{with \cite{bram:cerr:manf:96} } 
we can entirely rely on the Coifmann-Weiss setting \cite{coif:weis:71}. 
  

  Let us mention that results similar to 
 Theorem \ref{THE_THM1} have been obtained in \cite{CZ16} in the special kinetic case of \eqref{KOLM_OP}. 
  Their strategy is totally different and relies on the 
  Fefferman and Stein  approach \cite{feff:stei:72} for the non-degenerate variable. The estimate for the degenerate one
  is then derived from Bouchut's estimates \cite{Bo02}. 
We avoid here diagonal terms and time dependence in $A$ in \eqref{DEF_A}  for simplicity (see the extensions discussed in Appendix A of the preliminary preprint version  \cite{huan:meno:prio:16}).
 With zero diagonal, considering non-degenerate time dependent coefficients in \eqref{DEF_A}  would not change the analysis but make the notations in Section \ref{HOMO} below more awkward.
 Adding diagonal terms, even time-homogeneous,   would lead to time dependent constants in our parabolic estimates of Theorem \ref{THM_STRISCIA}.
Anyhow, introducing additional strictly  super-diagonal elements in $A$
breaks the homogeneity (see Section \ref{HOMO} and Appendix A of \cite{huan:meno:prio:16} for details). This seemingly small modification actually induces to consider estimates from harmonic analysis in non-doubling
spaces developed in a rather abstract setting by \cite{bram:10}. This is precisely the approach developed in \cite{bram:cupi:lanc:prio:10}. 
Handling a general matrix $A$ in the current framework will concern further research.
We finally point out that our approach also permits to recover the estimates of the non-degenerate case, i.e. when $n=1,N=d$ and $A=0$.


The article is organized as follows. We first discuss in Section \ref{HOMO} the appropriate homogeneous framework, depending on the index $\alpha \in (0,2]$, needed for our analysis. 
Importantly, 
we manage to express the fundamental solution of  $\partial_t-{\mathcal L}_{\sigma}$, 
in terms of the density at time $t>0$ of a 
non-degenerate stable process $(S_t)_{t \ge 0 }$ on $\R^N$ rescaled according to the homogeneous scales (see Proposition \ref{THE_PROP_FUND}) 
\footnote{The stable process $S$ is non-degenerate in the sense that its spectral measure $\mu_S$ satisfies \eqref{EQ_ND} on $\R^N $.
However, we point out that $\mu_S $ can be very singular w.r.t. to the surface measure of $\S^{N-1} $.}.
We eventually state our second main result Theorem \ref{THM_STRISCIA}, which is the parabolic version of Theorem \ref{THE_THM1}, and actually permits to prove Theorem \ref{THE_THM1}.

 We then describe in Section \ref{STRAT_PROOF} the strategy of the proof of Theorem \ref{THM_STRISCIA} 
relying on the theory  of  Coifman and Weiss 
(see Appendix \ref{APP_CW}).  
Section \ref{THE_TEC_SEC} is dedicated  to the key estimates ($L^2 $ bound and controls of the singular kernels) needed for the proof of our main results. This is the technical core of this paper.  Rewriting the fundamental solution of $\partial_t- {\mathcal L}_{\sigma}$ in terms of  the  density of the rescaled  stable process $(S_t)_{t \ge 0 }$  allows to analyse the terms $(\| \Delta_{\x_i}^{\frac {\alpha_i} {2} } v\|_{L^p(\R^N)})_{ i\in \leftB 1,n\rightB} $, through singular integral analysis techniques that  exploit the integrability properties of any non-degenerate stable density, 
no matter how singular the spectral measure is (see Lemma \ref{SENS_SING_STAB}). We use 
in Section  \ref{KEY_SECTION} the estimates of Section \ref{THE_TEC_SEC} in order to fit our specific framework following the strategy of Section \ref{STRAT_PROOF}. 
Some technical points  are proved for the sake of completeness in Appendix \ref{APP_TEC}. \textcolor{black}{The specific features associated with the limit local case $\alpha=2 $ are discussed in Appendix \ref{EST_LOCAL_CASE}}. 
 Finally, \textcolor{black}{as indicated at the beginning of the introduction}, we mention that possible applications of our estimates to well-posedness of  martingale problems for operators like 
$${\mathcal L}_{\sigma(\x)}\varphi(\x)=\langle A\x, \nabla_\x \varphi(\x)\rangle+L_{\sigma(\x)}\varphi(\x)
$$
will be a subject of a future work (cf. Appendix A in \cite{huan:meno:prio:16} for preliminary results in this direction).


\mysection{Homogeneity Properties}
\label{HOMO}

The key point to establish the estimates in Theorem \ref{THE_THM1} consists in first considering the parabolic setting. To this end, 
we consider the evolution operator defined for 
$\psi\in C_0^{\infty}(\R^{1+N}) $ by: 
\begin{equation}
\label{OP_PARAB}
{\mathscr L}_\sigma \psi(t,\x):= (\partial_t+{\mathcal L}_\sigma 
 ) \psi(t,\x),\ (t,\x)\in \R^{1+N}.
\end{equation}
The following proposition is fundamental and follows from the structure of $A$ and  $L_\sigma $ under \A{A}.
\begin{proposition}[Invariance by dilation]\label{PROP_INV_DIL}
Let \A{A} be in force and $u\in C_0^\infty(\R^{1 + N})$ solve the equation $ {\mathscr L}_\sigma u(t,\x)=0, (t,\x)\in 
\R^{1 + N}$. Then, for any $\delta >0 $, the function $u_\delta(t,\x):=u(\delta^\alpha t, \delta \x_1, \delta^{1+\alpha}\x_2,\cdots, \delta^{1 + (n-1)\alpha}\x_n) $ also solves $ {\mathscr L}_\sigma u_\delta(t,\x)=0, (t,\x)\in  \R^{1 + N}$.
\end{proposition}
\begin{proof}
It is actually easily checked that: 
$${\mathscr L}_\sigma u_\delta( t,\x
)=\delta^\alpha \big({\mathscr L}_\sigma u(\z_\delta) \big)\big|_{\z_\delta=(\delta^\alpha t, \delta \x_1, \delta^{1+\alpha}\x_2,\cdots, \delta^{1+ (n-1)\alpha}\x_n)}.$$
Since for all $\z\in  \R^{1+ N},\ {\mathscr L}_\sigma u(\z)=0 $, the result follows.
\end{proof}
The previous proposition leads us to define the following \textit{homogeneous} norm (cf. \cite{bram:cerr:manf:96} or \cite{bram:cupi:lanc:prio:10}):
\begin{definition}[Homogeneous Pseudo-Norm] Let $\alpha \in (0,2].$
We define for all  $(t,\x)\in \R^{1+N} $ the pseudo-norm:
\begin{equation}\label{rho_homo}
\rho 
(t,\x) = |t|^{\frac{1}{\alpha}} + \sum_{i=1}^{n} |\x_i|^{\frac{1}{1+\alpha(i-1)}}.
\end{equation}
\end{definition}
\begin{remark}
\label{REM_HOMO}
Let us observe that $\rho$ is \textit{not} a norm because the \textit{homogeneity} property fails, i.e. for $\delta>0, \rho(\delta t,\delta \x)\neq \delta \rho(t,\x)  $. Actually, the homogeneity appears through the dilation operator of Proposition \ref{PROP_INV_DIL}. Namely, for $\z=(t,\x)\in \R^{1+N} $ and $\z_{\delta}= (\delta^\alpha t, \delta \x_1, \delta^{1+\alpha}\x_2,\cdots, \delta^{1+ (n-1)\alpha}\x_n)$ we indeed have:
$\rho(\z_\delta)=\delta \rho(\z).$

The dilation operator can also be rewritten using the \textit{scale} matrix $\M_t$  in \eqref{DEF_T_ALPHA}. Namely, 
$\z_\delta= \Big( \delta^\alpha t, \delta \M_{\delta^{\alpha}} \x \Big)$.
\end{remark}

\begin{remark}[Regularity of ${\mathcal L}_{\sigma}v$]
\label{reg1} 
 If $v \in C_0^{\infty} (\R^N)$ it is not difficult to prove that ${\mathcal L}_{\sigma}v \in C^{\infty} (\R^N)$. Moreover, ${\mathcal L}_{\sigma}v$ and all its partial derivatives belong to $\cap_{p \in [1, \infty]}L^p (\R^N)$.

 We only check that, when $p \in [1, +\infty)$,  ${\mathcal L}_{\sigma}v \in L^p(\R^N)$.
 Since $\langle A \x ,\nabla_\x v \rangle \in L^p(\R^N)$ we concentrate on $L_\sigma v$ (see \eqref{cia1}). We can write
\begin{gather*}
L_\sigma v = v_1 + v_2,
 \;\;\;
v_1(\x) = \int_{|y|\le 1} \frac{ \big(v(\x+B\sigma y)-v(\x)- \nabla_{\x_1} v(\x) \cdot \sigma y \big)} {|y|^2} |y|^2\nu(dy),
\\
v_2(\x) = \int_{|y| >1} \big(v(\x+B\sigma y)-v(\x) \big)\nu(dy),
 \;\; \x \in \R^N.
\end{gather*}
We easily obtain that $v_1 \in C_0^{\infty} (\R^N)$. Moreover, from the H\"older inequality and the Fubini theorem, we get
$$
\int_{\R^N} |v_2(\x)|^p d\x \le c  \int_{|y| >1}  \nu(dy)  \int_{\R^N} |v(\x+B\sigma y)-v(\x) |^p d\x < \infty.
$$ 
\end{remark}
We now study  the homogeneous framework 
 of
the Ornstein-Uhlenbeck process  $(\Lambda_t)_{t\ge 0} $ satisfying the stochastic differential equation (SDE):
\begin{equation}\label{EDS_OU}
d\Lambda_t = A\Lambda_t dt+ B \sigma dZ_t, \;\;  \; \Lambda_0=\x\in \R^N,\\
\end{equation}
where $B$
again stands for the embedding matrix from $\R^d $ (the space where the noise lives) into $\R^N $.
 Here $\big(Z_t\big)_{t\ge 0}$ is a stable $\R^d$-dimensional process with L\'evy measure $\nu$ defined on some complete probability space $(\Omega,\F,\P) $.
For a given starting point $\x\in \R^N$, the above dynamics explicit integrates and gives:
\begin{equation}
\label{INTEGRATED_OU}
\Lambda_t^\x=\exp\big(tA\big)\x+\int_0^t \exp\big((t-s)A\big) B\sigma dZ_s.
\end{equation}
It is readily derived from \cite{prio:zabc:09} that, for $t>0$ the random variable $\Lambda_t$ has a density $p_{\Lambda}(t,\x,\cdot) $ w.r.t. the Lebsegue measure of $ \R^N$. Additionally, we derive in \eqref{THE_DENS_OU} of Proposition \ref{THE_PROP_FUND} below (similarly to Proposition 5.3 of \cite{huan:meno:15} which is stated in a more general time-dependent coefficients framework) that 
\begin {equation} \label{CORRESP_L_S}
p_{\Lambda}(t,\x, \y) =\frac{1}{\det(\M_t)}p_{S}\big(t,(\M_t)^{-1}( e^{tA}\x-\y)\big),\ t>0, 
\end{equation}
where the diagonal matrix
\begin{equation}
\label{DEF_T_ALPHA}
\M_t=\left( \begin{array}{cccc}
I_{d\times d}& 0_{d\times d_2}& \cdots&0_{d\times d_n}\\
0_{d_2\times d}   &tI_{d_2\times d_2}&0_{}& \vdots\\
\vdots & \ddots&\ddots & \vdots\\
0_{d_n\times d}& \cdots & 0_{}& t^{n-1}I_{d_n \times d_n}
\end{array}\right),
\end{equation}
gives the \textit{multi-scale} characteristics of the density of $\Lambda_t $ and
$(S_t)_{t\ge 0}$ is a stable process in $\R^N$ whose  L\'evy measure 
$\nu_{S}$ though having a very singular spherical part, satisfies the non-degeneracy assumption \A{ND} in $\R^N$. From a more analytical viewpoint the entries of $t^{\frac{1}{\alpha}}\M_t $, which correspond to the typical scales of a stable process and its iterated integrals, provide the underlying homogeneous structure. 

\nero{   
To prove our results 
we will often use  the following rescaling identity
\begin{equation} \label{str}
e^{rA} = \M_r  e^{A} \M_r^{-1}, \;\;\; r>0, 
\end{equation}
and its adjoint form $e^{rA^*} = \M_r^{-1}  e^{A^*} \M_r $.
To get \eqref{str} we first check directly  that $ \M_r  {A} \M_r^{-1} = r A$, $r>0$; then we have the assertion writing 
$$
e^{rA} =  I + rA + \frac{r^2}{2} A^2 + \ldots  
= 
I + \M_r  {A} \M_r^{-1} + \frac{ 1 
}{2} \M_r  {A} \M_r^{-1} \M_r  {A} \M_r^{-1} + \ldots. 
$$
}
The next  result is crucial for our main estimates. 
\begin{proposition}[Density and Fundamental Solution] 
\label{THE_PROP_FUND}
Under \A{A}, the process $(\Lambda_t^\x)_{t\ge 0} $ defined in \eqref{EDS_OU} has for all $t>0$ and starting point $\x\in \R^N $ a $C^\infty(\R^N)$-density $p_\Lambda (t,\x,\cdot) $ that writes for all $\y\in \R^N $:
\begin{eqnarray}
\label{THE_DENS_OU}
 p_\Lambda(t,\x,\y)=\\
 \frac{\det(\M_t^{-1})}{(2\pi)^N}\int_{\R^{N}} \exp\Big(-i\langle \M_t^{-1}(\y-\exp(t A)\x),\p\rangle \Big)\exp\Big(-t\int_{\S^{N-1}}|\langle \p, \bxi\rangle|^\alpha \mu_S(d\bxi)\Big)d\p,\nonumber
\end{eqnarray}
where $\M_t$ is defined in \eqref{DEF_T_ALPHA} and $\mu_S $ is a symmetric spherical measure on $\S^{N-1}$ satisfying the non-degeneracy condition \A{ND} on $\R^N $ instead of $\R^d $.

The analytical counterpart of this expression is that the operator $\partial_t - {\mathcal L}_\sigma $ has a fundamental solution given by $p_\Lambda \in C^\infty(\R_+^*\times \R^{2N}) $. Also, for every $u\in C_0^\infty(\R^{1+N})$, the following representation formula holds:
\begin{equation}
\label{REP}
u(t,\x) 
= - \int_{t}^{\infty} ds\int_{\R^N} p_\Lambda(s-t,\x,\y){\mathscr L}_\sigma u(s,\y)d\y, \;\;  (t,\x) \in \R^{1+N}.
\end{equation}
\end{proposition}

\begin{remark}[A useful identity in law]
\label{ID_LAW}
Let $(S_t)_{t\ge 0}$ be  a (unique in law)  $\R^N $-valued \nero{symmetric $\alpha$-stable} process with spectral measure $\mu_S $ defined on 
 $(\Omega, {\mathcal F}, \P)$, i.e., 
$$
\E [e^{i \langle \p , S_t \rangle}] = \exp\big(-t\int_{\S^{N-1}}|\langle \p, \bxi\rangle|^\alpha \mu_S(d\bxi)\big), \;\; t \ge 0, \; \p \in \R^N.
$$
\nero { By the previous result  we know that $S_t$ 
has a $C^{\infty}$-density $p_S(t,\cdot) $ for $t>0$. Note that $S_t  \overset{({\rm law})}{=} t^{\frac 1\alpha} S_1$ which is equivalent to 
$ p_S(t,\x) $ $= {t^{-\frac N\alpha}} \, p_S (1, t^{-\frac 1\alpha} \x),$ $ t>0, \, $ $\x \in \R^N.$ 
 Moreover,  \eqref{CORRESP_L_S} holds} and is equivalent to \eqref{THE_DENS_OU}.
In a more probabilistic way, this means that for any fixed $t>0$ the following identity in law holds: 
\begin{equation}
\label{d33}
\Lambda_t^\x\overset{({\rm law})}{=}\exp(t A)\x+\M_t S_t.
\end{equation}
\end{remark}


\begin{proof}[Proof of Proposition \ref{THE_PROP_FUND}]
Let $t\ge 0  $ be fixed.
Observe first that for a given $m\in \N$ considering the associated uniform partition $\Pi^m:=\{ (t_i:=\frac{i}{m}t)_{i\in \leftB 0,m \rightB}\}$ of $[0,t] $ yields for all $\p\in \R^{N}$:
\begin{eqnarray*}
\E\Big[\exp\Big(i\langle \p,\sum_{i=0}^{m-1} \exp\big((t-t_i)A\big) B\sigma(Z_{t_{i+1}}-Z_{t_i})\rangle \Big)\Big]\\
=\exp\Big(-\frac 1m\sum_{i=0}^{m-1}\int_{\S^{d-1}}|\langle \sigma^*B^*\exp\big( (t-t_i)A^*\big) \p, \s\rangle|^\alpha \mu(d\s)\Big).
\end{eqnarray*}
Thus, by the dominated  convergence theorem, one gets from \eqref{INTEGRATED_OU} that the characteristic function or Fourier transform of $\Lambda_t^\x $ writes for all $\p\in \R^N $: 
\begin{eqnarray}
\label{TF}
 \varphi_{\Lambda_t^{\x}}(\p)&:=&\E(e^{i \langle \p , \Lambda_t^{\x} \rangle})
=  \exp\left(i\langle \p,\exp\big( tA\big)\x\rangle -\int_0^t \int_{\S^{d-1}} |\langle \exp(u A^*) \p,  B \sigma \s\rangle|^\alpha \mu(d\s)du\right)\nonumber\\ 
&&=\exp\left(i\langle \p,\exp\big( tA\big)\x\rangle -t \int_0^1 \int_{\S^{d-1}} |\langle   \p, \exp(vt A)B \sigma \s\rangle|^\alpha \mu(d\s)dv\right)\nonumber\\
&&=\exp\left(i\langle \p,\exp\big( tA\big)\x\rangle -t \int_0^1 \int_{\S^{d-1}} |\langle   \M_t\p, \exp(v A) \M_t^{-1} B \sigma \s\rangle|^\alpha \mu(d\s)dv\right)
\\ \nonumber
&&=\exp\left(i\langle \p,\exp\big( tA\big)\x\rangle -t \int_0^1 \int_{\S^{d-1}} |\langle   \M_t\p, \exp(v A)  B \sigma \s\rangle|^\alpha \mu(d\s)dv\right),
\end{eqnarray}
where we have used   \eqref{str} and the fact that 
 $\M_t^{-1} B \y = B \y$,  for any $\y \in \R^{d}$.
\def\stima4 {
where the last equality follows from the structure of $\exp(tA)$. Indeed, write $[\exp(tA)]_{.,1}=\left(\begin{matrix}
[\exp(tA)]_{1,1}\\
[\exp(tA)]_{2,1}\\
\vdots\\
[\exp(tA)]_{n,1}
\end{matrix} \right) $, for the submatrix of $\exp(tA) $ in $\R^N\otimes \R^d $ corresponding to its entries $\big( [\exp(At)]_{j,1}\big)_{j\in \leftB 1,n\rightB} $.
Note that
$$
[\exp(tA)]_{j,1} \in \R^{d_{j}} \otimes \R^{d}.
$$
We have that each entry $\big( [\exp(tA)]_{j,1}\big)_{j\in \leftB 1,n \rightB} $ writes:
\begin{equation}
\label{FORM_COLUMN}
 [\exp(tA)]_{j,1}=  K_{j,1}(A) \, \frac{t^{j-1}}{(j-1)!}.
 \end{equation}
where  $K_{j,1}(A) $ is a matrix in $\R^{d_{j}} \otimes \R^{d}$ which has coefficients depending on the ones of $A$ but it is independent on $t$ (note that $K_{1,1}(A) = I_{d\times d}$). The previous  identity  gives that for all $\s\in {\mathbb S}^{d-1}$ and all $t>0,$ $v \ge 0$,
\begin{equation} \label{importa}
\M_{t}^{-1}\exp(vt A)B \sigma \s=\exp(v A)B \sigma \s .
\end{equation}
}
Introduce now the function
$$
\begin{array}{cccc}
f: &[0,1] \times \S^{d-1} &\longrightarrow &\S^{N-1}\\
&(v,\s) &\longmapsto &\frac{\exp(vA) B\sigma \s}{|\exp(vA) B\sigma \s|},
\end{array}
$$
and on $[0,1]\times \S^{d-1} $ the measure:
$
m_{\alpha}(dv,d\s) =   |\exp(vA) B\sigma \s|^\alpha  \mu(d\s)dv.
$
 The Fourier transform in  \eqref{TF} thus rewrites:
\begin{equation*}
\varphi_{\Lambda_t^{\x}}(\p)=\exp\left(i\langle \p,\exp\big( t A\big)\x\rangle -t  \int_{\S^{N-1}} |\langle   \M_t\p, \bxi \rangle|^\alpha \bar \mu(d\bxi)\right),
\end{equation*}
where $\bar \mu $ is the image of $m_\alpha $ by $f$. Introduce now the symmetrized version of $\bar \mu$, defining for all $A\in {\mathcal B}(\S^{N-1})$ (Borel $\sigma $-field of $\S^{N-1}$):
$\mu_S(A)=\frac{\bar \mu(A)+\bar \mu(-A)}{2}.$
We get that
\begin{equation}
\label{TF_MULTI_SCALE}
\varphi_{\Lambda_t^{\x}}(\p)=\exp\left(i\langle \p,\exp\big( tA\big)\x\rangle -t  \int_{\S^{N-1}} |\langle   \M_t\p, \bxi \rangle|^\alpha \mu_S(d\bxi)\right),
\end{equation}
which indeed involves the exponent of a symmetric stable process $(S_t)_{t\ge 0}$ with spectral measure $\mu_S $ at point $\M_t\p$. Up to now we have just proved \eqref{d33}. 
On the other hand, it follows from \eqref{TF}  and \A{ND} that there exists $c:=c(\A{A})>0$ s.t.:
\begin{equation}
\label{CTR_ND}
\int_0^1 \int_{\S^{d-1}} |\langle   \M_t\p, \exp(v A)B \sigma \s\rangle|^\alpha \mu(d\s)dv\ge c|\M_t\p|^\alpha.
\end{equation}
The above result is algebraic. We refer to Lemma \ref{bound} or to \cite{huan:meno:15} 
for a complete proof. 
Thus, the mapping $\p\in \R^N\mapsto \varphi_{\Lambda_t^\x}(\p)$ is in $L^1(\R^N) $ so that \eqref{THE_DENS_OU} follows by inversion and a direct change of variable. The smoothness of $p_\Lambda $ readily follows from \eqref{THE_DENS_OU} and \eqref{CTR_ND}.
It is then well known (see e.g.  \cite{dynk:65}), and it can as well be easily derived by direct computations,  
 that $p_\Lambda$ is a fundamental solution of $\partial_t - {\mathcal L}_\sigma$ (note that $(\partial_t - {\mathcal L}_\sigma) p(\cdot, \cdot, \y)=0 $  on $(0, + \infty) \times \R^N$ for all $\y\in \R^N $ and $p(t,\x,\cdot) \longrightarrow \delta_\x(\cdot) $ as $t \to 0^+$). 

 Equation \eqref{REP} can be easily obtained by using  the Fourier transform  taking into account that the symbol of ${\mathcal L}_{\sigma}$ is $\Psi (\lambda)$ (cf. Section 3.3.2 in \cite{appl:09}). It can also be derived by 
It\^o's formula applied to $\big(u(r +t ,\Lambda_r^{\x}) \big)_{r\ge 0} $ (cf. Section 4.4 in \cite{appl:09}). We have
$$
\E [u(r +t ,\Lambda_r^{\x})] = u(t,x) + \E \int_0^r (\partial_t + {\mathcal L}_\sigma) u(s +t ,\Lambda_{s}^{\x}) ds. 
$$
Letting $r \to +\infty$ and changing variable we get 
 \eqref{REP}. 
\end{proof}


Let $T>0$ be fixed and recall that ${\mathscr L}_\sigma =  \partial_t + {\mathcal L}_\sigma 
$. Theorem \ref{THE_THM1} will actually be a consequence of the following estimates on the strip $$\Sc:=[-T,T]\times \R^N .$$
\begin{theorem}[Parabolic $L^p$ estimates on the strip $\Sc$]\label{THM_STRISCIA}
For $p\in (1,+\infty)$, there exists $C_p:=C_p(\A{A})\ge 1$ independent of $T>0$
s.t. for all $u\in C_0^\infty((-T, T) \times \R^N$), $i \in \leftB 1,n \rightB ,$ we have:
\begin{equation}
\label{STIMA_STRISCIA}
\|\Delta_{\x_i}^{\frac{\alpha_i}{2}}  u\|_{L^p(\Sc)}\le C_p \|{\mathscr L }_\sigma u\|_{L^p(\Sc)}.
\end{equation}
\end{theorem}
\begin{proof}[Proof of Theorem \ref{THE_THM1}]
To show that Theorem \ref{THM_STRISCIA}  implies  Theorem \ref{THE_THM1}
we introduce 
$
\psi\in C_{0}^{\infty}\left(  \mathbb{R}\right)
$
 with support in $(-1,1)$ and such that $\psi (s)=1$ for  $s \in [-1/2, 1/2]$.
 If $v \in C_{0}^{\infty}(\R^N)$ we consider 
\begin{equation} \label{ci1}
u\left( t,\x \right)  :=v\big(  \x\big)  \psi\big(  \frac{t}{T}\big),\;\; (t, \x) \in (-T, T) \times \R^N.
\end{equation} 
Since 
$
{\mathscr L }_\sigma  u\left( t,\x \right) $ $ =\psi (t/T) {\mathcal L}_{\sigma}v(\x) + \frac{1}{T}v(\x) \psi'(t/T)$ we find from Theorem \ref{THM_STRISCIA} applied to $u$:
\begin{gather*}
\|\Delta_{\x_i}^{\frac{\alpha_i}{2}}  u\|_{L^p(\Sc)}^p
= \int_{-T}^{T} |\psi(t/T)|^p dt \int_{\R^N} |\Delta_{\x_i}^{\frac{\alpha_i}{2}} v(\x)|^p d \x = 
 T \int_{-1}^{1} |\psi(s)|^p ds \int_{\R^N} |\Delta_{\x_i}^{\frac{\alpha_i}{2}} v(\x)|^p d \x
\\
\le C_p^p \int_{-T}^{T} \int_{\R^N} |\psi (t/T) {\mathcal L}_{\sigma}v(\x) + \frac{1}{T}v(\x) \psi'(t/T)|^p \,  d\x dt 
= C_p^p \, T \int_{-1}^{1} \int_{\R^N} |\psi (s) {\mathcal L}_{\sigma}v(\x) + \frac{1}{T}v(\x) \psi'(s)|^p \,   d\x ds. 
\end{gather*}
It follows that 
$$
\int_{-1}^{1} |\psi(s)|^p ds \int_{\R^N} |\Delta_{\x_i}^{\frac{\alpha_i}{2}} v(\x)|^p d \x \le C_p^p \, \int_{-1}^{1} \int_{\R^N} \big|\psi (s) {\mathcal L}_{\sigma}v(\x) + \frac{1}{T}v(\x) \psi'(s) \big|^p  d\x ds; 
$$
passing to the limit as $T \to \infty$ by the Lebesgue convergence theorem we get the assertion since ${\int_{-1}^{1} |\psi(s)|^p ds >0}$. 
\end{proof}

\begin{remark} \label{chen} The fact that in the previous theorem the constant $C_p$ is independent on $T$ agrees with   the   singular integral estimates given in Section 3 of 
\cite{bram:cerr:manf:96} for $\mathcal{L}_{\sigma}$ when $\alpha =2$.
 On the other hand, the parabolic estimates  of Theorem 3 in \cite{bram:cupi:lanc:prio:10} when considered on   the strip $\Sc:=[-T,T]\times \R^N $ have a constant depending on $T$ since  a more general matrix $A$ is considered in that paper (indeed the exponential matrix $e^{tA}$ can growth exponentially in \cite{bram:cupi:lanc:prio:10}). This is why the elliptic estimate in Theorem 1 of \cite{bram:cupi:lanc:prio:10}  (see \eqref{cald}) contains an extra term $\| v\|_{L^p(\R^N)}$ which is not present in the right-hand side of our estimates \eqref{THE_CTR_LP_ELLIP}.

\end{remark}

\mysection{Strategy of the proof of Theorem \ref{THM_STRISCIA}}\label{STRAT_PROOF}

To prove Theorem \ref{THM_STRISCIA}, thanks to the  formula \eqref{REP}, we restrict to  consider
functions of the form
\begin{equation}
\label{DEF_GK}
u(t,\x)=Gf(t,\x)=\int_t^T ds\int_{\R^N} p_\Lambda(s-t,\x,\y) f(s,\y)d\y,\ (t,\x)\in \Sc.
\end{equation}
 We study \eqref{DEF_GK} when  
 $f$ belongs to the space of test functions
 \begin{eqnarray}
\label{DEF_TEST_CLASS}
\nonumber {\mathscr T}
(\R^{1+N}):=\Big\{f \in C^{\infty} (\R^{1+ N}): \exists R>0, \forall t \not \in [-R,R],  f(t, \cdot ) =0,\ \    
\forall  t \in [-R,R],\ 
\\
 \text{$f$ and all spatial derivatives } \; \;  
 \partial_{\x}^{{\mathbf i}} f\in \bigcap_{p \in [1, +\infty]} L^p (\R^{1+N}),\ \text{for all multi-indices}\  {{\mathbf i}}\Big\}.
\end{eqnarray}
Note that,  for any $u \in C_0^{\infty} 
(\R^{1+ N}) $,    
\begin{equation}
\label{d11}
 (\partial_t + {\mathcal L}_\sigma)u \in  {\mathscr T}
(\R^{1+ N})
\end{equation}
(cf. Remark \ref{reg1}). 
\nero{
With the notations preceding \eqref{DEF_ALPHA_I_B_I1}, we write for all $i\in \leftB 1,n \rightB $,}  $  { B_i=\left(\begin{array}{c}
0_{d_1\times d_i}\\
\vdots\\
I_{d_i\times d_i}\\
\vdots\\
0_{d_n\times d_i}
\end{array} \right)}$.
Recalling the definition of $\M_t$ in \eqref{DEF_T_ALPHA}, we then get from  \eqref{CORRESP_L_S}, for all $i\in \leftB 1,n\rightB $ and $\alpha_i=\frac{\alpha}{1+\alpha (i-1)} $:
\nero{
\begin{eqnarray*} 
\Delta_{\x_i}^{\frac{\alpha_i}{2}}p_\Lambda (s-t,\x,\y)=\frac 1{\det(\M_{s-t})}\int_{\R^{d_i}} \Big(p_S\big(s-t,\M_{s-t}^{-1}(e^{(s-t)A}(\x+B_iz)-\y)\big)-
 p_S\big(s-t,\M_{s-t}^{-1}(e^{(s-t)A}\x-\y)\big) 
\\
-\nabla_{\x_i}\Big(p_S\big(s-t,\M_{s-t}^{-1}(e^{(s-t)A}\x-\y)\big)\Big)\cdot z\I_{|z|\le 1}
\Big) \frac{dz}{|z|^{d_i+\alpha_i}}\\
=\frac 1{\det(\M_{s-t})}\int_{\R^{d_i}} \Big(p_S\big(s-t,\M_{s-t}^{-1}(e^{(s-t)A}\x-\y)+(s-t)^{-(i-1)} \, e^A B_i  z\big)
-p_S\big(s-t,\M_{s-t}^{-1}(e^{(s-t)A}\x-\y)\big)\\
-\nabla p_S\big(s-t,\M_{s-t}^{-1}(e^{(s-t)A}\x-\y)\big)\cdot (\M_{s-t}^{-1}e^{(s-t)A}) B_i z\  \I_{|z|\le 1}\big)
 \Big) \frac{dz}{|z|^{d_i+\alpha_i}}, 
\end{eqnarray*}
$\x, \, \y \in \R^N, \, s>t$,
 where we have used that $\M_{s-t}^{-1}e^{(s-t)A} B_i z 
  = e^{A} \M_{s-t}^{-1} B_i z $ (see \eqref{str}) and the fact that  
\begin{equation} \label{na1}
 \M_{s-t}^{-1} B_i z = (s-t)^{-(i-1)} B_i z,\;\;\;, r>0, \; z \in \R^d. 
\end{equation}
In the sequel we set 
$$
 (e^A)_i = e^A B_i 
$$
Introduce now for  $\varphi\in C_0^{\infty}(\R^N) $ and for all $i\in \leftB 1,n\rightB,\ s>t $ the operator: 
\begin{equation}
\label{DEF_DELTA_MOD}
\Delta^{\frac{\alpha_i}2,A,i,s-t}\varphi(\x):=\int_{\R^{d_i} } \Big (\varphi(\x+(s-t)^{-(i-1)}(e^A)_iz)-\varphi(\x)- (s-t)^{-(i-1)}\nabla \varphi(\x)\cdot (e^A)_iz \I_{|z|\le 1} \Big) \frac{dz}{|z|^{d_i+\alpha_i}}.
\end{equation}
We insist here on the fact that, in \eqref{DEF_DELTA_MOD}, $\nabla $ stands for the full gradient on $\R^N$.
We thus get the correspondence: 
\begin{equation}
\label{CORRISPONDENZA}
\begin{split}
p_\Lambda (s-t,\x,\y)&=\frac{1}{\det(\M_{s-t})}p_S(s-t, \cdot )) \big(\M_{s-t}^{-1}(e^{(s-t)A}\x-\y)\big),\\
\Delta_{\x_i}^{\frac{\alpha_i}{2}}p_\Lambda (s-t,\x,\y)&=\frac{1}{\det(\M_{s-t})} (\Delta^{\frac{\alpha_i}2, A,i,s-t}p_S(s-t, \cdot )) \big(\M_{s-t}^{-1}(e^{(s-t)A}\x-\y)\big).
\end{split}
\end{equation}
\nero{Equation \eqref{CORRISPONDENZA} reflects in particular how the effect of a stable operator  on  the 
$i^{\rm th}$-variable propagates to the next variables. This correspondence allows to develop  specific computations related to the  stable process $(S_t)$ having a  singular spectral measure.
}
 Formula \eqref{CORRISPONDENZA} 
 is meaningful since, for any $s>t$, $p_S (s-t, \cdot) \in C^{\infty}(\R^N) \cap L^1(\R^N)$ (this is a consequence of \eqref{CTR_ND}). Moreover all the partial derivatives of $p_S (s-t, \cdot) \in  L^1(\R^N)$. Arguing as in Remark \ref{reg1} one can prove that $\Delta^{\frac{\alpha_i}2, A,i,s-t}p_S(s-t, \cdot ) \in L^1(\R^N) \cap C^{\infty}(\R^N)$.
 This 
 gives in particular that for $f\in {\mathscr T}(\R^{1+N})$ the function $\Delta_{\x_i}^{\frac {\alpha_i} 2}Gf(t,\x) $ is pointwise  defined for all  $(t,\x)\in \Sc $. 
}
 
 Additionally, by using It\^o's formula (cf. the proof of  Proposition \ref{THE_PROP_FUND}) or Fourier analysis techniques  
 it can be   easily derived that for $f \in  {\mathscr T}
(\R^{1+ N}) $,  $Gf$ solves the equation:
\begin{equation}\label{PbCauchy}
\begin{cases}
(\partial_t + {\mathcal L}_\sigma) w(t,\x)  = -f(t,\x),\ (t,\x)\in \overset{\circ}{\Sc},\\
w(T,\x)=0.
\end{cases}
\end{equation}
  By \eqref{REP}  we know that $u = - G ({\mathscr L}_\sigma u)$ if $u \in  C_0^\infty((-T, T) \times \R^N)$. Hence 
estimates \eqref{STIMA_STRISCIA} 
follows if we prove 
\begin{equation}
\label{STIMA_STRISCIA_GREEN}
\sum_{i=1}^{n}\|\Delta_{\x_i}^{\frac {\alpha_i}{2}} Gf\|_{L^p(\Sc)}\le C_p \|f\|_{L^p(\Sc)},\;\;\; f \in   {\mathscr T}
(\R^{1+ N}).
\end{equation}
 To prove estimate \eqref{STIMA_STRISCIA_GREEN},
we introduce for $i\in \leftB 1,n \rightB $, the kernel $k_i\big( (t,\x),(s,\y)\big):=\Delta_{\x_i}^{\frac{\alpha_i}{2}} p_\Lambda(s-t,\x,\y)$ so that:
\begin{equation}
\label{DEF_K}
K_if(t,x):=\Delta_{\x_i}^{\frac{\alpha_i}{2}} Gf(t,x)=\int_{t}^T ds \int_{\R^N}
k_i\big( (t,\x),(s,\y)\big) f(s,\y)d\y .
\end{equation}
The goal is to show $\sum_{i=1}^n\|K_if \|_{L^p(\Sc)} \le C_p \|f\|_{L^p(\Sc)}$.
But the kernels $(k_i)_{i\in \leftB 1,n\rightB}$ are singular (see e.g. Lemma \ref{SENS_SING_STAB} below) and do not satisfy some \textit{a priori} integrability conditions.
We will  establish \eqref{STIMA_STRISCIA_GREEN} through uniform estimates on a \textit{truncated} kernel in time and space.
We need to introduce the following \textit{quasi-distance} on $\Sc$.
\begin{definition}[Quasi-distance on $\Sc$] \label{dis11}
For $\big((t,\x), (s,\y)\big)\in {\Sc}^2 $, we  introduce the \textit{quasi-distance}:
\begin{equation}
\label{DEF_QD}
d\big((t,\x),(s,\y)\big ):=\frac 12\Big( \rho
(t-s,\exp((s-t)A)\x-\y)+\rho
(t-s,\x-\exp((t-s)A)\y)\Big),
\end{equation}
where the homogeneous pseudo-norm $\rho$ has been defined in  \eqref{rho_homo}.
\end{definition}

The above definition takes into account the transport of the first spatial point $\x$ by the matrix $e^{(s-t)A} $ (\textit{forward transport} of the initial point) and the one of the second spatial point $\y$ by $e^{(t-s)A} $  (\textit{backward transport} of the final point). By \textit{transport} we actually intend here the action of the first order term in \eqref{OP_SPAZIALE} (corresponding to the deterministic differential system $\dot\btheta_u(\z)=A\btheta_u(\z), \btheta_0(\z)=\z\in \R^N  $)  w.r.t. the considered associated times and points.  Observe that the fact that $d$ is actually a quasi-distance, i.e. that it satisfies the quasi-triangle inequality, is not obvious \textit{a priori}.
From the invariance by dilation of the operator ${\mathscr L}_\sigma $ established in Proposition \ref{THE_PROP_FUND} and the underlying group structure on $\R^{1+N}$ endowed with the composition law
$(t,\x)\circ (\tau,\bxi)=(t+\tau,\bxi+e^{A\tau }\x),$
the quasi-triangle inequality can be derived  similarly to    
 \cite{difr:poli:06} in the diffusive setting, see also \cite{bram:cupi:lanc:prio:10}.
 In order to be self-contained, and to show that the quasi-distance property still holds without any underlying group structure,
 we anyhow provide a proof of this fact in Proposition \ref{PROP_DOUBLING} below.

Let us now introduce a \textit{non-singular} Green kernel. 
We choose to truncate w.r.t. time. 
For $\varepsilon>0$ and a given $c_0>0$, we define for all $ i\in \leftB 1,n\rightB$:
\begin{eqnarray}
k_{i,\varepsilon}\big( (t,\x),(s,\y)\big)&:=&\I_{|s-t|\ge \varepsilon}\Delta_{\x_i}^{\frac{\alpha_i}{2}} p_\Lambda(s-t,\x,\y) \notag\\
&=& \I_{|s-t|\ge \varepsilon} \I_{ d ( (t,\x) , (s,\y) ) \le c_0 } \Delta_{\x_i}^{\frac{\alpha_i}{2}} p_\Lambda(s-t,\x,\y)
+\I_{|s-t|\ge \varepsilon} \I_{ d ( (t,\x) , (s,\y) ) > c_0 } \Delta_{\x_i}^{\frac {\alpha_i}{2}} p_\Lambda(s-t,\x,\y)  \notag\\
&=:& k_{i,\varepsilon}^C  \big( (t,\x),(s,\y)\big)+ k_{i,\varepsilon}^F \big( (t,\x),(s,\y)\big). \label{DEF_NUCLEO_SINGOLARE_TRONCATO}
\end{eqnarray}
We then define
\begin{equation}
\label{DEF_K_EPSILON}
K_{i,\varepsilon} f(t,x):=\int_{t}^T ds \int_{\R^N}
k_{i,\varepsilon} \big( (t,\x),(s,\y)\big) f(s,\y)d\y,\;\; f \in   {\mathscr T}
(\R^{1+ N}). 
\end{equation}
Accordingly, the operators $K_{i,\varepsilon}^C f(t,x)$ and $K_{i,\varepsilon}^F f(t,x)$ are obtained replacing $k_{i,\varepsilon}$ by $k_{i,\varepsilon}^C $ and $k_{i,\varepsilon}^F $ in \eqref{DEF_K_EPSILON}.
Observe that $K_{i,\varepsilon}^C$ (resp.  $K_{i,\varepsilon}^F$) is concerned with the integrating points that are \textit{close} (resp. \textit{far}) from the initial point with respect to the underlying quasi-distance in \eqref{DEF_QD}.
Similarly, we denote by $G_\varepsilon f(t,\x) $ the quantity obtained replacing $t$ by $t+\varepsilon$ in \eqref{DEF_GK}.
Since
$$\|\Delta_{\x_i}^{\frac{\alpha_i}{2}} G_\varepsilon f\|_{L^p(\Sc)}=\|K_{i,\varepsilon} f\|_{L^p(\Sc)} \le  \| K_{i,\varepsilon}^C f \|_{L^p(\Sc)} + \| K_{i,\varepsilon}^F f \|_{L^p(\Sc)},$$
the result \eqref{STIMA_STRISCIA_GREEN} will follow from weak convergence arguments, provided that the following lemma holds.
\begin{lemma}[Key Lemma]\label{KEY_LEMMA}
There exists a constant $C_{p}>0$ independent of $\varepsilon>0$ and $T>0$ such that for all $f\in  
{\mathscr T}(\R^{1+N})$:
$$\| K_{i,\varepsilon}^C f \|_{L^p(\Sc)} + \| K_{i,\varepsilon}^F f \|_{L^p(\Sc)} \le C_{p} \| f\|_{L^p(\Sc)}.$$
\end{lemma}
We prove this estimate separately for $K_{i,\varepsilon}^C f$ and $K_{i,\varepsilon}^F f$ in Section \ref{KEY_SECTION} below.
For the latter, a direct argument can be used whereas
for $K_{i,\varepsilon}^C f$, some controls from singular integral theory are required. Precisely, we aim to use Theorem 2.4, Chapter III in  \cite{coif:weis:71} (see Theorem \ref{coi}). The choice to split the kernel $ k_{i,\varepsilon}$ into $k_{i,\varepsilon}^C+k_{i,\varepsilon}^F $ is then motivated by the fact that even though for all $\varepsilon>0 $, $k_\varepsilon\in L^1(\Sc) $ when integrating w.r.t. $dtd\x$ or $dsd\y $ we do not have that $k_{i,\varepsilon}\in L^2(\Sc^2)$  
which is the assumption required in the quoted reference. This is what induces us to introduce a spatial truncation to get the joint integrability in all the variables.

 \mysection{Technical Lemmas}
\label{THE_TEC_SEC}

We now give two \textit{global} results that will serve several times for the truncated kernel as well. These lemmas are the current technical core of the paper. Lemma \ref{LEMME_L2} is a global $L^2$ bound on the singular kernel $\Delta_{\x_i}^{\frac {\alpha_i} {2}} G_\varepsilon f $ which is based on Fourier arguments and the representation formula in \eqref{THE_DENS_OU}.
Lemma \ref{LEMME_DEV} is a key tool to control singular kernels (see 
 Appendix \ref{APP_CW} in the   framework of homogeneous spaces). 

\begin{lemma}[Global $L^2$ estimate]
\label{LEMME_L2}
 There exists a positive constant $C_2:=C_2(\A{A})$ such that, for all $\varepsilon>0 $,  $i\in \leftB 1,n\rightB $, and for all $f \in L^2( \R^{1+N})$,
$$
\|\Delta_{\x_i}^{\frac{\alpha_i}{2}} G_\varepsilon f\|_{L^2(\Sc)}\le C_2 \|f\|_{L^2(\Sc)}.
$$
\end{lemma}
The constant $C_2$ does not depend on $T>0$ considered in the strip $\Sc$. Also, this estimate  would hold under weaker assumptions than \A{ND}. No symmetry would a priori be needed, a control similar to \eqref{EQ_ND} for the real-part should be enough.
It would also hold for a wider class of initial operators $L_\sigma $ including those associated with tempered or truncated stable processes. The result seems to be new even in the limit local case $\alpha=2.$

\begin{lemma}[Deviation Controls]\label{LEMME_DEV}
There exist constants $K:=K(\A{A}),\ C:=C(\A{A})\ge 1$ s.t. for all  
$(\sigma,\bxi), (t,\x) \in \Sc$  the following control holds, for all $i\in \leftB 1,n\rightB $:
\begin{equation}
\label{CONTROL_DEV}
\int_{ s\ge t\vee \sigma , \, \rho\ge K \gamma }|\Delta_{\x_i}^{\frac{\alpha_i}{2}}p_\Lambda (s-t,\x,\y)-\Delta_{\x_i}^{\frac{\alpha_i}{2}}p_\Lambda (s-\sigma,\bxi,\y) | d\y ds \le C,
\end{equation}
where we have denoted $\rho:=\rho(s-t,e^{(s-t)A}\x-\y), \gamma:=\rho(\sigma-t,e^{(\sigma-t)A}\x-\bxi) $.
\end{lemma}

\subsection{Proof of Lemma \ref{LEMME_L2}}
We start from the representation of the density obtained in Proposition \ref{THE_PROP_FUND}:
\begin{eqnarray*}
 p_\Lambda(t,\x,\y)&=&
 \frac{\det(\M_t^{-1})}{(2\pi)^N}\int_{\R^{N}} \exp\Big(-i\langle \M_t^{-1}(\y-e^{t A}\x),\p\rangle \Big)\exp\Big(-t\int_{\S^{N-1}}|\langle \p, \bxi\rangle|^\alpha \mu_S(d\bxi)\Big)d\p\\
 &=&
 \frac{1}{(2\pi)^N} \int_{\R^N} \exp\Big(-i \langle \q , \x - e^{-tA}\y \rangle \Big)
 \exp\Big(-t\int_{\S^{N-1}}|\langle \M_t e^{-tA^*} \q, \bxi\rangle|^\alpha \mu_S(d\bxi)\Big) d\q,
\end{eqnarray*}
recalling that the specific form of $A$ yields $\det(e^{t A^*})=\det(e^{t A})=1$.
 Let us fix $j\in \leftB 1,n\rightB $ and prove the estimate for $\Delta_{\x_j}^{\frac {\alpha_j}2} G_\varepsilon f $. 
We can compute, for $f\in{\mathscr S}(\R^{1+N})$ (Schwartz class of $\R^{1+N}$), the Fourier transform:
 $$\bzeta\in \R^N\mapsto  \mathcal{F}(\Delta_{\x_j}^{\frac{\alpha_j}{2}}G_\varepsilon f) (t,\bzeta) = \int_{\R^{N}}  e^{i \langle \bzeta, \x \rangle} \Delta_{\x_j}^{\frac{\alpha_j}{2}}G_\varepsilon f(t,\x) d \x,\ t\in [-T,T].$$
 Indeed,  from the comments after \eqref{CORRISPONDENZA} and by the Young inequality  we have  that  $\Delta_{\x_j}^{\frac{\alpha_j}{2}}G_\varepsilon f(t, \cdot)\in L^1 (\R^N)\cap L^2(\R^{N}) $, $t \in [-T,T]$.
  For all \nero{ $(t,\bzeta)\in \Sc $ with $t+ \epsilon \le T$}  we get from the definition of $G_\varepsilon f $ following \eqref{DEF_K_EPSILON}:

\def\stima3{  
\begin{eqnarray*}
\mathcal{F}(\Delta_{\x_j}^{\frac{\alpha_j}{2}}G_\varepsilon f)  (t,\bzeta)
=
\int_{\R^{N}}  e^{i \langle \bzeta,\x \rangle} \Bigg(\int_{t+\varepsilon}^T 
\int_{\R^{N}} 
f(s,\y)\\
\times \frac{1}{(2\pi)^N} \int_{\R^N} e^{-i \langle \q , \x \rangle} \,
e^{i \langle \q ,  e^{-(s-t)A}\y \rangle } |\q_j|^{\alpha_j}
 \exp\Big(-(s-t)\int_{\S^{N-1}}|\langle  \M_{s-t} e^{-(s-t)A^*} \q, \bxi\rangle|^\alpha \mu_S(d\bxi)\Big) d\q d\y ds\Bigg) d\x.
%
\end{eqnarray*}
Using the Fubini theorem in the internal integral with respect to $(\q, \y,s)$ we obtain
\begin{align*}
\mathcal{F}(\Delta_{\x_j}^{\frac{\alpha_j}{2}}G_\varepsilon f)  (t,\bzeta)
=  \frac{1}{(2\pi)^N} 
\int_{\R^{N}}  e^{i \langle \bzeta,\x \rangle} \Bigg(
\int_{\R^{N}} e^{-i \langle \q , \x \rangle} 
 \\  
 \times \,
|\q_j|^{\alpha_j}  \exp\big(-(s-t)\int_{\S^{N-1}}|\langle  \M_{s-t} e^{-(s-t)A^*} \q, \bxi\rangle|^\alpha \mu_S(d\bxi) \big) 
\, \Big( \int_{t+\varepsilon}^T \int_{\R^N} e^{i \langle \q ,  e^{-(s-t)A}\y \rangle } f(s,\y) 
 ds d\y \Big ) d \q \Bigg) d \x 
\end{align*}
The Fourier inversion formula then yields:
\begin{eqnarray*}
&&\mathcal{F}(\Delta_{\x_j}^{\frac{\alpha_j}{2}}G_\varepsilon f)  (t,\bzeta) \\
&=&
 \int_{t+\varepsilon}^T  \int_{\R^{N}}  e^{i\langle \bzeta ,e^{-(s-t)A}\y\rangle} f(s,\y)  |\bzeta_j|^{\alpha_j}
\exp\left(- (s-t)\int_{\S^{N-1}} |\langle \M_{s-t} e^{-(s-t)A^*} \bzeta,\bxi \rangle|^{\alpha}\mu_S(d\bxi) \right)  d\y ds\\
&=& |\bzeta_j|^{\alpha_j}  \int_{t+\varepsilon}^T  \mathcal{F}(f)(s,e^{-(s-t)A^*} \bzeta)\exp\left(- (s-t)\int_{\S^{N-1}} |\langle \M_{s-t} e^{-(s-t)A^*} \bzeta,\bxi \rangle|^{\alpha}\mu_S(d\bxi) \right) ds
\end{eqnarray*}
($\mathcal{F}(f)$ denotes the Fourier transform of $f(s, \cdot)$).  
}  
\begin{eqnarray*}
\mathcal{F}(\Delta_{\x_j}^{\frac{\alpha_j}{2}}G_\varepsilon f)  (t,\bzeta)
=|\bzeta_j|^{\alpha_j}\mathcal{F}(G_\varepsilon f)(t,\bzeta)=|\bzeta_j|^{\alpha_j} \int_{\R^N}\exp(-i\langle \bzeta, \x\rangle)\Big(\int_{t+\varepsilon}^{T}\int_{\R^N}
p_\Lambda(s-t,\x,\y)f(s,\y)d\y ds \Big)d\x\\
=|\bzeta_j|^{\alpha_j} \int_{\R^N}\exp(-i\langle \bzeta, \x\rangle)\Big(\int_{t+\varepsilon}^{T}\int_{\R^N} 
\frac{1}{\det(\M_{s-t})}p_S(s-t,\M_{s-t}^{-1}(e^{(s-t)A}\x-\y))f(s,\y)d\y ds \Big) d\x,
\end{eqnarray*}
using equation \eqref{CORRESP_L_S} for the last identity. Rewrite now
\begin{eqnarray*}
\mathcal{F}(\Delta_{\x_j}^{\frac{\alpha_j}{2}}G_\varepsilon f)  (t,\bzeta)
=|\bzeta_j|^{\alpha_j} \int_{\R^N}\exp(-i\langle \M_{s-t}e^{-(s-t)A^*}\bzeta, \M_{s-t}^{-1}e^{(s-t)A}\x\rangle)\\
\Big(\int_{t+\varepsilon}^{T}\int_{\R^N} 
\frac{1}{\det(\M_{s-t})}p_S(s-t,\M_{s-t}^{-1}(e^{(s-t)A}\x-\y))f(s,\y)d\y ds \Big) d\x.
\end{eqnarray*}
Using the Fubini theorem yields: 
\begin{eqnarray*}
\mathcal{F}(\Delta_{\x_j}^{\frac{\alpha_j}{2}}G_\varepsilon f)  (t,\bzeta)
=|\bzeta_j|^{\alpha_j} \int_{t+\varepsilon}^{T} \int_{\R^N}\exp(-i\langle e^{-(s-t)A^*}\bzeta, \y) f(s,\y) \\
\Big( \int_{\R^N}\exp(-i\langle \M_{s-t}e^{-(s-t)A^*}\bzeta, \M_{s-t}^{-1}(e^{(s-t)A}\x-\y)\rangle)
\frac{1}{\det(\M_{s-t})}p_S(s-t,\M_{s-t}^{-1}(e^{(s-t)A}\x-\y))d\x \Big) d\y ds\\
=|\bzeta_j|^{\alpha_j} \Big(\int_{t+\varepsilon}^{T} \int_{\R^N}\exp(-i\langle e^{-(s-t)A^*}\bzeta, \y) f(s,\y) 
\Big( \int_{\R^N}\exp(-i\langle \M_{s-t}e^{-(s-t)A^*}\bzeta, \tilde \x\rangle)
p_S(s-t,\tilde \x)d\tilde \x \Big) d\y ds,
\end{eqnarray*}
 setting $\tilde \x=\M_{s-t}^{-1}(e^{(s-t)A}\x-\y) $ and recalling that $\det(e^{(s-t)A})=1 $ for the last identity. We finally get
\begin{eqnarray*}
\mathcal{F}(\Delta_{\x_j}^{\frac{\alpha_j}{2}}G_\varepsilon f)  (t,\bzeta) 
= |\bzeta_j|^{\alpha_j}  \int_{t+\varepsilon}^T  \mathcal{F}(f)(s,e^{-(s-t)A^*} \bzeta)\mathcal{\F}(p_S)(s-t,\M_{s-t}e^{-(s-t)A^*}\bzeta)ds\\
= |\bzeta_j|^{\alpha_j}  \int_{t+\varepsilon}^T  \mathcal{F}(f)(s,e^{-(s-t)A^*} \bzeta)\exp\left(- (s-t)\int_{\S^{N-1}} |\langle \M_{s-t} e^{-(s-t)A^*} \bzeta,\bxi \rangle|^{\alpha}\mu_S(d\bxi) \right) ds,
\end{eqnarray*}
where ($\mathcal{F}(f)(s,\cdot), \mathcal{F}(p_S)(s-t,\cdot)$ denote the Fourier transforms of $f(s, \cdot),p_{S}(s-t,\cdot)$.

Let us now com\-pu\-te $||\mathcal{F}(\Delta_{\x_j}^{\frac{\alpha_j}{2}}G_\varepsilon f)||_{L^2(\Sc)}$. From the non degeneracy of $\mu_S$,  we have:
\begin{eqnarray}
\label{THE_REF_FOR_LATER}
|\mathcal{F}(\Delta_{\x_j}^{\frac{\alpha_j}{2}}G_\varepsilon f)  (t,\bzeta) | &\le& C |\bzeta_j|^{\alpha_j} \int_{t}^T  |\mathcal{F}(f)(s,e^{-(s-t)A^*} \bzeta)|\exp\left(-C^{-1}(s-t)| \M_{s-t}e^{-(s-t)A^*}\bzeta|^\alpha \right)ds.
\end{eqnarray}
 For the $L^2$ norm of $\mathcal{F} (\Delta_{\x_j}^{\frac{\alpha_j}{2}}G f)$, using 
 the Cauchy-Schwarz inequality, we obtain:
\begin{align*}
&&||\mathcal{F}(\Delta_{\x_j}^{\frac{\alpha_j}{2}}G_\varepsilon f)||_{L^2(\Sc)}^2
\\ 
&\le& C \int_{-T}^T dt  \int_{\R^{N}} \Bigg(   |\bzeta_j|^{\alpha_j} 
 \,\int_{t}^T  |\mathcal{F}(f)(s,e^{-(s-t)A^*} \bzeta) |^2 \exp\left(- C^{-1}(s-t)| \M_{s-t}e^{-(s-t)A^*}\bzeta|^\alpha \right)ds\Bigg)\\
 &&\times \Bigg(|\bzeta_j|^{\alpha_j}\int_t^T \exp\left(- C^{-1}(s-t)| \M_{s-t}e^{-(s-t)A^*}\bzeta|^\alpha \right)ds\Bigg)d\bzeta. 
 \end{align*}
Now formula \eqref{str} 
\nero{yields that:}
$e^{-(s-t)A^*}=\M_{s-t}^{-1}e^{-A^*} \M_{s-t},$
where $e^{-A^*} $ is non-degenerate. Thus, there exists \nero{ $c = c(A)>0$} s.t. for all $T\ge s\ge t\ge -T,\ \bzeta\in \R^N $:
\begin{eqnarray}
(s-t)| \M_{s-t}e^{-(s-t)A^*}\bzeta|^\alpha\ge c  (s-t)| \M_{s-t}\bzeta|^\alpha\ge c(s-t)^{1+\alpha(j-1)}|\bzeta_j|^\alpha , \;\; \text{and}\notag\\
\label{SEPARAZIONI_IN_S}
||\mathcal{F}(\Delta_{\x_j}^{\frac{\alpha_j}{2}}G_\varepsilon f)||_{L^2(\Sc)}^2\notag
\\ \le C \int_{-T}^T \! \! \! dt  \int_{\R^{N}} \! \! \Bigg(   |\bzeta_j|^{\alpha_j} 
 \! \int_{t}^T \!\! |\mathcal{F}(f)(s,e^{-(s-t)A^*} \bzeta) |^2 \exp\left(- C^{-1}(s-t)| \M_{s-t}e^{-(s-t)A^*}\bzeta|^\alpha \right)ds\Bigg)
\Theta_j(t,T,\bzeta_j) d\bzeta,\notag
 \\
 \Theta_j(t,T,\bzeta_j):=\Bigg(|\bzeta_j|^{\alpha_j}\! \!\! \int_t^T \exp\left(- C^{-1}(s-t)^{1+\alpha(j-1)}|\bzeta_j|^\alpha \right)ds\Bigg).
\end{eqnarray}
Let us prove that $ \Theta_j(t,T,\bzeta_j)$ is bounded. Write \nero{ for $ |\bzeta_j | \not =0 $, changing variable, recalling that $\alpha_j=\frac{\alpha}{1+\alpha(j-1)} $:
\begin{eqnarray*}
\Theta_j(t,T,\bzeta_j)\le |\bzeta_j|^{\alpha_j}\int_0^{\infty} \exp\left(- C^{-1} \, [u \, |\bzeta_j|^{\, \frac{\alpha}{ 1+\alpha(j-1)}}  ]^{1+\alpha(j-1)} \right)du 
\\ 
= |\bzeta_j|^{\alpha_j}  \frac{1}{|\bzeta_j|^{\alpha_j} }
 \int_0^{\infty} \exp\left(- C^{-1} \,  v^{1+\alpha(j-1)} \right)dv = C_0 < \infty.
\end{eqnarray*}
}
\def\stima 5{
Setting $v=u ^{1+\alpha(j-1)}|\bzeta_j|^\alpha$,  so that $du= \frac{1}{|\bzeta_j|^{\alpha_j} v^{\frac{(j-1)\alpha}{1+\alpha(j-1)} }
} \frac{dv}{1+(j-1)\alpha}$, gives:
\begin{eqnarray*}
\Theta_j(t,T,\bzeta_j)= \int_0^{2T^{1+(j-1)\alpha}|\bzeta_j|^\alpha} \exp\left(- C^{-1}v \right)\frac{dv}{(1+(j-1)\alpha) v^{\frac{(j-1)\alpha}{1+(j-1)\alpha}}	}\le C\int_0^{+\infty} \exp\left(- C^{-1}v \right)\frac{dv}{v^{\frac{(j-1)\alpha}{1+(j-1)\alpha}}	}\le C.
\end{eqnarray*}
}
We eventually get from \eqref{SEPARAZIONI_IN_S} by the Fubini theorem  
\begin{equation*}
\begin{split}
||\mathcal{F}(\Delta_{\x_j}^{\frac{\alpha_j}{2}}G_\varepsilon f)||_{L^2(\Sc)}^2\\
\le C \int_{-T}^T dt  \int_{\R^{N}} \Bigg(   |\bzeta_j|^{\alpha_j} 
 \,\int_{t}^T  |\mathcal{F}(f)(s,e^{-(s-t)A^*} \bzeta) |^2 \exp\left(- C^{-1}(s-t)| \M_{s-t}e^{-(s-t)A^*}\bzeta|^\alpha \right)ds\Bigg)d\bzeta
\\
= \int_{-T}^T ds \int_{-T}^{s} dt
 \int_{\R^{N}}    
  |\bzeta_j|^{\alpha_j} 
 \,  |\mathcal{F}(f)(s,e^{-(s-t)A^*} \bzeta) |^2 \exp\left(- C^{-1}(s-t)| \M_{s-t}e^{-(s-t)A^*}\bzeta|^\alpha \right) d\bzeta. 
 \end{split} 
\end{equation*}
Setting $\q=e^{-(s-t)A^*} \bzeta$ in the space integral now yields:
\begin{eqnarray*}
||\mathcal{F}(\Delta_{\x_j}^{\frac{\alpha_j}{2}}G_\varepsilon f)||_{L^2(\Sc)}^2
&\le& C
\int_{-T}^T ds \int_{-T}^{s} dt
 \int_{\R^{N}}    |(e^{(s-t)A^*}\q)_j|^{\alpha_j} 
 \,  |\mathcal{F}(f)(s,\q)|^2 \exp\left(-C^{-1}(s-t)|\M_{s-t} \q |^\alpha \right) d\q \\
 &\le & C \int_{-T}^T ds \int_{\R^N}  |\mathcal{F}(f)(s,\q)|^2 \, d\q \int_{-T}^s   |(e^{(s-t)A^*}\q)_j|^{\alpha_j} \exp\left(-C^{-1}(s-t)|  \M_{s-t} \q |^\alpha \right) \, dt
 \\
 &\le & C \int_{-T}^T ds \int_{\R^N}  |\mathcal{F}(f)(s,\q)|^2\bar \Theta_j(s,T,\q) d \q.
\end{eqnarray*}
\nero{ with  $ \bar \Theta_j(s,T,\q) = \int_{-T}^s   |(e^{(s-t)A^*}\q)_j|^{\alpha_j} \exp\left(-C^{-1}(s-t)|  \M_{s-t} \q |^\alpha \right) \, dt$. Now indicating with $\q^*$ the row vector 
we have (see \eqref{str})
\begin{align} \label{chi}  \nonumber
 |(e^{r A^*}\q)_j| =  |B_j^* e^{rA^*} [(\q^*)^*] |
 =  | (\q^* e^{rA} B_j)^*| = 
 | \q^* e^{rA} B_j| 
\\
 =  | \q^* \M_r e^{A} \M_r^{-1} B_j | = r^{-(j-1)} | \q^* \M_r e^{A}  B_j | =  r^{-(j-1)} | \,[\q_1,  r \q_2, \ldots , r^{n-1}\q_n]^* \, e^A B_j | 
\\  \nonumber
 \le c  r^{-(j-1)} (r^{j-1}|\q_j| + r^j|\q_{j+1}| + \ldots + r^{n-1}|\q_{n}|) = c \sum_{k =j}^n r^{k-j} |\q_k|, \;\; r >0, 
\end{align}
since $ B_j = \left(\begin{array}{c}
0_{d_1\times d_j}\\
\vdots\\
I_{d_j\times d_j}\\
\vdots\\
0_{d_n\times d_j}
\end{array} \right)$,  $e^{A}  B_j = (e^{A})_j  = \left(\begin{array}{c}0_{d_1\times d_j}\\
\vdots\\
0_{d_{j-1}\times d_j}\\
{K}_{j,j}(A)
\\
\vdots\\
K_{n,j}(A)
\end{array}\right)$  for some $K_{h,j}(A) \in \R^{d_h}\otimes \R^{d_j}, h =j, \ldots, n$. 
}
We get 
\begin{eqnarray*}
\bar \Theta_j(s,T,\q)&\le& C\sum_{k=j}^n \int_{-T}^s | (s-t)^{(k-j)}\q_k|^{\alpha_j}\exp(-C^{-1}(s-t)^{1+\alpha(k-1)}|\q_k|^\alpha)dt\\
&\le& C\sum_{k=j}^n \int_{0}^{\infty}  u^{(k-j) \alpha_j}|\q_k|^{\alpha_j}\exp(-C^{-1} \, [ u \cdot |\q_k|^{\alpha_k} \, ]^{1+\alpha(k-1)})du,
\end{eqnarray*}
for $|\q_k| \not =0$, recalling $\alpha_k = \frac{\alpha}{1 + \alpha(k-1) }$.
\nero{ Setting $u   \cdot |\q_k|^{\alpha_k} = v$ in each integral we find 
\begin{align*}
 \bar \Theta_j(s,T,\q)\le 
 C\sum_{k=j}^n   |\q_k|^{\alpha_j - \alpha_k - (k-j) \alpha_j \alpha_k }\ \int_{0}^{\infty}  v^{(k-j) \alpha_j} \exp(-C^{-1} \, [ v \, ]^{1+\alpha(k-1)})dv \le n C_1 < \infty,
\end{align*}
since $\alpha_j - \alpha_k - (k-j) \alpha_j \alpha_k =0$, for any $ j \le k \le n$. }
\def\stima8{
Setting for each $k\in \leftB j,n\rightB $, $v^{jk}:=u^{1+(k-1)\alpha}|\q_k|^\alpha $, so that $u^{k-j}du=\frac{dv^{jk}}{|\q_k|^{\alpha_j}}  \frac{1}{(v^{jk})^{\frac{(k-1)\alpha}{1+(k-1)\alpha}}}(v^{jk})^{\frac{k-j}{1+(k-1)\alpha}\alpha_j}$, gives:
\begin{eqnarray*}
\Theta_j(s,T,\q)\le C \sum_{k=j}^n\int_0^{2T^{1+(k-1)\alpha}|q_k|^\alpha} \exp(-C^{-1}v^{jk}) (v^{jk})^{\frac{k-j}{1+(k-1)\alpha}\alpha_j}\frac{dv^{jk}}{(v^{jk})^{\frac{(k-1)\alpha}{1+(k-1)\alpha}}}\le C.
\end{eqnarray*}
}
We eventually get:
\begin{eqnarray*}
||\mathcal{F}(\Delta_{\x_j}^{\frac{\alpha_j}{2}}G_\varepsilon f)||_{L^2(\Sc)}^2
\le C  \int_{-T}^T \int_{\R^{N}}|\mathcal{F}(f)(s,\q)|^2 d\q ds = C ||\mathcal{F}(f)||_{L^2(\Sc)}^2.
\end{eqnarray*}
The assertion now follows for $f\in {\mathscr S}(\R^{1+N}) $ from Plancherel's lemma. The result for $f\in L^2(\R^{1+N}) $ is derived by density. \qed

\subsection{Proof of Lemma \ref{LEMME_DEV}}
\label{SEC_PROV_DEV}

To establish Lemma \ref{LEMME_DEV}, we will thoroughly exploit the important relation \eqref{CORRESP_L_S} for the marginals between the degenerate Ornstein-Uhlenbeck process $\Lambda^\x$ and the non degenerate $\R^N $-valued stable process $S$ introduced  in Remark \ref{ID_LAW}.
  Setting
\begin{equation}
\label{DEF_RHO_GAMMA}
\rho:=\rho(s-t,e^{(s-t)A}\x-\y) ,\ \gamma:=\rho(\sigma-t,e^{(\sigma-t)A}\x-\bxi) ,
\end{equation}
 we focus for $i\in \leftB 1,n\rightB $ on the quantities: 
\begin{eqnarray}
I_i&:=&\int_{s\ge t\vee \sigma , \, \rho\ge K \gamma }|\Delta_{\x_i}^{\frac{\alpha_i}{2}}p_\Lambda (s-t,\x,\y)-\Delta_{\x_i}^{\frac{\alpha_i}{2}}p_\Lambda (s-\sigma,\bxi,\y) | d\y ds\nonumber\\
&=&\int_{s\ge t\vee \sigma , \, \rho\ge K \gamma }\Big|\Delta_{\x_i}^{\frac{\alpha_i}{2}}\frac{1}{\det(\M_{s-t})}p_S \big(s-t,\M_{s-t}^{-1}(e^{(s-t)A}\x-\y)\big)
\nonumber\\  &-& \Delta_{\x_i}^{\frac{\alpha_i}{2}}\frac{1}{\det(\M_{s-\sigma})}p_S \big(s-\sigma,\M_{s-\sigma}^{-1}(e^{(s-\sigma)A}\bxi-\y)\big)\Big|d\y ds
\label{DEF_I}
\end{eqnarray}
 (cf. Proposition \ref{THE_PROP_FUND} and equation \eqref{CORRESP_L_S} for the last identity).
 Recalling the important correspondence \eqref{CORRISPONDENZA}, the quantity $I_i$ in \eqref{DEF_I} then rewrites:
\begin{eqnarray}
I_i=\int_{s\ge t\vee \sigma , \, \rho\ge K \gamma }\Big|\frac{\Delta^{\frac{\alpha_i}{2},A,i,s-t} p_S \big(s-t,\M_{s-t}^{-1}(e^{(s-t)A}\x-\y)\big)}{\det(\M_{s-t})}-\frac{\Delta^{\frac{\alpha_i}{2},A,i,s-\sigma} p_S \big(s-\sigma,\M_{s-\sigma}^{-1}(e^{(s-\sigma)A}\bxi-\y)\big)}{\det(\M_{s-\sigma})}\Big| d\y ds\nonumber\\
=\int_{\rho\ge K \gamma }\Big|\frac{\Delta^{\frac{\alpha_i}{2},A,i,s-t}p_S \big(s-t,\M_{s-t}^{-1}(e^{(s-t)A}\x-\y)\big) }{\det(\M_{s-t})}-\frac{\Delta^{\frac{\alpha_i}{2},A,i,s-\sigma}p_S \big(s-\sigma,\M_{s-\sigma}^{-1}(e^{(s-t)A}\x-\y)\big) }{\det(\M_{s-\sigma})}\Big| d\y ds \nonumber\\
+\int_{s\ge t\vee \sigma , \, \rho\ge K \gamma }
\Big|\frac{\Delta^{\frac{\alpha_i}{2},A,i,s-\sigma}p_S \big(s-\sigma,\M_{s-\sigma}^{-1}(e^{(s-t)A}\x-\y)\big)}{\det(\M_{s-\sigma})}-\frac{\Delta^{\frac{\alpha_i}{2},A,i,s-\sigma}p_S \big(s-\sigma,\M_{s-\sigma}^{-1}(e^{(s-\sigma)A}\bxi-\y)\big)}{\det(\M_{s-\sigma})}
 \Big| d\y ds\nonumber\\
 =:(I_{i,T}+I_{i,S}).\label{SPLITTING_TIME_SPACE}
\end{eqnarray}
Hence, to prove that $I_i$ is bounded we need to investigate the time and space sensitivities of $\Delta^{\frac {\alpha_i} 2,A,i,\cdot} p_S$. This is the purpose of the next subsection.

\subsubsection{Preliminary Estimates}
We begin with the following result.
\begin{lemma}[Bounds and Sensitivities of the Stable Singular Kernel]
\label{SENS_SING_STAB}
Let $(S_t)_{t\ge 0}$ be a symmetric $\alpha$-stable process in $\R^N$ with non degenerate L\'evy measure $\nu_S $, i.e. its spectral measure $\mu_S$ satisfies that there  exists $\eta\ge 1$ s.t. for all $  \p \in \R^N $:
$$\eta^{-1}|\p|^\alpha\le \int_{\S^{N-1}} |\langle \p, \s\rangle |^\alpha \mu_S(d\s)\le \eta|\p|^\alpha.$$
$\alpha \in (0,2)$. This condition amounts to say that $(S_t)_{t\ge 0}$ satisfies assumption \A{ND} in dimension $N$.
 In particular, this implies that for all $t>0 $, $S_t $ has a smooth density that we denote by $p_S(t,\cdot) $ \nero{ (cf. Remark \ref{ID_LAW}).}

There exists a family of probability densities
$\big( q(t,\cdot)\big)_{t>0}$   on $\R^N $ such that \nero{ $q(t,\x) = {t^{-N/ \alpha}} \, q (1, t^{-\frac 1\alpha} \x),$ $ t>0, $ $\x \in \R^N$},  
for all $\delta \in [0,\alpha) $, there exists a constant $C_{\delta}:=C_{\delta}(\A{A})>0$ s.t. 
\begin{equation}
\label{INT_Q}
\int_{\R^N}q(t,\x)|\x|^\delta d\x
\le C_{\delta}t^{\frac{\delta}{\alpha}},\;\;\; t>0,
\end{equation}
and the following controls hold:
\begin{trivlist}
\item[\textit{(i)}] There exists $C:=C(\A{A})$ s.t. for all $i\in \leftB 1,n\rightB $ and $t>0,\ \x\in \R^N $:
\begin{equation}
\label{CTR_BOUND_SING_KER}
|\Delta^{\frac{\alpha_i} 2,A,i,t}p_S(t,\x)|\le \frac{C}{t} q(t,\x).
\end{equation}
\item[\textit{(ii)}] 
For all $\beta \in (0,1] $, there exists $C:=C(\beta,\A{A})$ s.t. for all $i\in \leftB 1,n\rightB  $ and $t>0,\ (\x,\x')\in \R^{2N} $:
\begin{equation}
\label{CTR_SENSI_HOLDER}
|\Delta^{\frac{\alpha_i} 2,A,i,t}p_S(t,\x)-\Delta^{\frac{\alpha_i} 2,A,i,t}p_S(t,\x')|\le \frac{C}{t}
\left(\frac{|\x-\x'|}{t^{\frac{1}{\alpha}}}\right)^\beta\big[q(t,\x)+q(t,\x')\big].
\end{equation}
\item[\textit{(iii)}] 
There exists $C:=C(\A{A})$ s.t. for all $i\in \leftB 1,n\rightB $ and  $t>0,\ \x\in \R^N $:
\begin{equation}
\label{CTR_DER_TEMPS_SING_KER}
|\partial_t \Delta^{\frac{\alpha_i} 2,A,i,t}p_S(t,\x)|\le \frac{C}{t^2} q(t,\x).
\end{equation}
\end{trivlist}
\end{lemma}
 \begin{remark}
\label{NOTATION_Q} \nero{
From now on, for the family of stable densities $\big(q(t,\cdot)\big)_{t>0} $, 
we also use the notation  $q(\cdot):=q(1,\cdot) $,  i.e. without any specified argument $q(\cdot)$ stands for the density $q$ at time $1$.}
\end{remark} 
\begin{proof}
\nero{It is enough to find  a suitable  $q$ for each estimate \eqref{CTR_BOUND_SING_KER}, \eqref{CTR_SENSI_HOLDER} and \eqref{CTR_DER_TEMPS_SING_KER};  summing up such densities one gets the required final $q$.}

Let us first write for all $i\in \leftB 1,n\rightB $ (cf. \eqref{DEF_DELTA_MOD}):
\begin{eqnarray}  
\Delta^{\frac {\alpha_i} 2,A,i,t} p_S(t,\x)=\int_{\R^{d_i}}\Big(p_S(t,\x+t^{-(i-1)}(e^A)_iz)-p_S(t,\x)\Big)\frac{dz}{|z|^{d_i+\alpha_i}} \notag\\
=\int_{\R^{d_i}}\Big( p_S(t,\x+t^{-(i-1)}(e^A)_iz)-p_S(t,\x)\Big)\I_{|z|\le t^{\frac 1{\alpha_i}}}\frac{dz}{|z|^{d_i+\alpha_i}}\notag\\ \nonumber
+\int_{\R^{d_i}}\Big(p_S(t,\x+t^{-(i-1)}(e^A)_iz)-p_S(t,\x)\Big)\I_{|z|> t^{\frac 1{\alpha_i}}}\frac{dz}{|z|^{d_i+\alpha_i}}
\\ \label{DICO_PICCOLI_GRANDI_SALTI_DI_A}
=: \LALS{i}{t} p_S(t,\x)+\LALB{i}{t} p_S(t,\x), 
\end{eqnarray}
where $\LALS{i}{t}$ (resp. $\LALB{i}{t} $) corresponds to the \textit{small} jumps  (resp. to the \textit{large} jumps) part of $\LAL{i}{t} $.

 Let us recall that, for a given fixed $t>0$, we can use an It\^o-L\'evy  decomposition
 at the associated  characteristic stable time scale (i.e. the truncation is performed at the threshold $t^{\frac {1} {\alpha}} $) 
to write $S_t:=M_t+N_t$
where $M_t$ and $N_t $ are independent random variables \nero{ (we are considering a probability space $(\Omega,\F,\P) $ on which the process $S = (S_s)_{s\ge 0} $ is defined; $\E $ denotes the associated expectation).}  More precisely, 
 \begin{equation} \label{dec}
 N_s = \int_0^s \int_{ |x| > t^{\frac {1} {\alpha}} }
\; x  P(du,dx), \;\;\; \; M_s = S_s - N_s, \;\; s \ge 0,
 \end{equation} 
where $P$ is the  Poisson random measure associated with the process $S$; for the considered fixed $t>0$,
 $M_t$ and $N_t$ correspond to
 the \textit{small jumps part } and
\textit{large jumps part} respectively. 
A similar decomposition has been already used in
 \cite{wata:07},        \cite{szto:10} 
 and \cite{huan:meno:15}. It is useful to note that the cutting threshold in \eqref{dec} precisely yields for the considered $t>0$ that:
\begin{equation} \label{ind}
N_t  \overset{({\rm law})}{=} t^{\frac 1\alpha} N_1 \;\; \text{and} \;\;
M_t  \overset{({\rm law})}{=} t^{\frac 1\alpha} M_1.
\end{equation}  
To check the assertion about $N$ we start with 
$$
\E [e^{i \langle \p , N_t \rangle}] = 
\exp \Big(  t
\int_{\S^{N-1}} \int_{t^{\frac 1\alpha}}^{\infty}
 \Big(\cos (\langle \p, r\bxi \rangle)  - 1  \Big) \, \frac{dr}{r^{1+\alpha}}\tilde \mu_{S}(d\bxi) \Big), \;\; \p \in \R^N
$$
(see  \eqref{STABLE_MEAS} and \cite{sato}). Changing variable $\frac{r}{t^{\frac 1\alpha}} =s$
we get that $\E [e^{i \langle \p , N_t \rangle}]$ $= \E [e^{i \langle \p , t^{\frac 1\alpha} N_1 \rangle}]$ for any $\p \in \R^N$ and this shows the assertion (similarly we get
the statement for $M$).
The density of $S_t$ then writes
\begin{equation}
\label{DECOMP_G_P}
p_S(t,\x)=\int_{\R^N} p_{M}(t,\x-\bxi)P_{N_t}(d\bxi),
\end{equation}
where $p_M(t,\cdot)$ corresponds to the density of $M_t$ and $P_{N_t}$ stands for the law of $N_t$. \nero{From Lemma \ref{EST_DENS_MART} (see as well
Lemma B.1  }
in \cite{huan:meno:15}),  $p_M(t,\cdot)$  belongs to the Schwartz class ${\mathscr S}(\R^N) $ and satisfies that for all $m\ge 1 $ and all { multi-index $\i=(i^1,\cdots, i^N) \in \N^N , \ |\i|:=\sum_{j=1}^N  i_j\le 3$,} there exists $C_{m,\i}$ s.t. for all $(t,\x)\in \R_+^*\times \R^N  $:
\begin{equation}
\label{CTR_DER_M}
\nero{
|\partial_\x^\i p_M(t,\x)|\le \frac{\bar C_{m,\i}}{t^{\frac{|\i|}{\alpha} }} \, p_{\bar M}(t,\x),\;\; \text{where} \;\; p_{\bar M}(t,\x)
:=
\frac{C_{m}}{t^{\frac{N}{\alpha}}} \left( 1+ \frac{|\x|}{t^{\frac 1\alpha}}\right)^{-m}
}
\end{equation}
where the above modification of the constant is performed in order that {\it $p_{\bar M}(t,\cdot ) $ be a probability density.}
 Note that the asymptotic decay of $p_{\bar M} $ here depends on on the integer $m$ considered. For our analysis, recalling from Remark \ref{ID_LAW} and equation \eqref{CORRESP_L_S} that we are disintegrating the density of a non degenerate stable process in dimension $N$, and that we are led to investigate sensitivities, which involve for the small jumps derivatives up to order 2 or 3 (depending on $\alpha \in (0,1) $ or $\alpha \in [1,2) $), see e.g. \eqref{CTR_SMALL_DIAG} below, we can fix
 $$m:=N+4.$$
Let us emphasize that, to establish the indicated results, as opposed to \cite{huan:meno:15}, we only focus on integrability properties and not on pointwise density estimates. Our global approach therefore consists in exploiting \nero{\eqref{dec},}  \eqref{DECOMP_G_P} and \eqref{CTR_DER_M}.
The various sensitivities will be expressed through derivatives of $p_M(t,\cdot)$, which also gives the corresponding time singularities. However, as for general stable processes, the integrability restrictions come from the large jumps (here $N_t $) and only depend on its index $\alpha$.
A crucial point then consists in observing  that the convolution $\int_{\R^N}p_{\bar M}(t,\x-\bxi)P_{N_t}(d\bxi) $ actually corresponds to the density of the random variable 
\begin {equation} \label{we}
\bar S_t:=\bar M_t+N_t,\;\; t>0 
\end{equation}
 (where $\bar M_t $ has density $p_{\bar M}(t,.)$ and is independent of $N_t $; \nero{to have such decomposition one can define each $\bar S_t$ on a product probability space}). Then, the integrability properties of $\bar M_t+N_t $, and more generally of all random variables appearing below, come from those of $\bar M_t $ and $N_t$. 

One can easily check that $p_{\bar M}(t,\x) = {t^{-\frac N \alpha}} \, p_{\bar M} (1, t^{-\frac 1\alpha} \x),$ $ t>0, \, $ $\x \in \R^N.$  Hence 
$$
\bar M_t  \overset{({\rm law})}{=} t^{\frac 1\alpha} \bar M_1,\;\;\; N_t  \overset{({\rm law})}{=} t^{\frac 1\alpha}  N_1.
$$
By independence of $\bar M_t$ and $N_t$, using the Fourier transform, one can easily prove that 
\begin{equation} \label{ser1}
\bar S_t  \overset{({\rm law})}{=} t^{\frac 1\alpha} \bar S_1.
\end{equation} 
Moreover, 
$
\E[|\bar S_t|^\delta]=\E[|\bar M_t+N_t|^\delta]\le C_\delta t^{\frac\delta \alpha}(\E[|\bar M_1|^\delta]+\E[| N_1|^\delta])\le C_\delta t^{\frac\delta \alpha}, \; \delta \in (0,\alpha).
$ 
This shows that the density of $\bar S_t$ verifies \eqref{INT_Q}. 
We now give the details of the computations in case \textit{(ii)}. 
This case contains the main difficulties, the other ones can be derived similarly. Write for all $t>0,(\x,\x')\in \R^{2N} $ (cf. \eqref{DICO_PICCOLI_GRANDI_SALTI_DI_A}):
\begin{eqnarray}
\LALB{i}{t} p_S(t,\x)-\LALB{i}{t} p_S(t,\x')\notag\\
=\int_{|z|\ge t^{\frac 1{\alpha_i}}}[p_S(t,\x+t^{-(i-1)}(e^A)_iz)-p_S(t,\x'+t^{-(i-1)}(e^A)_iz)] \frac{dz}{|z|^{d_i+\alpha_i}}\nonumber\\
\nero{-}\Big(p_S(t,\x)-p_S(t,\x')\Big)\int_{|z|\ge t^{\frac 1{\alpha_i}}}\frac{dz}{|z|^{d_i+\alpha_i}},\nonumber\\
\Big|\LALB{i}{t} p_S(t,\x)-\LALB{i}{t} p_S(t,\x')\Big|\notag\\
\le \left(\int_{|z|\ge t^{\frac 1{\alpha_i}}} |p_S(t,\x+t^{-(i-1)}(e^A)_iz)-p_S(t,\x'+t^{-(i-1)}(e^A)_iz)|\frac{dz}{|z|^{d_i+\alpha_i}} \right)\nonumber\\
+\left(\frac{C}t \Big|p_S(t,\x)-p_S(t,\x') \Big|\right)
=:
(I_{i,1}+I_{2})(t,\x,\x').\label{CTR_PREAL}
 \end{eqnarray}
 Assume for a while that the following control holds. For all \nero{$\beta\in (0,1] $} there exists a constant $C:=C_\beta $ s.t. for all $t>0,\ (\x,\x')\in \R^{2N} $,
 \begin{equation}
 \label{HOLDER_S}
|p_S(t,\x)-p_S(t,\x')|\le C \left(\frac{|\x-\x'|}{t^{\frac 1\alpha}}\right)^\beta \big(p_{\bar S}(t,\x)+p_{\bar S}(t,\x')\big),
 \end{equation}
\nero{where  $p_{\bar S}(t,\cdot) $ stands for the density of $\bar S_t $
which verifies  \eqref{INT_Q}.}
  From \eqref{CTR_PREAL} and \eqref{HOLDER_S}   we readily derive:
\begin{equation}
\label{CTR_I2}
|I_2(t,\x,\x')|\le \frac{C}{t}\left( \frac{|\x-\x'|}{t^{\frac 1 \alpha}}\right)^\beta \big(p_{\bar S}(t,\x)+p_{\bar S}(t,\x') \big).
\end{equation}
Also, still from \eqref{CTR_PREAL} and \eqref{HOLDER_S},
\begin{eqnarray*}
|I_{i,1}(t,\x,\x')|&\le& C \left( \frac{|\x-\x'|}{t^{\frac 1 \alpha}}\right)^\beta
\int_{|z|\ge t^{\frac 1{\alpha_i}}}  \big(p_{\bar S}(t,\x+t^{-(i-1)}(e^A)_iz)+p_{\bar S}(t,\x'+t^{-(i-1)}(e^A)_iz) \big) \frac{dz}{|z|^{d_i+\alpha_i}}\\
&=:&\frac{C}t \left( \frac{|\x-\x'|}{t^{\frac 1 \alpha}}\right)^\beta
\int_{\R^{d_i}}  \big(p_{\bar S}(t,\x+t^{-(i-1)}(e^A)_iz)+p_{\bar S}(t,\x'+t^{-(i-1)}(e^A)_i z) \big) f_{\Gamma^i}(t,z)(dz),
\end{eqnarray*}
\nero{setting $f_{\Gamma^i}(t,z):=t c_{\alpha_i,d_i}\I_{|z|\ge t^{\frac 1{\alpha_i}}}\frac{1}{|z|^{d_i+\alpha_i}} $} with $c_{\alpha,d_i}>0$ s.t. $\int_{\R^{d_i}}f_{\Gamma^i}(t,z) dz=1 $. Hence, $f_{\Gamma^i}(t, \cdot) $ is the density of an $\R^{d_i}$-valued random variable $\Gamma_t^i $. The above integrals  can thus be seen as the densities, at point $\x$ and $\x'$ respectively, of the random variable
\begin{equation}
\label{DEF_bar_S1}
\bar S_t^{i,1}:=\bar S_t+t^{-(i-1)}(e^A)_i\Gamma_t^i,
\end{equation}
\nero { where $\bar S_t$ is as in \eqref{we} and $\Gamma_t^i$ is independent of $\bar S_t$ and has density $f_{\Gamma^i}(t, \cdot) $. Note that 
\begin{eqnarray*}
f_{\Gamma^i}(t,z)      =c_{\alpha_i,d_i} \I_{\frac{|z|}{t^{\frac 1{\alpha_i}}}\ge 1} \Big(\frac{|z|}{t^{\frac1{\alpha_i}}}\Big)^{- {d_i} - {\alpha_i}} t^{-\frac {d_i}{\alpha_i}}={t^{- \frac {d_i}{\alpha_i}}} f_{\Gamma^i}\Big(1,\frac{z}{t^{\frac1{\alpha_i}}}\Big).
 \end{eqnarray*}
If we set 
$\tilde \Gamma_t^i:=(t^{-(i-1)})\Gamma_t^i $ and denote by $f_{\tilde \Gamma^i}(t,\cdot)$ its density, we find 
\begin{eqnarray*}
f_{\tilde \Gamma^i}(t,z)=t^{d_i(i-1)}f_{\Gamma^i}(t,z t^{i-1})=t^{d_i(i-1)}t^{-\frac{d_i}{\alpha_i}}f_{\Gamma^i}(1,z t^{i-1}/t^{\frac1{\alpha_i}}):=\frac1{t^{d_i(-(i-1)+\frac 1{\alpha_i})}}f_{\Gamma^i}\Big(1,\frac{z}{t^{-(i-1)+\frac 1{\alpha_i}}}\Big), 
\end{eqnarray*}
 $z \in \R^{d_i}.$ Since $\alpha_i=\frac{\alpha}{1+\alpha(i-1)} $, we have
\begin{eqnarray*}
f_{\tilde \Gamma^i}(t,z)=\frac{1}{t^{\frac{d_i}{\alpha}}} f_{\Gamma^i}(1, \frac{z}{t^{\frac 1\alpha}}).
\end{eqnarray*}
It follows that $(e^A)_i \tilde \Gamma_t^i \overset{({\rm law})}{=}
t^{\frac 1\alpha} (e^A)_i \tilde \Gamma_1^i $. Hence arguing as for \eqref{ser1} we find 
\begin{equation} \label{dee}
\bar S_t^{i,1}:=\bar S_t+t^{-(i-1)}(e^A)_i\Gamma_t^i  \overset{({\rm law})}{=} t^{\frac 1\alpha} \bar S_1^{i,1}
\end{equation}
and, moreover, for any $ \delta \in (0,\alpha)$,
\begin{eqnarray}
\label{INT_2}
\E[|\bar S_t^{i,1}|^\delta]&\le& C_\delta(\E[|\bar S_t|^\delta]+C_A\, t^{-(i-1)\delta}\E[|\Gamma_t^i|^\delta])
 \le  C_{\delta,A}t^{\frac{\delta}{\alpha}}.
\end{eqnarray}
as required in \eqref{INT_Q}.  } We finally obtain,
\begin{equation}
\label{CTR_I1}
|I_{i,1}(t,\x,\x')|\le \frac Ct \left( \frac{|\x-\x'|}{t^{\frac 1 \alpha}}\right)^\beta \big(p_{\bar S^{i,1}}(t,\x)+p_{\bar S^{i,1}}(t,\x')\big).
\end{equation}
The control for $|\LALB{i}{t} p_S(t,\x)-\LALB{i}{t} p_S(t,\x')| $ follows plugging \eqref{CTR_I1} and \eqref{CTR_I2} into \eqref{CTR_PREAL} defining $q(t,\cdot):= 
\frac{1}{n+1}(\sum_{i=1}^n p_{\bar S^{i,1}}+p_{\bar S})(t,\cdot).$


It remains to control the difference associated with the \textit{small jumps} part. \textcolor{black}{We first recall that for any $\alpha=\alpha_1 \in(0,2) $, we have that for $i\in \leftB 2,n\rightB, \alpha_i\in (0,1) $. Also, we consider first for simplicity the case $\alpha_1=\alpha\in (0,1) $. Write then}, using also \eqref{DECOMP_G_P},
\begin{eqnarray*} 
&&\LALS{i}{t} p_S(t,\x)-\LALS{i}{t} p_S(t,\x')\nonumber\\
&=&\int_{|z|\le t^{\frac 1{\alpha_i}}}\Big[\big(p_S(t,\x+t^{-(i-1)}(e^A)_iz)-p_S(t,\x)\big) -\big(p_S(t,\x'+t^{-(i-1)}(e^A)_iz)-p_S(t,\x')\big)\Big] \frac{dz}{|z|^{d_i+\alpha_i}}\nonumber\\
&=&\int_{|z|\le t^{\frac 1{\alpha_i}}} \int_{\R^N} \int_0^1 d\mu \Big(\nabla p_{ M}(t,\x+\mu t^{-(i-1)}(e^A)_iz -\bxi)-\nabla p_{M}(t,\x'+\mu t^{-(i-1)}(e^A)_iz-\bxi)\Big)\cdot  t^{-(i-1)}(e^A)_iz\notag\\
&& \;\;\;\; \times \,P_{N_t}(d\bxi) \frac{dz}{|z|^{d_i+\alpha_i}}.
\end{eqnarray*}
In the sequel  we will use \eqref{CTR_DER_M} \nero{with $|{\mathbf i}|=2 $ and the following inequality
\begin{equation} \label{ste}
p_{\bar M}(t,\y+ \bzeta)\le C_m p_{\bar M}(t,\y),\;\;\; \text{if} \;\; |\bzeta|\le 2t^{\frac 1\alpha},\;\; t>0,\; \y,\, \bzeta \in \R^N,
\end{equation} 
 for some $C_m>0$.
This can be  easily proved considering separately the cases $|\y| \le 4t^{\frac 1\alpha}$ and 
$|\y| > 4t^{\frac 1\alpha} $ (if $|\y| > 4t^{\frac 1\alpha} $ one can use $\frac{|\y + \bzeta|}{t^{\frac 1\alpha}}$ $ \ge \frac{|\y| - |\bzeta|}{t^{\frac 1\alpha}} $ $\ge \frac{|\y|}{2t^{\frac 1\alpha}} $; if  $|\y| \le 4t^{\frac 1\alpha} $ one can use $ 1 \ge \frac{|\y|}{ 4t^{\frac 1\alpha}}$)}.  

We derive if $|\x-\x'|\le t^{\frac{1}{\alpha}} $:
\begin{align} \label{CTR_SMALL_DIAG}
|\LALS{i}{t} p_S(t,\x)-\LALS{i}{t} p_S(t,\x')|\nonumber\\
\le
C_m\frac{|\x-\x'|}{t^{\frac 2 \alpha}}
\int_{|z|\le t^{\frac 1{\alpha_i}}} \int_{\R^N} \int_0^1 d\mu\int_0^1 [ p_{\bar M}\big(t,\x'+\mu t^{-(i-1)}(e^A)_iz+\lambda(\x-\x') -\bxi\big) ] \, d\lambda \,  P_{N_t}(d\bxi)|z|t^{-(i-1)} \frac{dz}{|z|^{d_i+\alpha_i}}
\nonumber\\
\le
C_m \frac{|\x-\x'|}{t^{\frac 2\alpha}}  \int_{|z|\le t^{\frac 1{\alpha_i}}} \int_{\R^N} p_{\bar M}(t,\x'-\bxi)P_{N_t}(d\bxi)t^{-(i-1)}|z|\frac{dz}{|z|^{d_i+\alpha_i}}
\le
\frac{C_m}{t}\frac{|\x-\x'|}{t^{\frac 1 \alpha}} \nero{p_{\bar S}(t,\x')}
\\  \nonumber\le \frac{C_m}{t}\left( \frac{|\x-\x'|}{t^{\frac 1 \alpha}}\right)^\beta p_{\bar S}(t,\x').
\end{align}
The second inequality follows from \eqref{ste} taking $\y=\x'-\bxi,\ \bzeta=\mu t^{-(i-1)}(e^A)_iz+\lambda(\x-\x') $, using the fact that on the considered set, i.e.  $|\x-\x'|\le t^{\frac 1\alpha},$ $ |z|\le t^{\frac 1{\alpha_i}} $, since $\alpha_i=\frac{\alpha}{1+\alpha(i-1)} $, \nero{we have that $t^{-(i-1)} |(e^A)_i z|+ |\x-\x'|\le 2 t^{\frac 1\alpha} $. }
We have exploited as well that:
$$ t^{-(i-1)} \int_{|z|\le t^{\frac 1{\alpha_i}}} \frac{|z|}{|z|^{d_i+\alpha_i}}dz \le C_{\alpha,i}t^{-(i-1)} t^{-1+\frac{1}{\alpha_i}}=C_{\alpha,i}t^{-1+\frac 1\alpha}.$$
If now $|\x-\x'|> t^{\frac{1}{\alpha}} $, we derive from 
\nero{\eqref{CTR_DER_M}} \nero{(using again \eqref{ste} as before):}
\begin{eqnarray}
|\LALS{i}{t} p_S(t,\x)-\LALS{i}{t} p_S(t,\x')|\nonumber\\
\le \frac{C_m}{t^{\frac 1\alpha}}\int_{|z|\le t^{\frac 1{\alpha_i}}}\int_{\R^N} \int_0^1  \Big( p_{\bar M}(t,\x+\mu t^{-(i-1)} (e^A)_iz -\bxi)+ p_{\bar M}(t,\x'+\mu t^{-(i-1)}(e^A)_iz-\bxi) \Big) d\mu P_{N_t}(d\bxi)\notag\\
\times \, t^{-(i-1)}  |z|\frac{dz}{|z|^{d_i+\alpha_i}}
\le \frac{C_m}{t^{\frac 1\alpha}}\int_{|z|\le t^{\frac 1{\alpha_i}}}\int_{\R^N} \int_0^1 \Big( p_{\bar M}(t,\x -\bxi)+ p_{\bar M}(t,\x'-\bxi) \Big) d\mu
 P_{N_t}(d\bxi)  t^{-(i-1)}|z|\frac{dz}{|z|^{d_i+\alpha_i}}\nonumber\\
\le \frac{C_m}{t} \Big( p_{\bar S}(t,\x)+p_{\bar S}(t,\x')\Big)\le \frac{C_m}{t} \left( \frac{|\x-\x'|}{t^{\frac{1}{\alpha}}}\right)^\beta\Big( p_{\bar S}(t,\x)+p_{\bar S}(t,\x')\Big).
\label{CTR_SMALL_OUT_DIAG}
\end{eqnarray}
Equations \eqref{CTR_SMALL_DIAG} and \eqref{CTR_SMALL_OUT_DIAG} give the stated control for $|\LALS {i}{t}p_S(t,\x)-\LALS {i}{t} p_S(t,\x')|  $. This gives \textit{(ii)} for $\alpha\in (0,1) $. 
The control \textit{(i)} can be obtained following the same lines, without handling differences of starting points. To handle the small jumps in  the remaining case $i=1, \alpha_1=\alpha\in [1,2) $, a second order Taylor expansion is needed in the previous computations. Write, in short:
\begin{eqnarray} \label{PRELIM_SMALL_JUMPS_2}
&&\LALS{1}{t} p_S(t,\x)-\LALS{1}{t} p_S(t,\x')\nonumber\\
&=&\int_{|z|\le t^{\frac 1{\alpha}}}\Big[\big(p_S(t,\x+(e^A)_1z)-p_S(t,\x)-\nabla p_S(t,\x')\cdot (e^A)_1z\big) \notag\\
&&-\big(p_S(t,\x'+(e^A)_1z)-p_S(t,\x')-\nabla p_S(t,\x')\cdot (e^A)_1z \big)\Big] \frac{dz}{|z|^{d+\alpha}}\nonumber\\
&=&\int_{|z|\le t^{\frac 1{\alpha}}} \int_{\R^N} \int_0^1 d\mu (1-\mu)\notag\\
&&\left\langle\Big(D^2 p_{ M}(t,\x+\mu (e^A)_1z -\bxi)-D^2 p_{M}(t,\x'+\mu (e^A)_1z-\bxi)\Big)  (e^A)_1z, (e^A)_1z\right\rangle
\times \,P_{N_t}(d\bxi) \frac{dz}{|z|^{d+\alpha}}.
\end{eqnarray}
Now, if $|\x-\x'|\le t^{\frac 1 \alpha} $, similarly to \eqref{CTR_SMALL_DIAG}:
\begin{align}
|\LALS{1}{t} p_S(t,\x)-\LALS{1}{t} p_S(t,\x')|\nonumber\\
\le
C_m \frac{|\x-\x'|}{t^{\frac 3\alpha}}  \int_{|z|\le t^{\frac 1{\alpha}}} \int_{\R^N} p_{\bar M}(t,\x'-\bxi)P_{N_t}(d\bxi)|z|^2\frac{dz}{|z|^{d+\alpha}}
\le \frac{C_m}{t}\left( \frac{|\x-\x'|}{t^{\frac 1 \alpha}}\right)^\beta p_{\bar S}(t,\x').\label{NUMB2_2}
\end{align}
If now $|\x-\x'|> t^{\frac{1}{\alpha}} $, we derive from \eqref{PRELIM_SMALL_JUMPS_2},
\nero{\eqref{CTR_DER_M}} \nero{(using again \eqref{ste} as before):}
\begin{eqnarray}
|\LALS{1}{t} p_S(t,\x)-\LALS{1}{t} p_S(t,\x')|\nonumber\\
\le \frac{C_m}{t^{\frac 2\alpha}}\int_{|z|\le t^{\frac 1{\alpha}}}\int_{\R^N} \int_0^1  (1-\mu) \Big( p_{\bar M}(t,\x+\mu  (e^A)_1z -\bxi)+ p_{\bar M}(t,\x'+\mu (e^A)_1z-\bxi) \Big) d\mu P_{N_t}(d\bxi)\notag\\
\times \,   |z|^2\frac{dz}{|z|^{d+\alpha}}
\le \frac{C_m}{t} \Big( p_{\bar S}(t,\x)+p_{\bar S}(t,\x')\Big)\le \frac{C_m}{t} \left( \frac{|\x-\x'|}{t^{\frac{1}{\alpha}}}\right)^\beta\Big( p_{\bar S}(t,\x)+p_{\bar S}(t,\x')\Big).
\label{CTR_SMALL_OUT_DIAG_2AL}
\end{eqnarray}
Equations \eqref{NUMB2_2} and \eqref{CTR_SMALL_OUT_DIAG_2AL} complete the proof of \textit{(ii)}, \textit{(i)} for $\alpha \in [1,2) $.

Let us now deal with \textit{(iii)}. We consider for simplicity $\alpha\in (0,1) $. The case $\alpha\in [1,2) $ could be handled as above considering an additional first order term in the integral (see \eqref{PRELIM_SMALL_JUMPS_2}).
Write, for $i\in \leftB 1,n\rightB $, $t>0, \x\in \R^N $:
\begin{eqnarray}
\partial_t \Delta^{\frac{\alpha_i}2,A,i,t} p_S(t,\x)
&=&\partial_t\Bigg( \frac{1}{t^{(i-1)\alpha_i}}\int_{\R^{d_i}}\Big(p_S(t,\x+(e^A)_i \tilde z)-p_S(t,\x) \Big) \frac{d\tilde z}{|\tilde z|^{d_i+\alpha_i}}\Bigg)\notag\\
&=&-\frac{(\alpha_i(i-1))}{t^{(i-1)\alpha_i+1}} \int_{\R^{d_i}}\Big(p_S(t,\x+(e^A)_i \tilde z)-p_S(t,\x) \Big) \frac{d\tilde z}{|\tilde z|^{d_i+\alpha_i}}\notag\\
&&+\frac{1}{t^{(i-1)\alpha_i}}\int_{\R^{d_i}}\Big(\partial_t p_S(t,\x+(e^A)_i \tilde z)-\partial_t p_S(t,\x) \Big) \frac{d\tilde z}{|\tilde z|^{d_i+\alpha_i}}=:\Big(E_{i,1}+E_{i,2}\Big) (t,\x).\label{NUOVA_SENSI_TEMPO}
\end{eqnarray}
{Since}, $E_{i,1}(t,\x)=\frac{\alpha_i(1-i)}{t}\int_{\R^{d_i}}\Big(p_S(t,\x+t^{-(i-1)}(e^A)_iz)-p_S(t,\x) \Big) \frac{dz}{|z|^{d_i+\alpha_i}}=\frac{\alpha_i(1-i)}{t}\Delta^{\frac{\alpha_i}2,A,i,t} p_S(t,\x),$
point \textit{(i)} readily gives:
\begin{equation}
\label{NUOVA_SENSI_TEMPO_EI1}
|E_{i,1}(t,\x)|\le \frac{C}{t^2}q(t,\x).
\end{equation}
To investigate $E_{i,2}(t,\x) $ \nero{note  that $\partial_t p_S (t, z) = L_S p_S (t, z),$ $t > 0,$ $\z \in  \R^N$ (see, for instance, \cite{kolo:00}),  where $L_S$ is the generator of $S,$ namely, for $\varphi \in  C_0^{\infty} (\R^N ),$}
\def\stima11{
 observe first from \eqref{THE_DENS_OU} and \eqref{CORRESP_L_S} that for all $t>0, \z\in \R^N $:
$$p_S(t,\z)=\frac{1}{(2\pi)^N}\int_{\R^N}\exp(-i\langle \z,\p\rangle)\exp\Big(-t\int_{\S^{N-1}}|\langle \p,\bxi\rangle|^\alpha \mu_S(d,\bxi) \Big) d\p.$$
Thus,
\begin{eqnarray}
\partial_t p_S(t,\z)&=&\frac{1}{(2\pi)^N}\int_{\R^N}\exp(-i\langle \z,\p\rangle)\big(-\int_{\S^{N-1}}|\langle \p,\bxi\rangle|^\alpha \mu_S(d,\bxi) \big)\exp\Big(-t\int_{\S^{N-1}}|\langle \p,\bxi\rangle|^\alpha \mu_S(d,\bxi) \Big) d\p\notag \\
&=& L_Sp_S(t,\z), \label{SENSI_TEMPO_PS}
\end{eqnarray}
where $L_S$ is the generator of $S$. Namely, for $\varphi \in C_0^\infty(\R^N,\R) $,
}
\begin{equation}
\label{DECOMP_GENE_S}
L_S\varphi(\x)=\int_{\R^N} \Big(\varphi(\x+\z)-\varphi(\x)\Big) \nu_S(d\z)=\int_{0}^{+\infty}\frac{dr}{r^{1+\alpha}}\int_{\S^{N-1}} \Big(\varphi(\x+r \bxi)-\varphi(\x) \Big)\tilde \mu_S(d\bxi), \; \x \in \R^N,
\end{equation}
where $\tilde \mu_S=\frac{\mu_S}{C_{\alpha,N}} $ for a positive constant $C_{\alpha,N} $.
 We then rewrite from the definition in \eqref{NUOVA_SENSI_TEMPO} \nero{(see also the comments on $p_S$ after \eqref{CORRISPONDENZA})}:  
\begin{eqnarray*}
E_{i,2}(t,\x)&=&\frac{1}{t^{(i-1)\alpha_i}}\int_{\R^{d_i}}\Big(L_S p_S(t,\x+(e^A)_iz)-L_Sp_S(t,\x)\Big)\frac{dz}{|z|^{d_i+\alpha_i}}\\
&=&L_S\Big(\frac{1}{t^{(i-1)\alpha_i}}\int_{\R^{d_i}}\Big( p_S(t,\x+(e^A)_iz)-p_S(t,\x)\Big)\frac{dz}{|z|^{d_i+\alpha_i}}\Big) =L_S \Delta^{\frac{\alpha_i}2,A,i,t}p_S(t,\x)\\
&=&\int_{\R^N} \Big(\Delta^{\frac{\alpha_i}2,A,i,t}p_S(t,\x+\z)-\Delta^{\frac{\alpha_i}2,A,i,t}p_S(t,\x)\Big)\nu_S(d\z),
\end{eqnarray*}
using \eqref{DECOMP_GENE_S} for the last equality. The idea is now as above to introduce a cutting threshold at the characteristic time-scale $t^{\frac 1\alpha}$ for the variable $\z $. Write:
\begin{eqnarray}
E_{i,2}(t,\x)&=&\int_{|\z|\le t^{\frac 1\alpha}} \Big(\Delta^{\frac{\alpha_i}2,A,i,t}p_S(t,\x+\z)-\Delta^{\frac{\alpha_i}2,A,i,t}p_S(t,\x)\Big)\nu_S(d\z)\notag\\
&&+\int_{|\z|> t^{\frac 1\alpha}} \Big(\Delta^{\frac{\alpha_i}2,A,i,t}p_S(t,\x+\z)-\Delta^{\frac{\alpha_i}2,A,i,t}p_S(t,\x)\Big)\nu_S(d\z) :=\Big(E_{i,21}+E_{i,22}\Big)(t,\x).
\label{DECOMP_2_EI2}
\end{eqnarray}
 \nero{Hence we get from point \textit{(i)}  (see also \eqref{DECOMP_GENE_S}): 
\begin{align*}
|E_{i,22}(t,\x)|\le \frac{C}{t} \int_{|\z|> t^{\frac 1\alpha}} \big( q(t,\x+\z)+  q(t,\x)   \big) \nu_S(d\z)
\\
= \frac{C}{t} \int_{t^{\frac 1\alpha}}^{+\infty}\frac{dr}{r^{1+\alpha}}\int_{\S^{N-1}} \big( q(t,\x+ r \bxi)+  q(t,\x) \big) \tilde \mu_S(d\bxi)
= \frac{C}{t^2} \int_{1}^{+\infty}\frac{ds}{s^{1+\alpha}}\int_{\S^{N-1}} \big( q(t,\x+ s t^{\frac 1\alpha} \bxi)+  q(t,\x) \big) \tilde \mu_S(d\bxi) 
\\=
 \frac{C}{t^2} \int_{1}^{+\infty}\frac{ds}{s^{1+\alpha}}\int_{\S^{N-1}}  q(t,\x+ s t^{\frac 1\alpha} \bxi) \tilde \mu_S(d\bxi) +  \frac{C}{t^2}    q(t,\x), \; t>0, \,\x \in \R^N.
\end{align*}
Define the density $\bar q (t, \cdot)$:
$$\bar q(t, \x) := C_1 \int_{1}^{+\infty}\frac{ds}{s^{1+\alpha}}\int_{\S^{N-1}}  q(t,\x+ s t^{\frac 1\alpha} \bxi) \tilde \mu_S(d\bxi),\;\; t>0,\, \x \in \R^N.
$$
We  derive:
\begin{equation}
\label{CTR_EI22}
|E_{i,22}(t,\x)|\le \frac{C}{t^2}\big(\bar q(t,\x)+q(t,\x)\big).
\end{equation} 
By the properties of $q$ we deduce $\bar q(t, \x) = t^{- \frac N\alpha} \bar q(1,t^{-\frac1 \alpha}\,  \x )$. Moreover, by using the Fubini theorem, we can check  \eqref{INT_Q}  when $q$ is replaced by $\bar q$. 
} 
\def\stima12{
For $E_{i,22}(t,\x) $, we get from point \textit{(i)} that:
\begin{eqnarray*}
|E_{i,22}(t,\x)|\le \frac{C}{t} \Big(\int_{|\z|> t^{\frac 1\alpha}} q(t,\x+\z)\nu_S(d\z)+ \frac{1}t q(t,\x)\Big).
\end{eqnarray*}
Introducing the symmetric probability measure $\mu_{\Theta_t}(d\z)=t  c_{\alpha,N} \I_{|\z|>
 t^{\frac 1\alpha}} \nu_S(d\z) $, i.e. $ c_{\alpha,N}$ is chosen so that $\int_{\R^N}\mu_{\Theta_t}(d\z)=1 $, the integral $\int_{|\z|> t^{\frac 1\alpha}} q(t,\x+\z)\nu_S(d\z)=\frac{C}t\int_{\R^N} q(t,\x+\z)\mu_{\Theta_t}(d\z)$. This last integral corresponds to the density of the sum of independent random variables $Q_t$ with density $q(t,\cdot)$ and $\Theta_t $ with law $\mu_{\Theta_t} $. Setting $\bar q(t,\x):=\int_{\R^N} q(t,\x+\z)\mu_{\Theta_t}(d\z) $, we easily obtain, similarly to \eqref{INT_1},  that $\bar q(t,\cdot) $ satisfies the integrability condition \eqref{INT_Q}. We thus derive:
\begin{equation}
\label{CTR_EI22}
|E_{i,22}(t,\x)|\le \frac{C}{t^2}\big(\bar q(t,\x)+q(t,\x)\big).
\end{equation} 
}
For $E_{i,21}(t,\x) $ we use point \textit{(ii)} in the \textit{diagonal regime}, for $\x'=\x+\z, $ so that $|\x'-\x|\le t^{\frac 1\alpha}$, taking $\beta=1 $. \nero{ We get, arguing as before (recall that $\alpha \in (0,1)$), writing $\nu_S(d\z)=\frac{dr}{r^{1+\alpha}}\tilde \mu_S(d\bxi) $:
\begin{eqnarray*}
|E_{i,21}(t,\x)| \le \frac{C}{t}\int_{|\z|\le t^{\frac 1\alpha}} \Big(q(t,\x+\z)+q(t,\x)\Big) \frac{|\z|}{t^{\frac 1\alpha}}\nu_S(d\z)
\\
= \frac{C}{t^2} \int_{0}^{1}\frac{ds}{s^{\alpha}}\int_{\S^{N-1}} \big( q(t,\x+ s t^{\frac 1\alpha} \bxi)+  q(t,\x) \big) 
\tilde \mu_S(d\bxi),
\end{eqnarray*}
changing also variable $s=rt^{-\frac 1\alpha} $ to get the last integral. Defining the density $\tilde q (t, \cdot)$:
$$\tilde q(t, \x) := C \int_{0}^{1}\frac{ds}{s^{\alpha}}\int_{\S^{N-1}}  q(t,\x+ s t^{\frac 1\alpha} \bxi) 
\tilde \mu_S(d\bxi),\;\; t>0,\, \x \in \R^N,
$$
we note  that $\tilde q(t, \x) = t^{- \frac N\alpha} \tilde q(1,t^{-\frac 1\alpha}\,  \x )$;  we have  \eqref{INT_Q}  when $q$ is replaced by $\tilde q$. Finally 
$$
|E_{i,2}(t,\x)| \le \frac{C}{t^2}\Big(q(t,\x)+ \bar q(t,\x) + \tilde q(t,\x) \Big).
$$
Plugging this last control and }
  \eqref{NUOVA_SENSI_TEMPO_EI1} 
into \eqref{NUOVA_SENSI_TEMPO} gives the statement up to a modification of $q$.
\end{proof}
\begin{proof}[Proof of estimate \eqref{HOLDER_S}]
By independence, see equation \eqref{DECOMP_G_P}, we can write:
$$
p_S(t,\x)-p_S(t,\x') = \int_{\R^N} \Big( p_M(t,\x-\bxi) -p_M(t,\x'-\bxi) \big) P_{N_t}(d\bxi).
$$
Now, since $p_M$ is smooth, we can use Taylor formula to expand:
\begin{equation*}
\int_{\R^N} \Big(p_M(t,\x-\bxi) -p_M(t,\x'-\bxi)  \Big) P_{N_t}(d\bxi)= \int_{\R^N}  \int_0^1  \nabla_\x p_M(t,\lambda \x + (1-\lambda)\x'-\bxi)\cdot (\x- \x') d\lambda P_{N_t}(d\bxi).
\end{equation*}
Equation \eqref{CTR_DER_M} now yields that, for a given $m\in \N $, there exists a constant $C_m$ s.t. for all $t>0,\ (\x,\x')\in \R^{2N}$:
$$
\left| \nabla_\x p_M(t,\lambda \x + (1-\lambda)\x'-\bxi)  \right| \le C_m \frac{1}{t^{\frac{1}{\alpha}}} {t^{-{\frac{N}{\alpha}}}}{ \left(1+ \frac{|\lambda \x + (1-\lambda)\x'-\bxi|}{t^{\frac{1}{\alpha}}} \right)^{-m}}=\frac{\bar C_m}{t^{\frac 1\alpha}}p_{\bar M}(t,\lambda \x+(1-\lambda)\x'-\bxi).
$$
We thus derive:
\begin{equation}
\label{EQ_OK}
\nero{ \left|\int_{\R^N} \Big(p_M(t,\x-\bxi) -p_M(t,\x'-\bxi)  \Big) P_{N_t}(d\bxi)\right| \le  \int_{\R^N}  C_m    \frac{ |\x- \x'|}{t^{\frac{1}{\alpha}}}
 \int_0^1 p_{\bar M}(t,\lambda \x+(1-\lambda)\x'-\bxi)
d\lambda P_{N_t}(d\bxi).}
\end{equation}
Rewrite:
$
\lambda \x + (1-\lambda)\x'-\bxi = \lambda (\x -\bxi) + (1-\lambda)(\x'- \bxi).
$
Observe as well from \eqref{CTR_DER_M} that $p_{\bar{M}}(t,\cdot)$ is convex.
Consequently, we can bound:
$$
p_{\bar M}(t,\lambda (\x -\bxi) + (1-\lambda)(\x'- \bxi))\le \lambda  p_{\bar M}(t,\x-\bxi)+(1-\lambda) p_{\bar M}(t,\x'-\bxi) .
$$
Plugging the above control in \eqref{EQ_OK}, we obtain:
$$
\left| p_S(t,\x)-p_S(t,\x')\right| \le \frac{C_m}{2}  \frac{ |\x- \x'|}{t^{\frac{1}{\alpha}}} \nero{ \Big( p_{\bar{S}}(t,\x) +p_{\bar{S}}(t,\x') \Big), }
$$
recalling that we defined $p_{\bar{S}}(t,\cdot)$ as the convolution between $p_{\bar{M}}(t,\cdot)$ and the law of $N_t$ (see \eqref{we}). The above control readily gives \eqref{HOLDER_S} if $|\x-\x'|\le t^{\frac1\alpha} $. Indeed, in that case $\frac{|\x-\x'|}{t^{\frac 1 \alpha}}\le  \Big(\frac{|\x-\x'|}{t^{\frac 1 \alpha}}\Big)^\beta$ for all $\beta \in (0,1] $. On the other hand, if $|\x-\x'|>t^{\frac 1\alpha} $, we can again derive from 
 \eqref{CTR_DER_M}  \nero{ with $\i = {\bf 0}$}
\begin{equation*}
|p_S(t,\x)-p_S(t,\x')|\le p_S(t,\x)+p_S(t,\x')\le C\big(p_{\bar S}(t,\x)+p_{\bar S}(t,\x')\big) \le C\Big(\frac{|\x-\x'|}{t^{\frac1\alpha}}\Big)^\beta\big(p_{\bar S}(t,\x)+p_{\bar S}(t,\x')\big),
\end{equation*}
for all $\beta \in (0,1] $. This completes the proof of \eqref{HOLDER_S}.
\end{proof}

%
%

We now state a crucial result to deal with the estimation of the singularities  in \eqref{SPLITTING_TIME_SPACE} (see also Remark \ref{NOTATION_Q}).
\begin{lemma}[Integration of the singularities]
\label{LABEL_LEMME_A_FAIRE}
For all $\delta>0, \kappa\ge 0$ sufficiently small, there exists $C:=C(\A{A},\delta,\kappa)>0$ such that for all $\gamma>0$ and given $(t,\sigma)\in[-T,T]^2$, with $|t-\sigma|\le \gamma^\alpha$,
defining for $s \ge t \vee \sigma$:
$$
\rho_{s}(\z) := |s-t|^{\frac 1 \alpha} + \sum_{i=1}^{n} |\z_i|^{\frac{1}{1+\alpha(i-1)}} |u(s)|^{\frac1 \alpha},
$$
where  $u(s) = u_{\lambda}(s)= \lambda (s-t) + (1-\lambda)(s-\sigma)$ \nero{for any fixed}
 $\lambda\in [0,1]$, we have for $K$ large enough:
\begin{equation}\label{INTEGRABILITE_Q}
\gamma^{\delta \alpha} \int_{s\ge t\vee \sigma , \,  \rho_{s}(\z) \ge K\gamma} \frac{1}{|u(s)|^{1+\delta}} |\z|^\kappa q(\z) d\z ds\le C.
\end{equation}
\end{lemma}
\begin{proof}
Let us define:
\begin{eqnarray}
I=\gamma^{\delta\alpha} \int_{ s\ge t\vee \sigma  , \, \rho_s(\z) \ge K\gamma} \frac{1}{|u(s)|^{1+\delta}} |\z|^\kappa q(\z) d\z ds.
\end{eqnarray}
We introduce the following partition \nero{for a fixed $c_0>0$}:
\begin{equation*}
I= \gamma^{\delta\alpha} \int_{s\ge t\vee \sigma  , \,  \rho_s(\z) \ge K\gamma, \, \{|u(s)|>c_0\gamma^{\alpha}\}\cup \{|u(s)|\le c_0\gamma^{\alpha}\}} \frac{1}{|u(s)|^{1+\delta}} |\z|^\kappa q(\z) d\z ds.
=:I_{1}+I_2.
\end{equation*}
For $I_1$ we readily derive, integrating the function $|\z|^\kappa q(\z)$ in space, that:
\begin{equation}
\label{CTR_IS11}
I_1\le \gamma^{\delta\alpha} \int_{ \rho_s(\z) \ge K\gamma   , \,  |u(s)|>c_0\gamma^{\alpha} } \frac{1}{|u(s)|^{1+\delta}} |\z|^\kappa q(\z) d\z ds \le C\gamma^{\delta\alpha}
\left\{-\frac{1} {\nero{r}^{ \delta}}
\Big|_{c_0\gamma^\alpha}^{+\infty}\right\}\le C.
\end{equation}

We now turn to $I_2$.
Observe that with the definition $\rho_s(\z)$,
when $|u(s)|\le c_0\gamma^{\alpha}$ and $\rho_s(\z) \ge K \gamma$, since $|t-\sigma| \le \gamma^{\alpha}$, we have $|s-t| \le |u(s)|+|t-\sigma|\le (c_0+1)\gamma^\alpha $ \nero{(consider the cases $t> \sigma$ and $t \le \sigma$)} and

$$
\sum_{i=1}^{n} |\z_i|^{\frac{1}{1+\alpha(i-1)}} |u(s)|^{\frac{1}{\alpha}} \ge \big(K-(1+c_0)^{\frac 1 \alpha} \big) \gamma = :\tilde K \gamma ,
$$
where $\tilde K >0$ when $K$ is large enough. Hence,
$$\Big\{\rho_s(\z) >K \gamma,|u(s)|\le c_0\gamma^\alpha \Big\}\subset \Big\{ \sum_{i=1}^{n} |\z_i|^{\frac{1}{1+\alpha(i-1)}} |u(s)|^{\frac{1}{\alpha}}\ge \tilde K\gamma, |u(s)|\le c_0\gamma^\alpha \Big\}.$$
We then get:
\begin{equation}
\label{READY_FOR_CTR}
I_2\le  \gamma^{\delta\alpha} \int_{\sum_{i=1}^{n} |\z_i|^{\frac{1}{1+\alpha(i-1)}} |u(s)|^{\frac{1}{\alpha}}\ge \tilde K\gamma   , \,   |u(s)|\le c_0\gamma^{\alpha} }
\frac{1}{|u(s)|^{1+\delta} }
|\z|^\kappa q(\z)d\z ds.
\end{equation}
Thus, for all $s, \z $, there exists $i_0$ such that:
$$
|\z_{i_0}| \ge \left(\frac{\tilde{K}}{n} \right) ^{\alpha(i_0 -1)+1} \frac{\gamma^{1+ \alpha(i_0-1)}}{|u(s)|^{i_0-1 + \frac 1 \alpha}}.
$$
Consequently for $ \eta_j \in (0 ,\alpha)$, we derive: 
\begin{eqnarray*}
I_2 
&\le & C\gamma^{\delta\alpha} \int_{\sum_{i=1}^{n} |\z_i|^{\frac{1}{1+\alpha(i-1)}} |u(s)|^{\frac{1}{\alpha}}\ge \tilde K\gamma   , \,   |u(s)|\le c_0\gamma^{\alpha}}
\frac{1}{|u(s)|^{1+\delta} } \sum_{j=1}^{n}\left( \frac{|\z_{j}|  |u(s)|^{j-1+ \frac1 \alpha}}{\gamma^{1+\alpha (j-1)} } \right)^{\eta_j} |\z|^\kappa q(\z)d\z ds\\
&\le& C \sum_{j=1}^{n}\gamma^{\delta\alpha- \eta_j(1+\alpha (j-1))} \int_{  |u(s)|\le c_0\gamma^{\alpha}}
\frac{1}{|u(s)|^{1+\delta-\eta_j(j-1+ \frac1 \alpha)  }} ds  \int_{\R^N} |\z|^{\eta_j+\kappa} q(\z)d\z.
\end{eqnarray*}
Choosing for all $j\in \leftB 1,n\rightB$, $\eta_j+\kappa<\alpha $, we can integrate in space, i.e. $\int |\z|^{\eta_j+\kappa} q(\z)d\z \le C$.
Thus,
$$
I_2 \le  C\sum_{j=1}^{n}\gamma^{\delta\alpha- \eta_j(1+\alpha (j-1))} \int_{  |u(s)|\le c_0\gamma^{\alpha}}
\frac{1}{|u(s)|^{1+\delta-\eta_j(j-1+ \frac1 \alpha)  }} ds \le C,
$$
where for the last inequality, we choose for all $j\in \leftB1,n\rightB$,  $\eta_j=\eta$ such that:
$$
\eta > \delta \alpha \ge \frac{\alpha \delta}{\alpha(j-1)+1}    \Rightarrow   \delta- \eta_j(j-1+ \frac1 \alpha) <0  .
$$
We point out that the constraints on $\kappa$, $\delta$ and $\eta$  summarize in
$\kappa +\eta <\alpha$ and $\alpha \delta < \eta $, and can be fulfilled for $\kappa$ and $\delta$ small enough.
\end{proof}

\begin{remark}\label{DELTA_0}
Importantly, it can be derived from the previous proof that the term $I_2\le C$ even if $\delta=0 $.
Indeed, in this case, we handle:
\begin{equation}\label{EST_DELTA_0}
I_2 \le  C\sum_{j=1}^{n}\gamma^{- \eta_j(1+\alpha (j-1))} \int_{  |u(s)|\le c_0\gamma^{\alpha}}
\frac{1}{|u(s)|^{1-\eta_j(j-1+ \frac1 \alpha)  }} ds \le C.
\end{equation}
\end{remark}

\subsubsection{Proof of the deviation Lemma \ref{LEMME_DEV}: Boundedness of the terms in \eqref{SPLITTING_TIME_SPACE}} 
We are now in position to complete the proof of Lemma \ref{LEMME_DEV}. It suffices to establish that
 the terms \nero{ $I_{i,T} $ and $I_{i,S}$} in \eqref{SPLITTING_TIME_SPACE}, which respectively correspond to the \textit{time} and \textit{space} sensitivities, are bounded.

 \textcolor{black}{
 The purpose of the computations is then to derive that the initial integration domain $\{\rho>K\gamma\} $ can be expressed as or is included in a domain of the form $\{\rho_s(\z)>K\gamma\} $, for a possibly different $K$, using the notations introduced in Lemma \ref{LABEL_LEMME_A_FAIRE}. 
 This latter lemma is here the crucial tool to handle the singularities.}
\begin{trivlist}
\item[$\bullet $] \textit{Control of $(I_{i,S})_{i\in \leftB 1,n\rightB}$ in \eqref{SPLITTING_TIME_SPACE}}. From Lemma \ref{SENS_SING_STAB} and the notations introduced in \eqref{DEF_RHO_GAMMA}:
\begin{eqnarray} \label{sii} \nonumber
I_{i,S}&\le& \nero{ I_S := } C \int_{s\ge \sigma\vee t   , \,  \rho >K\gamma} \frac{1}{(s-\sigma)}\left(\frac{|\M_{s-\sigma}^{-1}(e^{(s-t)A}\x-e^{(s-\sigma)}\bxi)|}{(s-\sigma)^{\frac 1 \alpha}} \right)^\beta \frac{1}{\det(\M_{s-\sigma})}\Big(q\big(s-\sigma, \M_{s-\sigma}^{-1}(e^{(s-t)A}\x-\y) \big)\\   
&&\ +q\big(s-\sigma, \M_{s-\sigma}^{-1}(e^{(s-\sigma)A}\bxi-\y) \big) \Big)  d\y ds.
\end{eqnarray}
We can rewrite, using also \eqref{str},  
\begin{eqnarray*}
|\M_{s-\sigma}^{-1}(e^{(s-t)A}\x-e^{(s-\sigma)A}\bxi)|&=&\Big|\big(\M_{s-\sigma}^{-1}e^{(s-\sigma)A}\M_{s-\sigma}\big)
\M_{s-\sigma}^{-1}(e^{(\sigma-t)A}\x-\bxi)\Big|\\
&=&  
 \Big|e^{A} \,
\M_{s-\sigma}^{-1}(e^{(\sigma-t)A}\x-\bxi)\Big| \le 
C|\M_{s-\sigma}^{-1}(e^{(\sigma-t)A}\x-\bxi)|, 
\end{eqnarray*}
so that recalling from \eqref{DEF_RHO_GAMMA}, $\rho:=\rho(s-t, e^{(s-t)A}\x-\y) $ and $\gamma:=\rho(\sigma-t,e^{(\sigma-t)A}\x-\bxi)$:
\begin{eqnarray}
I_{S}
&\le &C\int_{s\ge \sigma\vee t   , \,  \rho >K\gamma} \sum_{i=1}^n
\left(\frac{|(e^{(\sigma-t)A}\x-\bxi)_i|}{(s-\sigma)^{i-1+\frac{1}{\alpha}}}\right)^{\beta}\frac{1}{(s-\sigma)}
\nonumber\\
&&\times\frac{1}{\det(\M_{s-\sigma})}\Big(q\big(s-\sigma, \M_{s-\sigma}^{-1}(e^{(s-t)A}\x-\y) \big)
+q\big(s-\sigma, \M_{s-\sigma}^{-1}(e^{(s-\sigma)A}\bxi-\y) \big) \Big) d\y ds\nonumber\\
&\le & C\sum_{i=1}^n \gamma^{\big( \alpha(i-1)+1\big) \beta} \int_{s\ge \sigma\vee t}\frac{1}{(s-\sigma)^{\beta(i-1+\frac{1}{\alpha})+1} } \int_{\R^N}   \I_{\rho >K\gamma}
\nonumber\\
&&\times\frac{1}{\det(\M_{s-\sigma})}\Big(q\big(s-\sigma, \M_{s-\sigma}^{-1}(e^{(s-t)A}\x-\y) \big)
\ +q\big(s-\sigma, \M_{s-\sigma}^{-1}(e^{(s-\sigma)A}\bxi-\y) \big) \Big)  d\y ds =:I_S^1+I_S^2. \label{DECOUP_IS}
\end{eqnarray}
Let us first deal with $I_S^1$ and set  for a fixed $s $,
\textcolor{black}{
\begin{equation}
\label{CAMBIO_VAR_IMP}
\z_1 
=(s-\sigma)^{-\frac1\alpha}\M_{s-\sigma}^{-1}(e^{(s-t)A}\x-\y).
\end{equation}
Observe that, in the variable $\z_1$, 
we have for all $i\in \leftB 1,n \rightB $:
$$(e^{(s-t)A}\x-\y)_i=\z_{1,i} (s-\sigma)^{\frac 1\alpha+(i-1)}=\z_{1,i} (s-\sigma)^{\frac {1+\alpha(i-1)}\alpha}\Rightarrow |(e^{(s-t)A}\x-\y)_i|^{\frac{1}{\alpha(i-1)+1}}=|\z_{1,i}|^{\frac{1}{\alpha(i-1)+1}}(s-\sigma)^{\frac 1\alpha}.$$
In other words, the component $(s-\sigma)^{\frac 1\alpha} $ factorizes by homogeneity for all components (see also Remark \ref{REM_HOMO}).
From the definition of $\rho $ in \eqref{DEF_RHO_GAMMA} we thus obtain:
\begin{eqnarray*}
\rho=(s-t)^{\frac1\alpha}+\sum_{i=1}^n |(e^{(s-t)A}\x-\y)_i|^{\frac{1}{1+\alpha(i-1)}}=(s-t)^{\frac1\alpha}+\sum_{i=1}^n |\z_{1,i}|^{\frac{1}{1+\alpha(i-1)}}(s-\sigma)^{\frac 1\alpha}=:\rho_s(\z_1),
\end{eqnarray*}
introduced in Lemma \ref{LABEL_LEMME_A_FAIRE} taking $u(s)=(s-\sigma)$.}
Recalling the  scaling property $q(t,\x)=\frac{1}{t^{\frac N\alpha}}q\big(1,\frac{\x}{t^{\frac1\alpha}}\big) $ and using the previous change of variable in the spatial integral,  we get
\begin{eqnarray*}
I_S^1\le C\sum_{i=1}^n \gamma^{\big( \alpha(i-1)+1\big) \beta} \int_{s\ge \sigma\vee t   , \,  \rho_s(\z_1) >K\gamma}
\frac{1}{(s-\sigma)^{\beta(i-1+\frac{1}{\alpha})+1} } q(\z_1)d\z_1 ds.
\end{eqnarray*}
Each term in the above summation can thus be controlled thanks to Lemma \ref{LABEL_LEMME_A_FAIRE} 
 taking for the $i^{{\rm th}} $ term, $ \delta= \beta(i-1+\frac{1}{\alpha}), \kappa=0$. Let us now control $I_S^2$ in \eqref{DECOUP_IS}.
\begin{eqnarray*}
I_S^2&=& C\sum_{i=1}^n \gamma^{\big( \alpha(i-1)+1\big) \beta} \int_{s\ge \sigma\vee t ,  \{|s-\sigma|> \gamma^\alpha  \}\cup\{|s-\sigma|\le \gamma^\alpha\}}
\frac{1}{(s-\sigma)^{\beta(i-1+\frac{1}{\alpha})+1} }\\
&&\times \Big(\int_{\R^N} \I_{   
\rho >K\gamma} \frac{1}{\det(\M_{s-\sigma})}
q\big(s-\sigma, \M_{s-\sigma}^{-1}(e^{(s-\sigma)A}\bxi-\y) \big)  d\y \Big)ds=:I_S^{21}+I_S^{22}.
\end{eqnarray*}
For $I_S^{21}$ we readily get the bound writing $\I_{   
\rho >K\gamma}\le 1 $ and integrating in $\y$ over $\R^N$ and in $s$ over $\{|s-\sigma|> \gamma^\alpha  \}$.
To analyze $I_S^{22} $, we first set $\z_2=(s-\sigma)^{-\frac 1 \alpha}\M_{s-\sigma}^{-1}(e^{(s-\sigma)A}\bxi-\y)$.
Recall as well from \eqref{DEF_RHO_GAMMA} that
\begin{eqnarray}
\label{MAJO_EVE}
\rho &\le & |t-s|^{\frac 1 \alpha}+ \sum_{i=1}^n\Big( \Big|(\y-e^{(s-\sigma)A}\bxi)_i\Big|^{\frac{1}{1+\alpha(i-1)}}+\Big|\big(e^{(s-\sigma)A}(\bxi-e^{(\sigma-t)A}\x)\big)_i\Big|^{\frac{1}{1+\alpha(i-1)}} \Big)\nonumber\\
&\le& \rho_s(\z_2)+\sum_{i=1}^n\Big|\big(e^{(s-\sigma)A}(\bxi-e^{(\sigma-t)A}\x)\big)_i\Big|^{\frac{1}{1+\alpha(i-1)}},
\end{eqnarray}
using again the notations of Lemma \ref{LABEL_LEMME_A_FAIRE} with $u(s)=(s-\sigma) $ for the last inequality.
 \nero{Observe now that arguing as in  \eqref{chi} we have, for any $\q \in \R^N$, $r>0$, $i\in \leftB 1,n\rightB $,
\begin{align} \label{FORM_RES}  \nonumber
 |(e^{r A}\q)_i| =  |B_i^* e^{rA} [(\q^*)^*] |
 =  | (\q^* e^{rA^*} B_i)^*| = 
 | \q^* e^{rA^*} B_i| 
 =  | \q^* \M_r^{-1} e^{A^*} \M_r  B_i | 
\\ = r^{(i-1)} | \q^* \M_r^{-1} e^{A^*}  B_i 
| 
=  r^{(i-1)} | \,[\q_1,  r^{-1} \q_2, \ldots , r^{-(n-1)}\q_n]^* \, e^{A^*} B_i |  
 = c \sum_{k =1}^i r^{i-k} |\q_k|.  
\end{align}
By  \eqref{FORM_RES}, } we get  for all $i\in \leftB 1,n\rightB $:
\begin{eqnarray*}
\Big|\big(e^{(s-\sigma)A}(\bxi-e^{(\sigma-t)A}\x)\big)_i\Big|&\le& C\sum_{j=1}^i |s-\sigma|^{i-j}|(\bxi-e^{(\sigma-t)A}\x)_j| 
\le
C\sum_{j=1}^i \gamma^{\alpha(i-j)} \gamma^{\alpha(j-1)+1} \le C \gamma^{\alpha(i-1)+1},
\end{eqnarray*}
recalling that $|s-\sigma|\le \gamma^\alpha $ for $I_S^{22} $ and $\gamma=\rho(|t-\sigma|,\bxi-e^{(\sigma-t)A}\x ) $, with $\rho(\cdot,\cdot)$ defined in \eqref{rho_homo}, for the last inequality. We thus get from \eqref{MAJO_EVE}:
\begin{equation*}
\rho\le \rho_s(\z_2)+C\gamma.
\end{equation*}
Hence,
$\{\rho >K\gamma, |s-\sigma|\le \gamma^\alpha\} \subset \{ \rho_s(\z_2)>(K-C)\gamma\},$
so that
\begin{eqnarray*}
I_S^{22}\le C\sum_{i=1}^n \gamma^{\big( \alpha(i-1)+1\big) \beta} \int_{s\ge \sigma\vee t, \, \rho_s(\z_2)>(K-C)\gamma }
\frac{1}{(s-\sigma)^{\beta(i-1+\frac{1}{\alpha})+1} } q(\z_2)d\z_2 ds
\end{eqnarray*}
which is again controlled by Lemma \ref{LABEL_LEMME_A_FAIRE} taking the same arguments as for $I_S^1$ for $K$ \textit{large enough}. The statement for $I_S$ follows from \eqref{DECOUP_IS} and the previous controls.
\\
\item[$\bullet $ ] \textit{Control of $(I_{i,T})_{i\in \leftB 1,n\rightB}$ in \eqref{SPLITTING_TIME_SPACE}}. For the analysis,
we need to exploit the relative position of $t,\sigma \le s $. We can assume w.l.o.g. that $t<\sigma\le s $. Note that, if $t=\sigma $, then for all $i\in \leftB 1,n\rightB $, $I_{i,T}=0$ (no time sensitivity).
Set for a fixed $s$,
\begin{equation}
\label{Z_CTR_T}
\z_3=\M_{s-t}^{-1} (e^{(s-t)A}\x-\y).
\end{equation} 
Recall as well  from \eqref{DEF_RHO_GAMMA} that:
\begin{equation}
\rho=(s-t)^{\frac 1\alpha}+\sum_{i=1}^{n}|\big(e^{(s-t)A}\x -\y\big)_i|^{\frac{1}{\alpha(i-1)+1}}
=(s-t)^{\frac 1\alpha}+\sum_{i=1}^{n}|(s-t)^{i-1}\big(\z_3\big)_i|^{\frac{1}{\alpha(i-1)+1}}=\rho_s\Big(\frac{\z_3}{(s-t)^{\frac1\alpha}}\Big),\label{CTR_CAMBIO_RHO}
\end{equation}
with the notation of Lemma \ref{LABEL_LEMME_A_FAIRE} with $u(s)=s-t $.
We now split for $i\in \leftB 1,n\rightB $ the terms in the following way:
\begin{eqnarray*}
I_{i,T}\\
=
\int_{s\ge \sigma\vee t   , \,  \rho >K\gamma} \Big| \frac{\Delta^{\frac {\alpha_i} 2,A,i,s-t}p_S\big(s-t,\M_{s-t}^{-1}(e^{(s-t)A}\x-\y)\big)}{\det(\M_{s-t})}
- \frac{\Delta^{\frac {\alpha_i} 2,A,i,s-\sigma}p_S\big(s-\sigma,\M_{s-\sigma}^{-1}(e^{(s-t)A}\x-\y)\big)}{\det(\M_{s-\sigma})}\Big|d \y ds\\
=
\int_{s\ge \sigma\vee t  , \,  \rho_s\Big(\frac{\z_3}{(s-t)^{\frac 1\alpha}}\Big) >K\gamma} \Big|\Delta^{\frac{\alpha_i} 2,A,i,s-t} p_S(s-t,\z_3)- \frac{\det(\M_{s-t})}{\det(\M_{s-\sigma})}\Delta^{\frac {\alpha_i} 2,A,i,s-\sigma}p_S(s-\sigma,\M_{s-\sigma}^{-1}\M_{s-t} \z_3)\Big|d \z_3 ds\\
\le
\int_{s\ge \sigma\vee t  , \,  \rho_s\Big(\frac{\z_3}{(s-t)^{\frac 1\alpha}}\Big) >K\gamma} \Big|\Delta^{\frac {\alpha_i} 2,A,i,s-t} p_S(s-t,  \z_3)-\Delta^{\frac {\alpha_i} 2,A,i,s-\sigma} p_S(s-\sigma,  \z_3)\Big|d \z_3 ds\\
+\int_{s\ge \sigma\vee t   , \,  \rho_s\Big(\frac{\z_3}{(s-t)^{\frac 1\alpha}}\Big) >K\gamma} \Big|\Delta^{\frac {\alpha_i} 2,A,i,s-\sigma} p_S(s-\sigma,  \z_3)-\Delta^{\frac{\alpha_i} 2,A,i,s-\sigma}p_S(s-\sigma,\M_{\frac{s-t}{s-\sigma}} \z_3)\Big|d \z_3 ds\\
+\int_{s\ge \sigma\vee t  , \,  \rho_s\Big(\frac{\z_3}{(s-t)^{\frac 1\alpha}}\Big) >K\gamma} \Big|\Big(1-\det\M_{\frac{s-t}{s-\sigma}}\Big)\Delta^{\frac{\alpha_i} 2,A,i,s-\sigma}p_S(s-\sigma,\M_{\frac{s-t}{s-\sigma}} \z_3)\Big|d \z_3 ds=:I_{i,T}^1+I_{i,T}^2+I_{i,T}^3.
\end{eqnarray*}
\end{trivlist}
\begin{trivlist}
\item[-] \textit{Control of $I_{i,T}^3 $}.
\end{trivlist}
Let us first handle the contribution $I_{i,T}^3 $ which is the most \textit{unusual} and which  specifically  appears in the \textit{degenerate} framework. From equation \eqref{CTR_BOUND_SING_KER} in Lemma  \ref{SENS_SING_STAB},
we have for all $i\in \leftB 1,n\rightB $:
\begin{eqnarray*}
I_{i,T}^3\le C\int_{s\ge \sigma\vee t   , \,  \rho_s\Big(\frac{\z_3}{(s-t)^{\frac 1\alpha}}\Big)>K\gamma} \frac{1}{(s-\sigma)}q(s-\sigma, \M_{\frac{s-t}{s-\sigma}} \z_3)\Big|1-\big(\frac{s-t}{s-\sigma}\big)^r \Big| d\z_3 ds=:I_T^3,
\end{eqnarray*}
where $r:=\sum_{j=1}^n (j-1) d_j = {N - d_1}$. Observe now that:
$$\Big|1-\big(\frac{s-t}{s-\sigma}\big)^r\Big|=\frac{r(\sigma-t)\int_0^1(s-\sigma+\lambda(\sigma-t))^{r-1} d\lambda}{(s-\sigma)^ r}\le r\frac{(\sigma-t)(s-t)^{r-1}}{(s-\sigma)^r},$$
recalling that for all $\lambda \in [0,1],\ s-\sigma+\lambda(\sigma-t)\le s-t$ for the last inequality. Set now from \eqref{Z_CTR_T},
$$\hat \z_3:=(s-\sigma)^{-\frac 1\alpha}\M_{\frac{s-t}{s-\sigma}} \z_3 =(s-\sigma)^{-\frac 1 \alpha}\M_{s-\sigma}^{-1}\big(e^{(s-t)A}\x -\y\big). $$ Recalling \eqref{CTR_CAMBIO_RHO}, we obtain with the notation of Lemma \ref{LABEL_LEMME_A_FAIRE} and $u(s)=s-\sigma$:
$$\rho=(s-t)^{\frac 1\alpha}+\sum_{i=1}^{n}|\big(e^{(s-t)A}\x -\y\big)_i|^{\frac{1}{\alpha(i-1)+1}} =(s-t)^{\frac 1\alpha}+\sum_{i=1}^{n}|\hat \z_{3,i}|^{\frac{1}{\alpha(i-1)+1}}(s-\sigma)^{\frac 1 \alpha}=:\rho_s(\hat \z_3).$$
Hence, since $d\z_3=(s-\sigma)^{\frac N\alpha}\left(\frac{s-\sigma}{s-t} \right)^r d\hat\z_3 $,
\begin{eqnarray*}
I_T^3&\le& C\int_{s\ge \sigma\vee t   , \,  \rho_s(\hat \z_3)>K\gamma} \frac{1}{(s-\sigma)} q(\hat \z_3)\big(\frac{s-\sigma}{s-t} \big)^r \frac{(\sigma-t)(s-t)^{r-1}}{(s-\sigma)^r}  d\hat \z_3 ds\\
&\le& C\int_{s\ge \sigma\vee t  , \,  \rho_s(\hat \z_3)>K\gamma} \frac{1}{(s-\sigma)} q(\hat \z_3)\frac {\sigma-t}{s-t}   d\hat \z_3 ds.
\end{eqnarray*}
Our previous choice $t<\sigma\le s $ then yields $ (\sigma-t)\le (s-t),\ s-\sigma\le (s-t)$ and therefore for some $\delta\in (0, \alpha) $:
\begin{eqnarray*}
I_T^3\le C(\sigma-t)^{\delta}\int_{s\ge \sigma\vee t   , \, \rho_s(\hat \z_3)>K\gamma}\frac{1}{(s-\sigma)(s-t)^\delta} q(\hat \z_3) d\hat \z_3 ds \le C(\sigma-t)^{\delta} \! \!\int_{s\ge \sigma\vee t   , \,  \rho_s(\hat \z_3)>K\gamma}\frac{1}{(s-\sigma)^{1+\delta}} q(\hat \z_3) d\hat \z ds.
\end{eqnarray*}
Hence, Lemma \ref{LABEL_LEMME_A_FAIRE} applied with the current $\delta, \kappa=0$ yields $I_T^3\le C $.
\begin{trivlist}
\item[-] \textit{Control of $I_{i,T}^2$}.
\end{trivlist}
Using Lemma \ref{SENS_SING_STAB}, we bound:
\begin{eqnarray} \label{DESCRIZIONE_I2} 
I_{i,T}^2 = \int_{s\ge \sigma\vee t   , \,  \rho_s\Big(\frac{\z_3}{(s-t)^{\frac 1\alpha}}\Big) >K\gamma} \Big|\Delta^{\frac{\alpha_i} 2,A,i,s-\sigma} p_S(s-\sigma,  \z_3)-\Delta^{\frac {\alpha_i} 2,A,i,s-\sigma}p_S(s-\sigma,\M_{\frac{s-t}{s-\sigma}} \z_3)\Big|d \z_3 ds
\\
\nonumber
\le C \int_{s\ge \sigma\vee t  , \, \rho_s\Big(\frac{\z_3}{(s-t)^{\frac 1\alpha}}\Big) >K\gamma} \frac{1}{(s-\sigma )^{1+ \frac{\beta}{\alpha}}} \|I_{N\times N} - \M_{\frac{s-t}{s-\sigma}} \|^\beta | \z_3|^\beta \Big(q(s-\sigma, \z_3) +  q(s-\sigma , \M_{\frac{s-t}{s-\sigma}} \z_3)\Big)
 d \z_3 ds=:I_T^2,
\end{eqnarray}
  where $\|I_{N\times N} - \M_{\frac{s-t}{s-\sigma}} \|$ indicates the operator norm of $I_{N\times N} - \M_{\frac{s-t}{s-\sigma}}$.
For the first term in the second line of \eqref{DESCRIZIONE_I2}, we set $\bar{\z}_3 =(s-\sigma)^{-\frac 1 \alpha}  \z_3$ for $\z_3 $ as in \eqref{Z_CTR_T}. In this variable we have from \eqref{CTR_CAMBIO_RHO} that
\begin{equation}
\label{DEF_RHO_BAR}
\rho_s\Big(\frac{ \z_3}{(s-t)^{\frac 1\alpha}}\Big)=(s-t)^{\frac 1\alpha}+\sum_{i=1}^n |\bar \z_{3,i} (s-\sigma)^{\frac 1\alpha}(s-t)^{i-1}|^{\frac{1}{1+\alpha(i-1)}}=:\bar \rho_s (\bar \z_3).
\end{equation}
Changing variable in the second term of the second line of \eqref{DESCRIZIONE_I2} to $\hat{\z}_3 =(s-\sigma)^{-\frac 1 \alpha} \M_{\frac{s-t}{s-\sigma}} \z_3$, as in $I_{T,3} $, we recall from \eqref{CTR_CAMBIO_RHO} that (note that $(s- \sigma)^{\frac{\alpha (i-1) + 1 }{\alpha (\alpha (i-1) + 1 ) } } = (s- \sigma)^{\frac{1}{\alpha}} $)
\begin{equation*}
\rho_s\Big(\frac{ \z_3}{(s-t)^{\frac 1\alpha}}\Big)=(s-t)^{\frac 1\alpha}+\sum_{i=1}^n |\hat \z_{3,i} |^{\frac{1}{1+\alpha(i-1)}}(s-\sigma)^{\frac 1\alpha} =\rho_s (\hat \z_3),
\end{equation*}
with the notations of Lemma \ref{LABEL_LEMME_A_FAIRE} with $u(s)=(s-\sigma) $.
Recalling now that $s-t \ge s-\sigma$, so that $\det (\M_{\frac{s-\sigma}{s-t}}) \le 1$ and $\big \| \M_{\frac{s-\sigma}{s-t}} \big\| \le c$, we derive:
\begin{eqnarray*}
I_T^2
&\le&C \Big(\int_{s\ge \sigma\vee t  , \, \bar \rho_s(\bar \z_3) >K\gamma} \frac{1}{(s-\sigma )} \|I_{N\times N} - \M_{\frac{s-t}{s-\sigma}} \|^\beta |\bar \z_3|^\beta q( \bar \z_3) d \bar\z_3 ds\\
&&+ \int_{s\ge \sigma\vee t   , \,  \rho_s(\hat \z_3) >K\gamma} \frac{1}{(s-\sigma )} \|I_{N\times N} - \M_{\frac{s-t}{s-\sigma}} \|^\beta |\hat \z_3|^\beta q( \hat \z_3) d \hat\z_3 ds\Big).
\end{eqnarray*}
Observe now that
$$\|I_{N\times N} - \M_{\frac{s-t}{s-\sigma}} \| \le C \left| 1- \Big( \frac{s-t}{s-\sigma} \Big)^{n-1}\right| \le C \frac{|t-\sigma| |s-t|^{n-2} }{|s-\sigma|^{n-1}}.
$$
We thus get
\begin{eqnarray*}
I_T^2
 \le C\Big( \int_{s\ge \sigma\vee t  , \, \bar \rho_s(\bar \z_3) >K\gamma} \frac{1}{(s-\sigma )}
\left(\frac{|t-\sigma| |s-t|^{n-2} }{|s-\sigma|^{n-1}}\right)^\beta
|\bar\z_3|^\beta q(\bar\z_3) d\bar \z_3 ds\\
+\int_{s\ge \sigma\vee t  , \, \rho_s(\hat \z_3) >K\gamma} \frac{1}{(s-\sigma )}
\left(\frac{|t-\sigma| |s-t|^{n-2} }{|s-\sigma|^{n-1}}\right)^\beta
|\hat\z_3|^\beta q(\hat\z_3) d\hat \z_3 ds \Big)
=:I_T^{21}+I_T^{22}.
\end{eqnarray*}
Let us first deal with $I_T^{22} $ which already has the good form to apply Lemma \ref{LABEL_LEMME_A_FAIRE} since it involves $\rho_s(\hat \z_3) $ and not $\bar \rho_s(\bar \z_3)$. Precisely,
\begin{eqnarray*}
I_T^{22}&
\le
&
C \int_{s\ge \sigma\vee t   , \, \rho_s(\hat \z_3) >K\gamma   , \,  \{s-\sigma \le \frac12( s-t)\}   \cup  \{s-\sigma > \frac12( s-t)\}  } \frac{1}{(s-\sigma ) }
\left(\frac{|t-\sigma| |s-t|^{n-2} }{|s-\sigma|^{n-1}}\right)^\beta
|\hat\z_3|^\beta q( \hat\z_3) d \hat\z_3 ds\\
&
=:
& 
I_T^{221}+ I_T^{222}. 
\end{eqnarray*}
For the second contribution, since $ s-\sigma > \frac{1}{2} (s-t)$ we can bound the ratio: $\frac{(s-t)^{n-2}}{(s-\sigma)^{n-1}}\le C \frac{1}{s-\sigma}$ so that:
\begin{eqnarray*}	
 I_T^{222} &\le&  \int_{s\ge \sigma\vee t   , \,  \rho_s(\hat \z_3) >K\gamma   , \, s-\sigma > \frac12( s-t)  }|t-\sigma|^\beta   \frac{1}{(s-\sigma )^{\beta+1} }
|\hat\z_3|^\beta q(\hat \z_3) d\hat \z_3 ds\\
&\le& \gamma^{\alpha\beta}\int_{s\ge \sigma\vee t   , \,  \rho_s(\hat \z_3) >K\gamma  }   \frac{1}{(s-\sigma )^{\beta+1} }
|\hat\z_3|^\beta q(\hat \z_3) d \hat\z_3 ds.
\end{eqnarray*}
We conclude $I_T^{222} \le C$ using Lemma \ref{LABEL_LEMME_A_FAIRE} with
$\delta = \kappa= \beta$ (choosing $\beta$ small enough and $K$ large enough).
We now turn to $I_T^{221}$.
We can bound
\begin{equation}
\label{STIMA_SOTTILE}
\left(\frac{|t-\sigma| |s-t|^{n-2} }{|s-\sigma|^{n-1}}\right)^\beta \le \left( \frac{|t-\sigma|}{|s-t|} \right)^{\beta} \frac{|s-t|^{(n-1)\beta}}{(s-\sigma)^{(n-1)\beta}}  \le C\frac{\gamma^{\alpha(n-1)\beta}}{(s-\sigma)^{(n-1)\beta}},
\end{equation}
recalling that since $s-\sigma\le \frac 12 (s-t) $ then $s-t\le 2(\sigma-t)\le 2\gamma^\alpha $ (moreover, $|t- \sigma| \le |s-t|$).
Thus, we obtain:
\begin{eqnarray*}	
 I_T^{221} &\le&\gamma^{\alpha(n-1)\beta}  \int_{s\ge \sigma\vee t   , \, \rho_s(\hat \z_3) >K\gamma   , \, s-\sigma \le \frac12( s-t)  }
\frac{1}{(s-\sigma)^{1+(n-1)\beta}}
|\hat\z_3|^\beta q( \hat\z_3) d \hat\z_3 ds\\
&\le&
\gamma^{\alpha(n-1)\beta}  \int_{s\ge \sigma\vee t   , \,  \rho_s(\hat \z_3) >K\gamma  }
\frac{1}{(s-\sigma)^{1+(n-1)\beta}}
|\hat\z_3|^\beta q(\hat \z_3) d \hat\z_3 ds,
\end{eqnarray*}
where we are once again in position to apply Lemma \ref{LABEL_LEMME_A_FAIRE},
with $\delta = (n-1)\beta$, $\kappa=\beta$.
The control $I_T^{22}\le C$ follows.
Let us now turn to $I_T^{21} $ and write:
\begin{eqnarray*}
I_T^{21}&\le&
C \int_{s\ge \sigma\vee t   , \,  \bar \rho_s(\bar \z_3) >K\gamma   , \,  \{s-\sigma \le \frac12( s-t)\} \cup \{s-\sigma > \frac12( s-t)\}  } \frac{1}{(s-\sigma ) }
\left(\frac{|t-\sigma| |s-t|^{n-2} }{|s-\sigma|^{n-1}}\right)^\beta
|\bar\z_3|^\beta q( \bar \z_3) d \bar\z_3 ds\\
&=:& I_T^{211}+ I_T^{212}.
\end{eqnarray*}
For $I_T^{212} $, it is easily seen from the definition of $\bar \rho_s $ in \eqref{DEF_RHO_BAR} that on the considered integration set:
$\bar \rho_s(\bar \z_3)\le 2\rho_s(\bar \z_3) $ with $u(s)=s-\sigma $ in the notation of Lemma \ref{LABEL_LEMME_A_FAIRE}. Thus,
\begin{eqnarray*}
I_T^{212}\le C \int_{s\ge \sigma\vee t   , \,  \rho_s(\bar \z_3) >\frac K2 \gamma } \frac{1}{(s-\sigma ) }
\left(\frac{|t-\sigma| |s-t|^{n-2} }{|s-\sigma|^{n-1}}\right)^\beta
|\bar\z_3|^\beta q( \bar \z_3) d \bar\z_3 ds.
\end{eqnarray*}
This term can then, up to choosing $K$ large enough, be treated as $I_T^{22} $. For $I_T^{211} $ we again use \eqref{STIMA_SOTTILE} to get:
\begin{eqnarray*}
I_T^{211}&\le&
C\gamma^{\alpha(n-1)\beta}\int_{s\ge \sigma\vee t  , \, \bar \rho_s(\bar \z_3) >K\gamma   , \,  s-\sigma \le \frac12( s-t)   } \frac{1}{(s-\sigma )^{1+(n-1)\beta} }
|\bar\z_3|^\beta q( \bar \z_3) d \bar\z_3 ds.
\end{eqnarray*}
On the considered set, since $s-t\le s-\sigma+\sigma-t \le 2 \gamma^\alpha $,  we have  
  $ \sum_{i=1}^n \Big ( |\bar \z_{3,i}| \frac{(s-\sigma)^{\frac 1\alpha}(s-t)^{i-1}}{\gamma^{1+\alpha(i-1)}} \Big)^{\frac{1}{1+\alpha(i-1)}} $ $\ge  K-2^{ 1/\alpha};$ now for
 $K$ large enough s.t. $\frac{ K-2^{ 1/\alpha}} {n} \ge 1$, we find
$$
\Big\{\bar \rho_s(\bar \z_3) >K\gamma , \, s-\sigma \le \frac12( s-t)\Big\}\subset \Big\{ \sum_{i=1}^n |\bar \z_{3,i}| \frac{(s-\sigma)^{\frac 1\alpha}(s-t)^{i-1}}{\gamma^{1+\alpha(i-1)}}\ge \frac{K- 2^{1/\alpha} }{n}, s-\sigma \le \frac 12 (s-t) \Big\} =:A(\bar \z_3).
$$
Also, again from $s-t\le 2\gamma^\alpha $,
$A(\bar \z_3)\subset \Big\{ \sum_{i=1}^n |\bar \z_{3,i}| 2^{i-1}\frac{(s-\sigma)^{ 1/\alpha}}{\gamma}\ge \frac {K- 2^{\frac 1\alpha}}n, s-\sigma \le \gamma^\alpha \Big\} = B(\bar \z_3).
$
 On $B(\bar \z_3)$ we have $\gamma \le \tilde C \, |\bar \z_3| (s-\sigma)^{\frac 1\alpha}$. Choosing $\beta \in (0,1],\theta>0 $ s.t. $(n-1)\beta-\frac\theta \alpha<0 $ and $\beta+\theta< \alpha$
 we get:
\def\stima12{
Observe now that, since $s-t\le s-\sigma+\sigma-t \le 2 \gamma^\alpha $ on the considered set, we have for $(K-2)/n\ge 1$:
$$\Big\{\bar \rho_s(\bar \z_3) >K\gamma , \{s-\sigma \le \frac12( s-t)\}\Big\}\subset\Big\{ \sum_{i=1}^n |\bar \z_{3,i}| \frac{(s-\sigma)^{\frac 1\alpha}(s-t)^{i-1}}{\gamma^{1+\alpha(i-1)}}\ge \frac{K-2}n, s-\sigma \le \frac 12 (s-t) \Big\} =:A(\bar \z_3).$$
Also, again from $s-t\le 2\gamma^\alpha $,
$$A(\bar \z_3)\subset \Big\{ \sum_{i=1}^n |\bar \z_{3,i}| 2^{i-1}\frac{(s-\sigma)^{\frac 1\alpha}}{\gamma}\ge \frac {K-2}n, s-\sigma \le \gamma^\alpha \Big\},$$
and choosing $\beta,\theta>0 $ s.t. $(n-1)\beta-\frac\theta \alpha<0 $ and $\beta+\theta< \alpha$ we get:
}
\begin{eqnarray} \label{vaii}
I_T^{211}&\le&
C\gamma^{\alpha(n-1)\beta-\theta}\int_{\{0\le s-\sigma \le \gamma^\alpha\}   } \frac{1}{(s-\sigma )^{1+(n-1)\beta-\frac\theta \alpha} }\int_{\R^N}
|\bar\z_3|^{\beta+\theta} q( \bar \z_3) d \bar\z_3 ds\le C.
\end{eqnarray}
\begin{trivlist}
\item[-] \textit{Control of $I_{i,T}^1$.}
\end{trivlist}
Finally, let us deal with $(I_{i,T}^1)_{i\in \leftB 1,n\rightB}$. Write
\begin{eqnarray*}
I_{i,T}^1&=&\int_{s\ge \sigma\vee t   , \, \rho_s\Big(\frac{\z_3}{(s-t)^{\frac 1\alpha}}\Big) >K\gamma} \Big|\Delta^{\frac{\alpha_i}2,A,i,s-t} p_S(s-t,  \z_3)-\Delta^{\frac{\alpha_i} 2,A,i,s-\sigma} p_S(s-\sigma,  \z_3)\Big|d \z_3 ds\\
&\le & \int_{s\ge \sigma\vee t  , \, \rho_s\Big(\frac{\z_3}{(s-t)^{\frac 1\alpha}}\Big) >K\gamma   , \,  \{\sigma\le \frac{s+t}{2} \} \cup\{ \sigma> \frac{s+t}{2}\}}\Big|\Delta^{\frac {\alpha_i} 2,A,i,s-t} p_S(s-t,  \z_3)-\Delta^{\frac{\alpha_i} 2,A,i,s-\sigma} p_S(s-\sigma,  \z_3)\Big|d \z_3 ds\\
&=:&I_{i,T}^{11}+I_{i,T}^{12}.
\end{eqnarray*}
From equation \eqref{CTR_DER_TEMPS_SING_KER} in Lemma \ref{SENS_SING_STAB},  introducing for all $\lambda \in [0,1],\ u_\lambda(s):=\lambda(s-t)+(1-\lambda)(s-\sigma) $,
we derive: 
\begin{eqnarray*}
I_{i,T}^{11}&=&\int_{s\ge \sigma\vee t   , \,  \rho_s\Big(\frac{\z_3}{(s-t)^{\frac 1\alpha}}\Big) >K\gamma   , \,   \sigma\le \frac{s+t}{2} }\int_0^1   \Big| \partial_t \Delta^{\frac {\alpha_i} 2,A,i,u_\lambda(s)} p_S( u_\lambda(s)
,  \z_3)\Big| |t-\sigma| \, d\lambda d \z_3 ds\\
&\le &C \int_{s\ge \sigma\vee t   , \,  \rho_s\Big(\frac{\z_3}{(s-t)^{\frac 1\alpha}}\Big) >K\gamma   , \,  \sigma\le \frac{s+t}{2} } \int_0^1  \frac{|t-\sigma|}{|
u_\lambda(s)
|^2}q(u_\lambda(s)
 ,  \z_3)\, d\lambda  d \z_3 ds=:I_{T}^{11}.
 \end{eqnarray*}
Changing  variable:  $ \tilde \z_3 = [ u_\lambda(s)
]^{-\frac 1 \alpha}\z_3 $ and using as before  $q(t,\x) = {t^{- \frac N\alpha}} \, q (1, t^{- \frac 1\alpha} \x),$  we find 
$$
I_{T}^{11} 
\le C \int_{s\ge \sigma\vee t   , \,  \rho_s\Big(\frac{\tilde \z_3 \, [ u_\lambda(s)
]^{\frac 1 \alpha}}{(s-t)^{\frac 1\alpha}}\Big) >K\gamma   , \,  \sigma\le \frac{s+t}{2} } \int_0^1  \frac{|t-\sigma|}{|
u_\lambda(s)
|^2}q(  \tilde \z_3)\, d\lambda  d \tilde \z_3 ds.
$$
Since $u_{\lambda}(s) \le s-t$, $\lambda \in [0,1]$, we have that  $\rho_s\Big(\frac{\tilde \z_3 \, [ u_\lambda(s)
]^{\frac 1 \alpha}}{(s-t)^{\frac 1\alpha}}\Big) \le 
 \rho_s (\tilde \z_3 )$
 and
we get
$$
I_{T}^{11} 
\le C \int_{s\ge \sigma\vee t   , \,  \rho_s (\tilde \z_3  ) >K\gamma   , \,  \sigma\le \frac{s+t}{2} } \int_0^1  \frac{|t-\sigma|}{|
u_\lambda(s)
|^2}q(  \tilde \z_3)\, d\lambda  d \tilde \z_3 ds.
$$
\def\stima13{
 $$
\tilde \z_3 = \big( u_\lambda(s)
\big
)^{-\frac 1 \alpha}\z_3 
 \big)^{-\frac 1 \alpha}  \M_{s-t}^{-1}( e^{(s-t)A}\x-\y),
$$
with $\z_3 $ as in \eqref{Z_CTR_T}, we deduce from \eqref{CTR_CAMBIO_RHO} that
\begin{eqnarray*}
\rho_s\Big(\frac{\z_3}{(s-t)^{\frac 1\alpha}} \Big)&=&(t-s)^{\frac 1\alpha}+\sum_{i=1}^n |\tilde \z_{3,i}|^{\frac{1}{\alpha(i-1)+1}}  (t-s)^{\frac{i-1}{\alpha(i-1)+1}}(u_\lambda(s))^{\frac{1}{\alpha(\alpha(i-1)+1)}}\\
&\le & (t-s)^{\frac 1\alpha}+\sum_{i=1}^n |\tilde \z_{3,i}|^{\frac{1}{\alpha(i-1)+1}}  (u_\lambda(s))^{\frac{1}{\alpha}}=\rho_s(\tilde \z_3),
\end{eqnarray*}
for $u(s)=u_\lambda(s) $ with the notations of Lemma \ref{LABEL_LEMME_A_FAIRE}. 
%
}
Observe now that on $\{\sigma\le \frac{s+t}{2}\} $, we have
 $(s-t) = \lambda (s-t) +(1-\lambda) (s-t)
 \le 2 u_\lambda(s) $. Hence
$$
\frac{|t-\sigma|}{|u_\lambda(s)|}=\frac{|t-\sigma|}{|
\lambda (s-t) + (1-\lambda)(s-\sigma)
|} \le 2,\;\; \lambda \in [0,1].
$$
Hence, for all $\beta \in (0,1] $,
\begin{eqnarray*}
I_T^{11}&
\le 
& 
C \int_{s\ge \sigma\vee t   , \,  \rho_s(\tilde \z_3) >K\gamma   , \,  \sigma\le \frac{s+t}{2}}\int_0^1  \frac{1}{
|u_\lambda(s)|
}\left(\frac{|t-\sigma|}{| u_\lambda(s)
|}\right)^{\beta}q(\tilde \z_3) \, d\lambda  d\tilde \z_3 ds\\
&
\le 
&
C\gamma^{\alpha\beta}\int_0^1 d\lambda \int_{s\ge \sigma\vee t  , \, \rho_s(\tilde \z_3) >K\gamma} \frac{1}{|u_\lambda(s)|^{1+\beta}}q(\tilde \z_3) \,   d\tilde \z_3 ds.
\end{eqnarray*}
We again conclude from Lemma \ref{LABEL_LEMME_A_FAIRE} taking $\delta=\beta, \kappa=0$. This gives $I_T^{11}\le C $.


For $I_{i,T}^{12}$, let us emphasize that on $\{ \sigma >\frac{s+t}{2}\} $ we have $s-\sigma\le \sigma-t\le \gamma^\alpha $. Thus, on the considered set we also have $ s-t\le 2\gamma^\alpha$. In this case, the Taylor expansion is not useful, we directly use the estimate \eqref{CTR_BOUND_SING_KER} of Lemma \ref{SENS_SING_STAB} on $\Delta^{\frac {\alpha_i} 2,A,i,.} p_S $.   We write, for all $i\in \leftB 1,n\rightB $:
\begin{eqnarray*}
I_{i,T}^{12}&\le& C\Bigg(\int_{s\ge \sigma\vee t , \, \rho_s\Big(\frac{\z_3}{(s-t)^{\frac 1 \alpha}} \Big) >K\gamma , \, \sigma> \frac{s+t}{2} } \frac{1}{s-\sigma}q(s-\sigma, \z_3) d\z_3ds\\
&&+\int_{s\ge \sigma\vee t , \, \rho_s\Big( \frac{\z_3}{(s-t)^{\frac 1\alpha}}\Big) >K\gamma , \, \sigma> \frac{s+t}{2}}\frac{1}{s-t}q(s-t, \z_3)  d\z_3 ds\Bigg)
=:I_T^{121}+I_T^{122}.
\end{eqnarray*}
For $I_T^{122} $, we change variable to $\check \z_3:=\frac{ \z_3}{(s-t)^{\frac 1 \alpha} }$; recalling Lemma \ref{LABEL_LEMME_A_FAIRE}  with $u(s)= s-t$ we have $\{ \rho_{s}(\check \z_3)>K\gamma\} $ $ =:
\big\{ (s-t)^{\frac 1 \alpha} $ $+\sum_{i=1}^n |\check \z_{3,i}|^{\frac{1}{1+\alpha(i-1)}} (s-t)^{\frac 1\alpha}$ $ >K\gamma\big\};  $ hence
\begin{align*}
I_T^{122}\le  \int_{s\ge \sigma\vee t , \,\rho_{s}(\check \z_3 ) >K\gamma , \, |s-t|\le c\gamma^\alpha} \frac{C}{s-t}q(\check \z_3) d\check\z_3 ds
\\
\le  \int_{\sum_{i=1}^n|\check \z_{3,i}|^{\frac{1}{\alpha(i-1)+1}}|s-t|^{\frac 1\alpha} >(K-c^{1/\alpha} )\gamma , \, |s-t|\le c\gamma^\alpha} \frac{C}{|s-t|}q(\check \z_3) d\check\z_3 ds,
\end{align*}
which is exactly equation \eqref{READY_FOR_CTR} with $\kappa=\delta =0 $.  From Remark \eqref{DELTA_0}, this yields $ I_T^{122}\le C $.

For $I_T^{121} $, we again change variable to $\bar \z_3:=\frac{ \z_3}{(s-\sigma)^{\frac 1 \alpha} }$ as in $I_T^{21} $. This yields from \eqref{CTR_CAMBIO_RHO} that
$$
\Big\{ \rho_s\Big(\frac{\z_3}{(s-t)^{\frac 1\alpha}} \Big)>K\gamma\Big\}=\Big\{ (s-t)^{\frac 1 \alpha}+\sum_{i=1}^n |\bar \z_{3,i}(s-\sigma)^{\frac 1\alpha}(s-t)^{i-1}|^{\frac{1}{\alpha(i-1)+1}} >K\gamma\Big\}=:\{ \bar \rho_{s}(\bar \z_3)>K\gamma\},
$$
with $\bar \rho_s(\bar \z_3) $  as in \eqref{DEF_RHO_BAR}. Since
  $\Big\{ \rho_s\Big(\frac{\z_3}{(s-t)^{\frac 1\alpha}} \Big)>K\gamma, \, 2\sigma> {s+t} \Big \}$ $ \subset \{ \bar \rho_{s}(\bar \z_3)>K\gamma, \, |s-\sigma| \le \frac{1}{2} (s-t)\}.$
 Proceeding as for $I_T^{211} $, we arrive at  \eqref{vaii}
with $\beta =0$.
 We eventually get $I_T^1\le C $ which completes the proof. \hfill $\square $

\mysection{Proof of the 
Key Lemma \ref{KEY_LEMMA}
}
\label{KEY_SECTION}
We split the analysis for the two terms involved in Lemma \ref{KEY_LEMMA}. Roughly speaking, for the $(K_{i,\varepsilon}^Ff)_{i\in \leftB 1,n\rightB}$ there are \textit{no singularity} and we derive the required boundedness rather directly. On the other hand, the $(K_{i,\varepsilon}^Cf)_{i\in \leftB 1,n\rightB}$ need to be analysed much more carefully.

\subsection{Proof of the estimate on $(K_{i,\varepsilon}^Ff)_{i\in \leftB 1,n\rightB}$} \label{SEC_K_EPS_F}

In this section, we prove the following estimate. For a given 
$p \in (1,\infty)$, we have that there exists a constant $C_{p}$ such that for all $i\in \leftB 1,n\rightB $  and all  $f\in  
{\mathscr T}(\R^{1+N})$:
\begin{equation} \label{CTR_LP_LOIN}
\| K_{i,\varepsilon}^F f\|_{L^p(\Sc)}^p \le C_{p} \| f\|_{L^p(\Sc)}^p.
\end{equation}
We first write that for $i\in \leftB 1,n\rightB $:
\begin{eqnarray*}
\| K_{i,\varepsilon}^F f\|_{L^p(\Sc)}^p = \int_{\Sc}  |K_{i,\varepsilon}^F f (t,\x) |^p dt d\x
=\int_{\Sc} \left| \int_t^T \int_{\R^N} \I_{ d\big( (t,\x),(s,\y)\big) >c_0 }  \I_{|s-t| > \varepsilon } \Delta_{\x_i}^{\frac{\alpha_i} 2}p_\Lambda(s-t,\x,\y) f(s,\y) d\y ds  \right|^p  d\x dt\\
\le \int_{\Sc} \left( \int_t^T \int_{\R^N} \I_{ d\big( (t,\x),(s,\y)\big)>c_0 } \frac{1}{s-t}  \det(\M_{s-t})^{-1}q\Big(s-t, \M_{s-t}^{-1}\big(\y-e^{(s-t)A}\x\big) \Big) |f(s,\y)| d\y ds  \right)^p  d\x dt,
\end{eqnarray*}
where to get the last inequality, we discarded the time indicator and applied equation \eqref{CTR_DER_TEMPS_SING_KER} from Lemma \ref{SENS_SING_STAB} to estimate the fractional derivative (recalling 
the correspondence between $\Delta_{\x_i}^{\frac{\alpha_i} 2 }p_\Lambda(s,t,\x,\y) $ and $\Delta^{\frac{\alpha_i}{2},A,i,s-t}p_S\big(s-t, \M_{s-t}^{-1}(e^{(s-t)A}\x-\y)\big) $;
see \eqref{CORRISPONDENZA}).
Now, let us introduce $r>1$ such that $\frac 1 r + \frac 1 p = 1$.
Write
\begin{eqnarray*}
\frac{\det(\M_{s-t})^{-1}}{s-t}q\Big(s-t, \M_{s-t}^{-1}\big(\y-e^{(s-t)A}\x\big) \Big) |f(s,\y)| \\
= \Big[\frac{\det(\M_{s-t})^{-1}}{s-t}q\Big(s-t, \M_{s-t}^{-1}\big(\y-e^{(s-t)A}\x\big) \Big) \Big]^{\frac{1}{r}}
\Big[\frac{\det(\M_{s-t})^{-1}}{s-t}q\Big(s-t, \M_{s-t}^{-1}\big(\y-e^{(s-t)A}\x\big) \Big) \Big]^{\frac{1}{p}}|f(s,\y)|.
\end{eqnarray*}
To proceed we recall the following important equivalence result:
There exists $\kappa:=\kappa(\A{A})\ge 1$ s.t. for all $\big( (t,\x), (s,\y) \big) \in {\Sc}^2 $:
\begin{equation}
\label{EQUIV_SIMM_E_NO}
\kappa^{-1}\rho\big(s-t,e^{(s-t)A}\x-\y)  \le        d\big( (t,\x),(s,\y) \big) \le \kappa \rho\big(s-t,e^{(s-t)A}\x-\y).
\end{equation}
This equivalence can be deduced from \eqref{EQUIV_FLOW} established in the proof of Proposition \ref{PROP_DOUBLING} below.
Setting $\rho(s-t,e^{(s-t)A}\x-\y) := \rho$ the H\"older inequality yields:
\begin{eqnarray*}
\| K_{i,\varepsilon}^F f\|_{L^p(\Sc)}^p 
&\le& \int_{\Sc} \left( \int_t^T \int_{\R^N} \I_{ \rho >\frac{c_0}{ C}} \; \frac{1}{s-t}  \det(\M_{s-t})^{-1}q\Big(s-t, \M_{s-t}^{-1}\big(\y-e^{(s-t)A}\x\big) \Big)  d\y ds \right)^{\frac{p}{r}} 
\\
&&\times 
\left( \int_t^T \int_{\R^N} \I_{ \rho >\frac{c_0}C } \frac{1}{s-t}  \det(\M_{s-t})^{-1}q\Big(s-t, \M_{s-t}^{-1}\big(\y-e^{(s-t)A}\x\big) \Big) |f(s,\y)|^p  d\y ds\right)  d\x  dt,
\end{eqnarray*}
for $C:=C(\A{A})\ge 1$. Now, set $\z = \frac{\M_{s-t}^{-1}\big(\y-e^{(s-t)A}\x\big)} {|s-t|^{1/\alpha}}$; as in \eqref{CTR_CAMBIO_RHO} $\rho= \rho_s(\z)$. By  
the Fubini theorem, setting $\tilde c_0:=\frac{c_0}{C} $, we get
\begin{equation}\label{RHO_MACRO}
\begin{split}
\int_{t}^T\int_{\R^N} \I_{ \rho >\tilde c_0 } \frac{1}{s-t}  \det(\M_{s-t})^{-1}q\Big(s-t, \M_{s-t}^{-1}\big(\y-e^{(s-t)A}\x\big) \Big) d\y ds\le
\int_{ \rho_{s}(\z) >\tilde c_0 }  \frac{1}{|s-t|}  q(\z) d\z ds\le C ,\\
\int_{-T}^s\int_{\R^N} \I_{ \rho >\tilde c_0 } \frac{1}{s-t}  \det(\M_{s-t})^{-1}q\Big(s-t, \M_{s-t}^{-1}\big(\y-e^{(s-t)A}\x\big) \Big) d\x dt \le
\int_{ \rho_{t}(\z) >\tilde c_0 }  \frac{1}{|s-t|}  q(\z) d\z dt\le C ,
\end{split}
\end{equation}
with the notations of Lemma \ref{LABEL_LEMME_A_FAIRE}, taking $u(s) = s-t$ for the first integral and $u(t)=s-t $ for the second (see  also the next remark).
\begin{remark}\label{REM_RHO_MACRO}
Note that in \eqref{RHO_MACRO},  $\tilde c_0$ is a fixed constant. This allows to  get rid of the singularity.
Indeed, as in Lemma \ref{LABEL_LEMME_A_FAIRE}, either $|s-t|\ge \frac{\tilde c_0}{2}$ and there is no singularity, or
there exists $i \in \leftB 1,n\rightB$  such that $|\z_i|^{\frac{1}{\alpha(i-1)+1}}(s-t)^{\frac 1 \alpha} \ge \frac{\tilde c_0}{2n}$.
By the arguments of the proof of  Lemma \ref{LABEL_LEMME_A_FAIRE}, see also Remark \ref{DELTA_0}, we get the estimates of \eqref{RHO_MACRO}.
\end{remark}

\subsection{Proof of the estimate on the $(K_{i,\varepsilon}^Cf)_{i\in \leftB 1,n\rightB}$}\label{SEC_K_EPS_C}

In this section, we prove the following estimate for all $i\in \leftB 1,n\rightB,\ p \in (1,\infty)$:
\begin{equation}
\label{EST_LP_TRUNC_SING}
\| K_{i,\varepsilon}^C f\|_{L^p(\Sc)} \le C_{p} \| f\|_{L^p(\Sc)}.
\end{equation}
The idea is to rely first on Theorem \ref{coi}
to derive the estimate for $p \in (1,2]$.  To this purpose  
 we use that $(\Sc,d, \mu) $, where $\mu$ is the Lebesgue measure and $d$ is defined in \eqref{DEF_QD}, can be viewed as a homogeneous space (cf. Proposition \ref{PROP_DOUBLING}).

Then, to extend the estimate for all $p>2$ we use a duality argument.
 Let us recall  that by definition of $K_{i,\varepsilon}^Cf$, the kernel:
$$
k_{i,\varepsilon}^C \Big( (t,\x),(s,\y)\Big)=\I_{|s-t|\ge \varepsilon} \I_{ d ( (t,\x) , (s,\y) ) \le c_0 } \Delta_{\x_i}^{\frac{\alpha_i} 2} p_\Lambda(s-t,\x,\y),
$$
is in $L^2(\Sc\times \Sc)$, thanks to the truncations. The crucial points are an 
 $L^2$ estimate and a deviation lemma.
\begin{lemma}[$L^2$ estimate for the truncated kernel] \label{LEMME_L2_TRUNC}
There exists $C_{2,T}>0$ such that for all $i\in \leftB 1,n\rightB $: 
$$
\| K_{i,\varepsilon}^C f\|_{L^2(\Sc)} \le C_{2} \| f\|_{L^2(\Sc)},\;\;
 f \in  
{\mathscr T}(\R^{1+N}).
$$
\end{lemma}
\begin{proof} We have:
$$
\| K_{i,\varepsilon}^C f\|_{L^2(\Sc)}^2 = \int_\Sc | K_{i,\varepsilon}^C f (t,x) |^2  d\x dt \le   \int_\Sc | K_{i,\varepsilon} f (t,x) |^2  d\x dt + \int_\Sc | K_{i,\varepsilon}^F f (t,x) |^2  d\x dt.
$$
We now use Lemma \ref{LEMME_L2} to control the first contribution and \eqref{CTR_LP_LOIN} for the second.
\end{proof}

\begin{lemma}[Deviation Controls for the truncated kernel]\label{LEMME_DEV_TRUNC}
There exist constants $K:=K(\A{A}), C:=C(\A{A}\ge 1$ s.t. for all $i\in \leftB1,n\rightB $, $\varepsilon >0$ and $(t,\x),(\sigma,\bxi) \in {\Sc}$, the following control holds:
\begin{equation}
\label{CONTROL_DEV_TRONQ}
\int_{{\Sc}\cap \big\{d\big ((t,\x),(s,\y)\big)\ge K d\big( (t,\x),(\sigma,\bxi)\big)\big\}} \left|k_{i,\varepsilon}^C \Big( (t,\x),(s,\y)\Big) - k_{i,\varepsilon}^C \Big( (\sigma,\bxi),(s,\y)\Big)\right|
 d\y ds\le C.
\end{equation}
\end{lemma}

\begin{proof}
We will rely on the \textit{global} deviation Lemma \ref{LEMME_DEV}.  We write:
\begin{eqnarray}
&&\int_{\Sc \cap \{d\big ((t,\x),(s,\y)\big)\ge K d\big( (t,\x),(\sigma,\bxi)\big)\}} \left|k_{i,\varepsilon}^C \Big( (t,\x),(s,\y)\Big) - k_{i,\varepsilon}^C \Big( (\sigma,\bxi),(s,\y)\Big)\right|
 d\y ds\notag\\
&\le&\int_{\rho\ge K \gamma , \,  d((t,\x),(s,\y))\le c_0 , \, d((\sigma,\bxi),(s,\y))\le c_0}\left|k_{i,\varepsilon}^C \Big( (t,\x),(s,\y)\Big) - k_{i,\varepsilon}^C \Big( (\sigma,\bxi),(s,\y)\Big)\right|  d\y ds\notag\\
&&+\int_{\rho\ge K \gamma , \, d((t,\x),(s,\y))\le c_0, \, d((\sigma,\bxi),(s,\y))> c_0}\left|k_{i,\varepsilon}^C \Big( (t,\x),(s,\y)\Big) - k_{i,\varepsilon}^C \Big( (\sigma,\bxi),(s,\y)\Big)\right|  d\y ds\notag\\
&&+\int_{\rho\ge K \gamma , \,  d((t,\x),(s,\y))> c_0, \, d((\sigma,\bxi),(s,\y))\le c_0}\left|k_{i,\varepsilon}^C \Big( (t,\x),(s,\y)\Big) - k_{i,\varepsilon}^C \Big( (\sigma,\bxi),(s,\y)\Big)\right|  d\y ds=:T_1+T_2+T_3,\label{TAGLIA_DISTANZA}
\end{eqnarray}
where we used the equivalence \eqref{EQUIV_SIMM_E_NO} implicitly modifying $K$; here  $\rho = \rho(s-t,e^{(s-t)A}\x-\y)$  
and $\gamma:=\rho(\sigma-t,e^{(\sigma-t)A}\x-\bxi)$ as in  \eqref{DEF_RHO_GAMMA}. The second and third integrals $T_2,T_3$ are dealt similarly.
For $T_1$, the difference with  Lemma \ref{LEMME_DEV} is the time indicator and the spatial localization.

\textit{Control of $T_1$ in \eqref{TAGLIA_DISTANZA}.} For this term, we split according to the relative position of $s-t$ and $s-\sigma$.
We write:
\begin{eqnarray*}
&&\int_{\rho\ge K \gamma , \, d((t,\x),(s,\y))\le c_0 , \, d((\sigma,\bxi),(s,\y))\le c_0}\left|k_{i,\varepsilon}^C \Big( (t,\x),(s,\y)\Big) - k_{i,\varepsilon}^C \Big( (\sigma,\bxi),(s,\y)\Big)\right| d\y ds\\
&\le& \int_{(s-t)\wedge (s-\sigma)\ge \varepsilon  , \, \rho\ge K \gamma}\left|\Delta_{\x_i}^{\frac{\alpha_i}{2}}p_\Lambda (s-t,\x,\y) - \Delta_{\x_i}^{\frac{\alpha_i}{2}}p_\Lambda (s-\sigma,\bxi,\y)\right| d\y ds\\
&&+ \int_{\rho\ge K \gamma} \I_{ s-t \ge \varepsilon, s-\sigma < \varepsilon} |\Delta_{\x_i}^{\frac{\alpha_i}{2}}p_\Lambda (s-t,\x,\y)|d\y ds+\int_{\rho\ge K \gamma}\I_{ s-t  < \varepsilon, s-\sigma \ge \varepsilon} |\Delta_{\x_i}^{\frac{\alpha_i}{2}}p_\Lambda (s-\sigma,\bxi,\y)| d\y ds.
\end{eqnarray*}
Since $\{(s-t)\wedge (s-\sigma)\ge \varepsilon\} \subset \{ s\ge t\vee \sigma\}  $,
the first contribution above is dealt directly with  Lemma \ref{LEMME_DEV}.
The following two are dealt similarly, and we focus on the first one.
By \eqref{CORRISPONDENZA} and   \eqref{CTR_BOUND_SING_KER} of Lemma \ref{SENS_SING_STAB}  
we get:
\begin{eqnarray*}
\int_{\rho\ge K \gamma} \I_{ s-t \ge \varepsilon, s-\sigma < \varepsilon} |\Delta_{\x_i}^{\frac{\alpha_i}{2}}p_\Lambda (s-t,\x,\y)| d\y ds\\
\le \int_{\rho\ge K \gamma} \I_{ s-t \ge \varepsilon, s-\sigma < \varepsilon} \frac{C}{s-t} \frac{1}{\det (\M_{s-t})} q\Big(s-t , \M_{s-t}^{-1} (\y-e^{(s-t)A}\x) \Big)d\y ds.
\end{eqnarray*}
We now discuss according to the position of $\gamma$ relatively to $\varepsilon$.
\begin{trivlist}	
\item[-] Assume first that $\varepsilon \le \gamma^\alpha$.
In this case, we can write
\begin{equation}
\label{c2}
|s-t| \le |s-\sigma| + |\sigma - t| \le \varepsilon + \gamma^{\alpha} \le 2 \gamma^\alpha.
\end{equation}
Consequently, we write for all $\beta>0$:
$
\left(\frac{2 \gamma^\alpha}{s-t}\right)^\beta \ge 1,
$
 and changing variables to $\z=(s-t)^{-\frac 1\alpha}\M_{s-t}^{-1}(\y-e^{(s-t)A}\x) $ in the last integral leads to:
$$
\int_{\rho\ge K \gamma} \I_{ s-t \ge \varepsilon, s-\sigma \le \varepsilon} |\Delta_{\x_i}^{\frac{\alpha_i}{2}}p_\Lambda (s-t,\x,\y)| d\y ds\le C\int_{\rho_s(\z) \ge K \gamma} \frac{\gamma^{\alpha \beta}}{(s-t)^{1+\beta}} q(\z)  d\z ds,
$$
and we conclude with Lemma \ref{LABEL_LEMME_A_FAIRE}, taking $u(s) = s-t$.

\item[-] Assume now that $\varepsilon > \gamma^\alpha$.
In this case we write directly
\begin{eqnarray*}
&&\int_{\rho\ge K \gamma} \I_{ s-t \ge \varepsilon, s-\sigma < \varepsilon} \frac{1}{s-t} \frac{1}{\det \M_{s-t}}q\Big( s-t ,\M_{s-t}^{-1} (\y-e^{(s-t)A}\x)\Big)  d\y ds \\
&\le& \frac{1}{\varepsilon}\int_{\rho\ge K \gamma} \I_{ s-t \ge \varepsilon, s-\sigma < \varepsilon}\frac{1}{\det \M_{s-t}}q\Big( s-t ,\M_{s-t}^{-1}( \y-e^{(s-t)A}\x )\Big)  d\y ds.
\end{eqnarray*}
Also, since $\varepsilon  > \gamma^\alpha$, using \eqref{c2} 
 we actually have $|s-t|\le 2 \varepsilon$, so that:

\begin{eqnarray*}
\int_{\rho\ge K \gamma} \I_{ s-t \ge \varepsilon, s-\sigma < \varepsilon} |\Delta_{\x_i}^{\frac{\alpha_i}{2}}p_\Lambda (s-t,\x,\y)|   d\y ds
\le \frac{C}{\varepsilon}\int_{|s-t|\le 2 \varepsilon} \frac{1}{\det \M_{s-t}}q\Big( s-t ,\M_{s-t}^{-1}( \y-e^{(s-t)A}\x )\Big)  d\y ds \le C.
\end{eqnarray*}
\end{trivlist}

\textit{Control of $T_2,T_3 $ in \eqref{TAGLIA_DISTANZA}.} We focus on $T_2$ for which $d((t,\x),(s,\y))\le c_0,d((\sigma,\bxi),(s,\y))> c_0 $. The term $T_3 $ could be handled similarly.
We have:
\begin{eqnarray*}
T_2
= \int_{\rho\ge K \gamma , \, d((t,\x),(s,\y))\le c_0 , \, d((\sigma,\bxi), (s,\y))> c_0}\left|k_{i,\varepsilon}^C \Big( (t,\x),(s,\y)\Big)\right| d\y ds.
\end{eqnarray*}
Using the quasi triangle inequality \eqref{QDT} below, we have:
$$
c_0 \ge d((t,\x),(s,\y)) \ge \frac{d((\sigma,\bxi),(s,\y))}{\Lambda}-d((\sigma,\bxi),(t,\x))  \ge \frac{c_0}{\Lambda} -\gamma \kappa,
$$
exploiting as well \eqref{EQUIV_SIMM_E_NO} for the last inequality.
We now split according to the relative position of $\gamma$ and $c_0$.
\begin{trivlist}
\item[-] Assume first that $\gamma \le \frac{c_0}{2\Lambda \kappa}$.
In this case, we get $\frac{c_0}{\Lambda} - \gamma \kappa\ge \frac{c_0}{2\Lambda }$, and we are left with the integral:
\begin{eqnarray*}
&&\int_{\rho\ge K \gamma, d((t,\x),(s,\y))\le c_0 , \, d((\sigma,\bxi),(s,\y))> c_0}\left|k_{i,\varepsilon}^C \Big( (t,\x),(s,\y)\Big)\right| d\y ds\\
&\le&
C \int_{\rho \ge \frac{c_0}{2\Lambda} } \frac{1}{s-t} \frac{1}{\det \M_{s-t}}q\Big( s-t ,\M_{s-t}^{-1} (\y-e^{(s-t)A}\x)\Big)  d\y ds
\le C
\int_{\rho_s(\z) \ge \frac{c_0}{2\Lambda} }\frac{1}{s-t}  q(\z) d\z ds,
\end{eqnarray*}
 $\z =  \M_{s-t}^{-1} (e^{(s-t)A}\x-\y) (s-t)^{-1/\alpha}$,
bounding the fractional derivative with estimate \eqref{CTR_BOUND_SING_KER}.
We are thus exactly in the same position as in Remark \ref{REM_RHO_MACRO} that allows to control the above integral.

\item[-] Suppose now that $\gamma  > \frac{c_0}{2\Lambda}$.
In this case, we readily get:
$
\rho \ge K \gamma \ge \frac{K c_0}{2\Lambda}
$
and the corresponding integral can be bounded similarly.
\end{trivlist}
\end{proof}
 Recalling that by construction, for all $i\in \leftB 1,n\rightB $, $k_{i,\varepsilon}^C\in L^2(\Sc\times \Sc) $, estimate \eqref{EST_LP_TRUNC_SING} now follows from Lemmas \ref{LEMME_L2_TRUNC}, \ref{LEMME_DEV_TRUNC} and  Theorem \ref{coi} below
for $p\in (1,2] $ since from Proposition \ref{PROP_DOUBLING}, $(\Sc,d, \mu) $ can be viewed as a homogeneous space. Estimate \eqref{EST_LP_TRUNC_SING}  for $p\in (2,+\infty)$ can then be derived by duality as follows. Let $p\in (2,+\infty) $  be given and consider $g\in L^r(\Sc)$ with $r>1,\ \frac{1}{r}+\frac{1}p=1 $. Then, for all $i\in \leftB 1,n\rightB $ and $f\in L^p(\Sc)$,
\begin{eqnarray*}
\int_\Sc K_{i,\varepsilon}^Cf(t,\x) g(t,\x) d\x dt&=&\int_{-T}^T \int_{\R^N}\Big(\int_t^T  \int_{\R^N} k_{i,\varepsilon}^C\big((t,\x),(s,\y)\big)f(s,\y)  d\y ds\Big) g(t,\x) d\x dt\\
&=&\int_{-T}^T \int_{\R^N}\Big( \int_{-T}^s \int_{\R^N}k_{i,\varepsilon}^C\big((t,\x),(s,\y)\big) g(t,\x)  d\x dt\Big) f(s,\y) d\y ds \\
&=:&\int_{-T}^T \int_{\R^N} \bar K_{i,\varepsilon}^Cg(s,\y)f(s,\y) d\y ds,
\end{eqnarray*}where $\bar K_{i,\varepsilon}^C $ is the adjoint of $K_{i,\varepsilon}^C $. Recall   that $k_{i,\varepsilon}^C\big((t,\x),(s,\y)\big):=\Delta_{\x_i}^{\frac{\alpha_i}{2}} p_{\Lambda}(s-t,\x,\y)\I_{|s-t|\ge \varepsilon} \I_{ d ( (t,\x) , (s,\y) ) \le c_0 }$; see  \eqref{DEF_NUCLEO_SINGOLARE_TRONCATO}.
From \eqref{CORRISPONDENZA}, using that 
$$\M_{s-t}^{-1}(e^{(s-t)A}\x-\y) =
\M_{s-t}^{-1} \M_{s-t} e^{A} \M_{s-t}^{-1} (\x - e^{-(s-t)A}\y),
$$
see the scaling property \eqref{str}, we find
$$
p_{\Lambda}(s-t,\x,\y)=\frac{1}{\det(\M_{s-t})}p_S\big(s-t,\M_{s-t}^{-1}(e^{(s-t)A}\x-\y)\big)=:\frac{1}{\det(\M_{s-t})}p_{\tilde S}\big(s-t,  \M_{s-t}^{-1}(e^{-(s-t)A}\y-\x)\big),
$$
with $\tilde S_{r}=e^{-A}S_{r}, r\ge 0 $, where the symmetric $\R^N$-valued, $\alpha $-stable process $S $ is defined in Remark \ref{ID_LAW}.  
We get:
\begin{eqnarray*}
\Delta_{\x_i}^{\frac{\alpha_i}{2}} p_{\Lambda}(s-t,\x,\y)=\Delta_{\x_i}^{\frac{\alpha_i}{2}}\frac{1}{\det(\M_{s-t})}p_{\tilde S}\big(s-t,  \M_{s-t}^{-1}(e^{-(s-t)A}\y-\x)\big)\\
=\frac{1}{\det(\M_{s-t})}{\rm v.p.}\int_{\R^{d_i} } \Big (p_{\tilde S}(s-t, \M_{s-t}^{-1}(e^{-(s-t)A}\y-\x+B_i z))-p_{\tilde S}(s-t, \M_{s-t}^{-1}(e^{-(s-t)A}\y-\x)) \Big) \frac{dz}{|z|^{d_i+\alpha_i}}\\
=\frac{1}{\det(\M_{s-t})}{\rm v.p.}\int_{\R^{d_i} } \Big (p_{\tilde S}(s-t, \M_{s-t}^{-1}(e^{-(s-t)A}\y-\x)+(s-t)^{-(i-1)}B_i z))\\
-p_{\tilde S}(s-t, \M_{s-t}^{-1}(e^{-(s-t)A}\y-\x)) \Big) \frac{dz}{|z|^{d_i+\alpha_i}}=:\frac{1}{\det(\M_{s-t})} \bar \Delta^{\frac{\alpha_i}2,i,s-t}p_{\tilde S}\big(s-t, \M_{s-t}^{-1}(e^{-(s-t)A}\y-\x)\big),
\end{eqnarray*}
where for  $\varphi\in C_0^{\infty}(\R^N) $ and for all $i\in \leftB 1,n\rightB,\ s>t $ the operator: 
\begin{equation}
\label{DEF_DELTA_MOD_2}
\bar \Delta^{\frac{\alpha_i}2,i,s-t}\varphi(\x):=\int_{\R^{d_i} } \Big (\varphi(\x+(s-t)^{-(i-1)}B_iz)-\varphi(\x)- (s-t)^{-(i-1)}\nabla \varphi(\x)\cdot B_iz \I_{|z|\le 1} \Big) \frac{dz}{|z|^{d_i+\alpha_i}},
\end{equation}
where $\nabla $ again stands for the full gradient on $\R^N$.
We thus conclude that
\begin{eqnarray*}
\bar K_{i,\varepsilon}^C g(s,\y)
&=&\int_{-T}^s \int_{\R^N} \frac{1}{\det(\M_{s-t})} \bar \Delta^{\frac{\alpha_i}2,i,s-t}p_{\tilde S}\big(s-t, \M_{s-t}^{-1}(e^{-(s-t)A}\y-\x)\big) \I_{|s-t|\ge \varepsilon} \I_{ d ( (t,\x) , (s,\y) ) \le c_0 }g(t,\x)  d\x dt\\
&=:&\int_{-T}^s \int_{\R^N}\bar k_{i,\varepsilon}^C\big((s,\y),(t,\x)\big) g(t,\x)  d\x dt.
\end{eqnarray*}
We derive similarly to the previous computations that for all $r\in (1,2]$, there exists $C_{r}$, s.t. for all $g\in L^r(\Sc) $,
$\|\bar K_{i,\varepsilon}^C g\|_{L^r(\Sc)} \le C_{r}\|g\|_{L^r(\Sc)}$. The control
$\| K_{i,\varepsilon}^C f\|_{L^p(\Sc)} \le C_{p}\|f\|_{L^p(\Sc)}$
follows now by duality. Lemma \ref{KEY_LEMMA} eventually  readily derives for such $p$ from \eqref{CTR_LP_LOIN} and \eqref{EST_LP_TRUNC_SING}. \qed 

\appendix

\mysection{Auxiliary Technical Results}
\label{APP_TEC}
\noindent Here we give the proof of some technical results for the sake of completeness.
We first prove estimate \eqref{CTR_ND}.
\begin{lemma}\label{bound}
There exists a constant $ c:=c(\A{A})>0$,
 such that for all $t\in [-T,T] $, and all $\p \in \R^N$:
$$
\int_0^1 \int_{\S^{d-1}} |\langle   \M_t\p, \exp(v A)B \sigma \s\rangle|^\alpha \mu(d\s)dv\ge c|\M_t\p|^\alpha.
$$
\end{lemma}
\begin{proof}
Thanks to \A{ND}, we have:
$$ \int_0^1 \int_{\S^{d-1}} |\langle   \M_t\p, \exp(v A)B \sigma \s\rangle|^\alpha \mu(d\s)dv\ge c|\M_t\p|^\alpha \int_0^1|(\exp(vA)B\sigma)^*\frac{\M_t\p}{|\M_t\p|}|^\alpha  dv.
$$
Defining
$\bar C:= \inf_{\btheta \in \S^{N-1}} \int_0^1 |(\exp(v A)B \sigma)^* \btheta |^\alpha dv,$
it thus  actually suffices to prove that $\bar C>0$.
By continuity of the involved functions and compactness of $\S^{N-1}$, the infimum is actually a minimum.
We need to show that this quantity is not zero.
 Arguing as in \cite{prio:zabc:09} page 49, we can use the fact that assumptions \A{UE} and \A{H} imply the rank condition
 \begin{equation} \label{rank}
\! \! \! \! \! \!  \! \! \! \! \! \! \! \! \!  \;\;\; \mbox{Rank}
\, [B \sigma, AB\sigma, \ldots, A^{N-1}B\sigma]=N.
\end{equation}
 Here $[B \sigma, AB \sigma, \ldots, $ $A^{N-1}B \sigma]$  denotes the $N \times Nd$
 matrix,
 formed by the matrices   $B \sigma, \ldots, A^{N-1}B \sigma$.

 Set  $  M = \sup \{ | (B\sigma)^* e^{v A^*}  \btheta | \, :\,   v \in [0,1],\,
|\btheta| \in {\mathbb S}^{N-1} \}  +1. $
  since $ \Big |
\frac{ (B\sigma)^* e^{s A^*}  \btheta }
 {\, M } \Big| \le 1$, $s \in [0,1]$, we get
$$
\int_0 ^1    | (B\sigma)^* e^{v A^*}   \btheta |^{\alpha} dv
 =    M ^{\alpha}  \, \int_0 ^1
  \Big | \frac{ (B\sigma)^* e^{v A^*}   \btheta  }
 { M } \Big|^{\alpha} dv \,
 \ge \,
  M^{\alpha}  \, \int_0 ^1
  \Big | \frac{ (B\sigma)^* e^{v A^*}   \btheta  }
 {M } \Big|^{2} dv.
$$
 Let us recall  that
   the  rank condition \eqref{rank} implies (actually the two conditions are  equivalent)  the existence of
 $C_1 >0$ such that, for any $u \in \R^N$,
  $\int_0 ^1  | (B \sigma)^* e^{v A^*}  u|^2 \, dv$
 $\ge C_1 |u|^2$ (see e.g.   \cite{prio:zabc:09}; \nero{an alternative proof 
of this fact can be done following \cite{huan:meno:15} (see page 33 in \cite{huan:meno:prio:16}).
This completes the proof.}
\end{proof}

We now turn to the estimate \eqref{CTR_DER_M} for the derivative of the density $p_M$.
\begin{lemma}[Derivative of the density of the small jumps part.]\label{EST_DENS_MART}
For all $m\ge 1 $ and all multi-indices $\i=(i_1,\cdots, i_N) \in \N^N ,\ |\i|:=\sum_{j=1}^N|i_j|\le 3$, there exists $C_{m,\i}$ s.t. for all $(t,\x)\in \R_+^*\times \R^N  $:
$$
|
\partial_\x^\i p_M(t,\x)|\le \frac{C_{m,\i}}{t^{\frac{N+|\i|}{\alpha}}} \left( 1+ \frac{|\x|}{t^{\frac 1\alpha}}\right)^{-m}.
$$
\end{lemma}
\begin{proof}
Below, similarly to \eqref{STABLE_MEAS}, we decompose the L\'evy measure $\nu_S $ of $S $ as
$$
  \nu_S (D) = \int_{\S^{N-1}} \tilde \mu_S(d \bxi) \int_0^{\infty} {\I}_D (r \bxi)
 \frac{dr}{r^{1+ \alpha}}, \;\; D\in {\mathcal B}(\R^N).
 $$
 According to \eqref{dec}, expressing $p_M(t,\x)$ 
 as an inverse Fourier transform, for all multi-indices $\i=(i^1,\cdots, i^N) \in \N^N ,\ |\i|:=\sum_{j=1}^N|i_j|$ $\le 3$, 
 we have by the L\'evy-Khintchine formula (recall that $\nu_S$ is symmetric): 
 \begin{eqnarray*}
\partial_\x^\i p_M(t,\x)= \frac{1}{(2\pi)^{N}} 
 \int_{\R^{N}} e^{-i\langle \p,\x \rangle} (-i\p)^\i 
\exp \left(  t
\int_{\S^{N-1}} \int_0^{\infty}
 \Big(\cos (\langle \p, r\bxi \rangle)  - 1  \Big)\I_{\{r\le t^{\frac 1\alpha}\}} \frac{dr}{r^{1+\alpha}}\tilde \mu_{S}(d\bxi)\right)d \p,
\end{eqnarray*}
(recall that we have the term $\cos (\langle \p, r\bxi \rangle) -1 $ since $\nu_S$ is symmetric). 
Changing variables in $t^{\frac 1\alpha}\p =\q$ yields:
\begin{eqnarray*} 
\partial_\x^\i p_M(t,\x)= \frac{t^{-\frac{N+|\i|}{\alpha}}}{(2\pi)^{N}} \int_{\R^{N}} e^{-i \big\langle \q,\frac{\x}{t^{\frac 1\alpha}} \big\rangle}\\
\cdot \;  (-i\q)^\i \exp \left(  t \int_{\S^{N-1}} \int_0^{\infty}  \Big(\cos \big(\langle \q, \frac{r\bxi}{t^\frac{1}{\alpha}} \rangle \big)- 1   \Big)\I_{\{r\le t^{\frac 1\alpha}\}} \frac{dr}{r^{1+\alpha}}\tilde \mu_{S}(d\bxi)\right)d \q.
\end{eqnarray*}
Observe that changing variables to $\rho=rt^{-\frac 1 \alpha}$, we have:
\begin{eqnarray*}
&&\q^\i \exp \left(  t \int_{\S^{N-1}} \int_0^{\infty} \Big( \cos \big (\big\langle \q, \frac{r\bxi}{t^\frac{1}{\alpha}}  \big\rangle \big) - 1   \Big)\I_{\{r\le t^{\frac 1\alpha}\}} \frac{dr}{r^{1+\alpha}}\tilde \mu_{S}(d\bxi)\right)\\
&=&\q^\i \exp \left( - \int_{\S^{N-1}} \int_0^{\infty} \Big(1-\cos \langle \q, \rho\bxi\rangle \Big)\I_{\{ \rho\le 1 \}} \frac{d\rho}{\rho^{1+\alpha}}\tilde \mu_{S}(d\bxi)\right)=:\hat{f}(\q).
\end{eqnarray*}
It is not difficult to differentiate under the integral sign and get that 
$\hat f$ is infinitely differentiable as a function of $\q$
(cf. Theorem 3.7.13 in \cite{Jacob1}). 
 Besides, we can bound the truncated measure by the complete one, up to a multiplicative constant. Since $\I_{\{ \rho\le 1 \}} = 1 - \I_{\{ \rho > 1 \}}$, we obtain
 \begin{eqnarray*}
|\hat f( \q)| &\le& |\q^\i| \exp\Big(\frac{2}{\alpha}\tilde \mu_S(\S^{N-1}) \Big) \exp \left( - \int_{\S^{N-1}} \int_0^{\infty} \Big(1-\cos \langle \q, \rho\bxi\rangle \Big)\frac{d\rho}{\rho^{1+\alpha}}\tilde \mu_{S}(d\bxi)\right)
\\
&\le& C |\q|^{|\i|}\exp\Big(-\int_{\S^{N-1}} | \langle \q, \bxi \rangle |^\alpha \mu_S (d\bxi)\Big) \le C|\q|^{|\i|} e^{-c^{-1} |\q|^\alpha},
\end{eqnarray*}
 using that the spectral measure $\mu_S$ satisfies the non-degeneracy condition \A{ND} for the last inequality.
Thus, $\hat f$ belongs the Schwartz space ${\mathscr S}(\R^N) $. Denoting by $f$ its Fourier transform, we have:
$$
\forall m \ge 0, \exists C_m\ge 1, \forall \z \in \R^{N},\  |f(\z)| \le C_m (1+|\z|)^{-m}.
$$
Now since $|\partial_\x^\i p_M(t,\x)| = t^{-\frac{N+|\i|}{\alpha}} |f\left( \frac{\x}{t^{\frac 1\alpha}}\right)|$, the announced bound follows.
\end{proof}

\section{$L^p$ estimates for the local case $\alpha =2$}
\label{EST_LOCAL_CASE}

We focus here on the case $\alpha=2 $, i.e. $L_\sigma =\frac 12 {\rm Tr}(\sigma\sigma^* D_{\x_1}^2)$. In this simpler  case our main results of Theorems \ref{THE_THM1} and \ref{STIMA_STRISCIA} continue to hold. Let us consider parabolic estimates in  Theorem \ref{STIMA_STRISCIA}. We only show how to prove the result following our previous arguments.

We first note that the case  $\alpha_1=\alpha =2$ follows by \cite{bram:cerr:manf:96} and \cite{bram:cupi:lanc:prio:10}.
On the other hand we concentrate on the new case when  $i\in \leftB 2,n\rightB $. 
 We start as in  Section \ref{HOMO}, for $t>0$:  
$$\Lambda_t^\x =e^{tA}\x+\int_0^t e^{(t-s)A} B \sigma dW_s, 
$$
where $W = (W_t)$ is a $d$-dimensional Wiener process. Here we use an alternative   approach to obtain an analogous of  \eqref{d33} (this more probabilistic approach could be also used  in stable case of Section \ref{HOMO}).
 \\
By  the independence of increments of the Wiener process,  
the Fourier transform shows  that 
\begin{equation} \label{c11}
\Lambda_t^\x \overset{({\rm law})}{=}  e^{tA}\x+\int_0^t e^{s A}  B \sigma dW_s.
\end{equation} 
Using that $e^{ r t A}  = \M_t e^{ r  A} \M_t^{-1} $ (see
 \eqref{str}) and the fact that $\M_t^{-1} B \sigma = B \sigma$ we obtain
 \begin{eqnarray*}
J_t = \int_0^t e^{ s A}  B \sigma dW_s  
\overset{({\rm law})}{=}  t^{1/2} \int_0^1 e^{ r t A}  B \sigma dW_r
=   t^{1/2} \M_t \int_0^1 e^{ r  A}  B \sigma dW_r =  \M_t S_t,\\ 
\text{with} \;\; 
S_t = t^{1/2}  \int_0^1 e^{ r  A}  B \sigma dW_r \overset{({\rm law})}{=} {\bf K} {\bf W_t}, \;\;\; t \ge 0,
\end{eqnarray*} 
where ${\bf K}$ is an $N \times N$ square root of the non degenerate matrix $ {\bf \Gamma} = \int_0^1 e^{As}B \sigma \sigma^*B^*e^{A^*
s} ds$ (the non degeneracy follows by our 
assumptions on $\sigma$ and $A$), i.e. ${\bf\Gamma}={\bf K}{\bf K}^* $, and $({\bf W_t})$ is  a standard $N$-dimensional Wiener process. 

\nero{ The structure of \eqref{THE_DENS_OU} in Proposition \ref{THE_PROP_FUND} still holds
in the limit case, replacing
 $\exp\Big(-t\int_{\S^{N-1}}|\langle \p, \bxi\rangle|^\alpha \mu_S(d\bxi)\Big)$ with $\exp\big( -\frac{t}2 \langle \Gamma \p, \p \rangle\big) $. Observe as well that we can enter the previous framework constructing a symmetric measure $\mu_S $ on $\S^{N-1} $ s.t. $ \int_{\S^{N-1}}\p^* \btheta\btheta^*\p \mu_S(d\btheta)=  \int_{\S^{N-1}}|\langle \p, \btheta \rangle |^2 \mu_S(d\btheta)=\frac{\langle \Gamma \p,\p \rangle}{2} $. 
 Denoting by ${\bf K}_i$ the $i^{{\rm th}} $  column of ${\bf K} $, one can take $\mu_S=\frac 14 \sum_{i=1}^N |{\bf K}_i|(\delta_{\frac{{\bf K}_i}{|{\bf K}_i|}}+\delta_{-\frac{{\bf K}_i}{|{\bf K}_i|} }) $, where $\delta_{{\bf u}} $ stands for the Dirac mass at point  ${\mathbf u} $.
 }


The Fourier argument of Lemma \ref{LEMME_L2} still apply and the $L^2$ control stated therein remains valid. Observe 
that, to investigate the $L^p $ case for $p\neq 2 $, we clearly need to consider the quasi-distance in \eqref{DEF_QD} with  $\alpha=2 $. 

For Lemma \ref{LEMME_DEV} the case $i=1 $ involves the local operator $\Delta_{\x_1} $. Such deviations results have been proved in this framework by Bramanti \textit{et al.} \cite{bram:cerr:manf:96} (see Proposition 3.4 therein).
 We thus focus on the new contributions associated with $i\in \leftB 2,n \rightB $, which involve the operators $(\Delta_{\x_i}^{\frac{1}{1+2(i-1)}})_{i\in \leftB 2,n\rightB}$. For those contributions, the proof of Lemma  \ref{LEMME_DEV} remains the same  provided we prove the key controls of Lemmas \ref{SENS_SING_STAB} and \ref{LABEL_LEMME_A_FAIRE}. We present below a proof of Lemma \ref{SENS_SING_STAB} for $\alpha=2 $. With this result, Lemma \ref{LABEL_LEMME_A_FAIRE} still holds for  $\alpha=2 $ with the same proof.
  The final derivation of the main results of Theorems \ref{THE_THM1} and \ref{THM_STRISCIA} is then the same as in Section \ref{HOMO} and Sections \ref{STRAT_PROOF}-\ref{KEY_SECTION} respectively.
  
 \smallskip
 
\textit{Proof of Lemma \ref{SENS_SING_STAB} for $\alpha=2 $ and $i\in \leftB 2,n\rightB $.} Observe indeed that for such indexes the correspondence \eqref{CORRISPONDENZA} still holds. The main difference with the case $\alpha \in (0,2) $ that we considered before is that we do not use the previous decomposition \eqref{DECOMP_G_P}, which splits for $\alpha\in (0,2) $ the small and large jumps, but directly exploit the Gaussian character of $p_S(t,\cdot)$. 

Let  us prove point \textit{(ii)}. With the notations of \eqref{DICO_PICCOLI_GRANDI_SALTI_DI_A}, i.e.  $\Delta_{\x_i}^{\frac{\alpha_i}{2},i,A,t,l}p_S(t,\cdot),\Delta_{\x_i}^{\frac{\alpha_i}{2},i,A,t,s}p_S(t,\cdot) $ corresponding  respectively  to the large and small jumps part in the operator, we rewrite for all $t>0, (\x,\x')\in \R^{2N} $:
 \begin{eqnarray*}
\Delta_{\x_i}^{\frac{\alpha_i}{2},i,A,t}p_S(t,\x)-\Delta_{\x_i}^{\frac{\alpha_i}{2},i,A,t}p_S(t,\x') \\
=\Big(\Delta_{\x_i}^{\frac{\alpha_i}{2},i,A,t,l}p_S(t,\x)-\Delta_{\x_i}^{\frac{\alpha_i}{2},i,A,t,l}p_S(t,\x')\Big)+\Big(\Delta_{\x_i}^{\frac{\alpha_i}{2},i,A,t,s}p_S(t,\x)-\Delta_{\x_i}^{\frac{\alpha_i}{2},i,A,t,s}p_S(t,\x')\Big).
 \end{eqnarray*}
 Similarly to \eqref{CTR_PREAL},
 \begin{eqnarray}
\Big|\LALB{i}{t} p_S(t,\x)-\LALB{i}{t} p_S(t,\x')\Big|\notag\\
\le \left(\int_{|z|\ge t^{\frac 1{\alpha_i}}} |p_S(t,\x+t^{-(i-1)}(e^A)_iz)-p_S(t,\x'+t^{-(i-1)}(e^A)_iz)|\frac{dz}{|z|^{d_i+\alpha_i}} \right)\nonumber\\
+\left(\frac{C}t \Big|p_S(t,\x)-p_S(t,\x') \Big|\right)
=:
(I_{i,1}+I_{2})(t,\x,\x').\label{CTR_PREAL_2}
 \end{eqnarray}
 In the current Gaussian case, a control similar to the previous \eqref{HOLDER_S} also holds. Precisely, for all  $t>0, (\x,\x')\in \nero{\R^{2N}}, \beta \in (0,1] $: 
 \begin{trivlist}
\item[-] If $|\x-\x'|\ge t^{\frac 12} $, then $|p_S(t,\x)-p_S(t,\x')|\le \left(\frac{|\x-\x'|}{t^{\frac 12}} \right)^\beta\big(p_S(t,\x)+p_S(t,\x')\big)$.
\item[-] If $|\x-\x'|\le t^{\frac 12} $, then usual computations yield (using also $(a-b)^2 \ge \frac{a^2}{2} - b^2$): 
\begin{eqnarray*}
|p_S(t,\x)-p_S(t,\x')|&\le& \int_{0}^1 |\nabla p_S\big(t,\x'+\lambda(\x-\x')\big)||\x-\x'|d\lambda \le C \frac{|\x-\x'|}{t^{\frac{1+N}2}}\exp\Big(-\frac{c}{t}\big(\frac12 |\x'|^2-|\x-\x'|^2\big)\Big)\\
&\le & C |\x-\x'|^{\beta} \, {t^{-  \frac \beta 2}} \,  g_c(t,\x'),
\end{eqnarray*}
where $C\ge 1, 0< c \le 1 $, and the Gaussian density 
$$
\ g_c(t,\z):=\frac{c^{\frac N2}}{(2\pi t)^{\frac{N}2}}\exp\left(-c\frac{|\z|^2}{2t} \right) = \nero{ p_{\bar S}(t,\z)},\;\;\; t>0, \;  \z\in \R^N, 
$$ 
verifies: $|\partial_\x^\i p_S(t,\x)|\le \bar C_{\i}  
t^{-\frac{|\i| }{2}}  \, g_c(t,\x)$,  for  all   $\i=(i^1,\cdots, i^N) \in \N^N,$  $\ |\i|:=\sum_{j=1}^N  i_j\le 2$, $t>0$, $\x \in \R^N$.
 \end{trivlist} Hence, by symmetry we derive:
 \begin{equation}
\label{HOLDER_S_2}
|p_S(t,\x)-p_S(t,\x')|\le C\left( \frac{|\x-\x'|}{t^{\frac 12}}\right)^{\beta}\Big(g_c(t,\x)+g_c(t,\x') \Big)=:C\left( \frac{|\x-\x'|}{t^{\frac 12}}\right)^{\beta}\Big(p_{\bar S}(t,\x)+p_{\bar S}(t,\x') \Big).
 \end{equation} 
 From \eqref{HOLDER_S_2} and \eqref{CTR_PREAL_2} we get: 
\begin{eqnarray}
\label{CTR_I2_2}
|I_2(t,\x,\x')|\le \frac{C}{t}\left( \frac{|\x-\x'|}{t^{\frac 1 2}}\right)^\beta \big(p_{\bar S}(t,\x)+p_{\bar S}(t,\x') \big).
\\ \nonumber 
|I_{i,1}(t,\x,\x')|\le C \left( \frac{|\x-\x'|}{t^{\frac 1 2}}\right)^\beta
\int_{|z|\ge t^{\frac 1{\alpha_i}}}  \big(p_{\bar S}(t,\x+t^{-(i-1)}(e^A)_iz)+p_{\bar S}(t,\x'+t^{-(i-1)}(e^A)_iz) \big) \frac{dz}{|z|^{d_i+\alpha_i}}\\ \nonumber
=:\frac{C}t \left( \frac{|\x-\x'|}{t^{\frac 1 2}}\right)^\beta
\int_{\R^{d_i}}  \big(p_{\bar S}(t,\x+t^{-(i-1)}(e^A)_iz)+p_{\bar S}(t,\x'+t^{-(i-1)}(e^A)_i z) \big) f_{\Gamma^i}(t, z)(dz),
\end{eqnarray}
setting \nero{$f_{\Gamma^i}(t,z): =t c_{\alpha,d_i}\I_{|z|\ge t^{\frac 1{\alpha_i}}}\frac{1}{|z|^{d_i+\alpha_i}} $} with $c_{\alpha,d_i}>0$ s.t. $\int_{\R^{d_i}}f_{\Gamma^i}(t,z) dz=1 $. Hence, $f_{\Gamma^i}(t, \cdot) $ is the density of an $\R^{d_i}$-valued random variable $\Gamma_t^i $. The above integrals can thus be seen as the densities, at point $\x$ and $\x'$ respectively, of the random variable
\begin{equation}
\label{DEF_bar_S1_2}
\bar S_t^{i,1}:=\bar S_t+t^{-(i-1)}(e^A)_i\Gamma_t^i,
\end{equation}
where $\bar S_t = c_1 {\bf W_t} $ has density $p_{\bar S} (t, \cdot)$ (the process $(\bar S_t)_{t \ge 0}$  is proportional to a standard  $N$-dimensional Wiener process ${\bf W}$)
and $\Gamma_t^i $ is independent of $\bar S_t$ and has density $f_{\Gamma^i}(t, \cdot) $.
We finally obtain,
\begin{equation}
\label{CTR_I1_2}
|I_{i,1}(t,\x,\x')|\le \frac Ct \left( \frac{|\x-\x'|}{t^{\frac 1 2}}\right)^\beta \big(p_{\bar S^{i,1}}(t,\x)+p_{\bar S^{i,1}}(t,\x')\big).
\end{equation}
From the definition of $\alpha_i=\frac{2}{1+ 2(i-1)} $ and \eqref{DEF_bar_S1_2},  one can check that 
 $\bar S_t^{i,1}\overset{({\rm law})}{=} t^{\frac 12} \bar S_1^{i,1}$ and 
 similarly to \eqref{INT_2} that the density  $p_{\bar S^{i,1}}(t,\cdot) $ of the random variable $\bar S_t^{i,1}$ satisfies \eqref{INT_Q}. Here, the integrability constraint is still given by $\Gamma_t^i $, there are no constraints on $\bar S_t$ which is Gaussian. The controls for $|\LALB{i}{t} p_S(t,\x)-\LALB{i}{t} p_S(t,\x')| $ follows plugging \eqref{CTR_I1_2} and \eqref{CTR_I2_2} into \eqref{CTR_PREAL_2} defining $q(t,\cdot):=\frac{1}{n+1}(\sum_{i=1}^n p_{\bar S^{i,1}}+p_{\bar S})(t,\cdot).$

It remains to control the difference associated with the \textit{small jumps} part. Write
\begin{eqnarray} \label{PRELIM_SMALL_JUMPS}
&&\LALS{i}{t} p_S(t,\x)-\LALS{i}{t} p_S(t,\x')\\
&=&\int_{|z|\le t^{\frac 1{\alpha_i}}}\Big[\big(p_S(t,\x+t^{-(i-1)}(e^A)_iz)-p_S(t,\x)\big) -\big(p_S(t,\x'+t^{-(i-1)}(e^A)_iz)-p_S(t,\x')\big)\Big] \frac{dz}{|z|^{d_i+\alpha_i}}\nonumber\\ \nonumber
&=&\int_{|z|\le t^{\frac 1{\alpha_i}}}  \int_0^1 d\mu \Big(\nabla p_{ S}(t,\x+\mu t^{-(i-1)}(e^A)_iz)-\nabla p_{S}(t,\x'+\mu t^{-(i-1)}(e^A)_iz)\Big)\cdot  t^{-(i-1)}(e^A)_iz
\frac{dz}{|z|^{d_i+\alpha_i}},
\end{eqnarray}
Usual Gaussian calculations then give that,  if $|\x-\x'|\le t^{\frac{1}{2}} $:
\begin{eqnarray} \label{CTR_SMALL_DIAG_2}
|\LALS{i}{t} p_S(t,\x)-\LALS{i}{t} p_S(t,\x')|\\
\le
C\frac{|\x-\x'|}{t}
\int_{|z|\le t^{\frac 1{\alpha_i}}} \int_0^1 d\mu\int_0^1d\lambda \, g_{c}\big(t,\x'+\mu t^{-(i-1)}(e^A)_iz+\lambda(\x-\x') \big)  |z|t^{-(i-1)} \frac{dz}{|z|^{d_i+\alpha_i}}
\nonumber\\ \nonumber
\le
C \frac{|\x-\x'|}{t}  g_c(t,\x')\int_{|z|\le t^{\frac 1{\alpha_i}}} t^{-(i-1)}|z|\frac{dz}{|z|^{d_i+\alpha_i}}
\le
\frac{C}{t}\frac{|\x-\x'|}{t^{\frac 1 2}} p_{\bar S}(t,\x)
\le \frac{C}{t}\left( \frac{|\x-\x'|}{t^{\frac 1 2}}\right)^\beta p_{\bar S}(t,\x').
\end{eqnarray}
We recall here that the second inequality follows from the fact that on the considered set, i.e.  $|\x-\x'|\le t^{\frac 12},$ $ |z|\le t^{\frac 1{\alpha_i}} $, since $\alpha_i=\frac{2}{1+(i-1)2} $, we have that $t^{-(i-1)} |(e^A)_i z|+ |\x-\x'|\le Ct^{\frac 12} $. Hence,  there exists $C>0$ s.t. for all $(\lambda,\mu)\in [0,1]^2 $, 
$$
g_{c}(t,\x'+\mu t^{-(i-1)}(e^A)_iz+\lambda (\x-\x')-\bxi) \le Cg_{c}(t,\x'-\bxi)
$$ up to a modification of $c$ (cf. \eqref{ste} for a similar estimate). We have exploited as well that:
$$ t^{-(i-1)} \int_{|z|\le t^{\frac 1{\alpha_i}}} \frac{|z|}{|z|^{d_i+\alpha_i}}dz \le C_{\alpha,i}t^{-(i-1)} t^{-1+\frac{1}{\alpha_i}}=C_{\alpha,i}t^{-1+\frac 12}.$$
If now $|\x-\x'|> t^{\frac{1}{\alpha}} $, we derive from \eqref{PRELIM_SMALL_JUMPS}:
\begin{eqnarray}
|\LALS{i}{t} p_S(t,\x)-\LALS{i}{t} p_S(t,\x')|\nonumber\\
\le \frac{C}{t^{\frac 12}}\int_{|z|\le t^{\frac 1{\alpha_i}}} \int_0^1 d\mu \Big( p_{S}(t,\x+\mu t^{-(i-1)} (e^A)_iz )+ p_{S}(t,\x'+\mu t^{-(i-1)}(e^A)_iz) \Big) \notag\\
\times t^{-(i-1)}  |z|\frac{dz}{|z|^{d_i+\alpha_i}}
\le \frac{C}{t^{\frac 12}}\Big( g_{c}(t,\x )+ g_{c}(t,\x') \Big) \int_{|z|\le t^{\frac 1{\alpha_i}}} t^{-(i-1)}|z|\frac{dz}{|z|^{d_i+\alpha_i}}\nonumber\\
\le \frac{C}{t} \Big( p_{\bar S}(t,\x)+p_{\bar S}(t,\x')\Big)\le \frac{C}{t} \left( \frac{|\x-\x'|}{t^{\frac{1}{2}}}\right)^\beta\Big( p_{\bar S}(t,\x)+p_{\bar S}(t,\x')\Big).
\label{CTR_SMALL_OUT_DIAG_2}
\end{eqnarray}
Equations \eqref{CTR_SMALL_DIAG_2} and \eqref{CTR_SMALL_OUT_DIAG_2} give the stated control for $|\LALS {i}{t}p_S(t,\x)-\LALS {i}{t} p_S(t,\x')|  $. This completes the proof of \textit{(ii)} for $\alpha=2 $. The controls \textit{(i)} and \textit{(iii)} are obtained following the lines of  the proof of Lemma \ref{SENS_SING_STAB} with the above modifications. 

\section{Singular integrals in homogeneous spaces  }
\label{APP_CW}
\subsection{The main result
}
 We recall here the basic Coifmann-Weiss theorem on singular integrals (see \cite{coif:weis:71}).
 Let $X$ be a set.
A function 
 $d: X \times X \to \R_+$ is called a {\it quasi-distance} if 
it satisfies :

1) $d(x,y) >0$ if and only if $x \not = y$;

2) $d(x,y) = d(y,x)$, for $x,y \in X$;

3) there exists $C>0$ such that 
$
d(x,z)$ $ \le C \big (d(x,y) + d (y,z) \big),$ for $x,y,z \in X.$

\smallskip \noindent We also introduce  related balls $B(x,r) = \{ y \in X \; : \; d(y,x) <r \}$, $x \in X$, $r>0$, which form a complete system of neighbourhoods of $X$, so that $X$ is a Hausdorff space. We require that balls are open sets in this topology.

 A {\it homogeneous space} is a triple $(X, d, \mu)$ where $d$ is a quasi-distance and $\mu $ is a Borel measure  such that there exists $A>0$ for which
\begin{equation} \label{ball}
 0 <  \mu (B(x,2r)) \le A \, \mu (B(x, {r}
 )) < \infty,\;\;\; x \in X, \;\; r>0.
\end{equation} 
\begin{theorem} \label{coi} [Coifman-Weiss] Let $(X,d, \mu)$ be a homogeneous space. Let  $k(x,y) \in L^2 (X \times X, \mu \otimes \mu)$ and consider the operator $K :  L^2 (X ,\mu) \to  L^2 (X ,\mu) $,
$
K f (x)$ $ = \int_X k (x,y) f(y) \mu (dy),$ $ f \in  L^2 (X ,\mu). 
$ Assume that 

\smallskip 
H1) There exists $C_1>0$ such that $\|Kf \|_{L^2} \le C_1 \| f\|_{L^2}$, for any $f \in L^2 (X ,\mu)$. 
 
H2)   There exists $C_2, \, C_3 >0$ such that, for any $y, y_0 \in X$,
\begin{align*}
\int_{ d(x, y_0) > C_2 d(y, y_0)} \, |k (x,y) - k(x,y_0)| \mu (dx) \le  C_3.
\end{align*}
Then, for any $p \in (1,2]$ there exists $A_p$ (depending  on $p$, $C_i$, $i=1,2,3$) such that for $f \in L^2 \cap L^p$ one has:
\begin{align*}
 \|Kf \|_{L^p} \le A_p \| f\|_{L^p}. 
\end{align*}
Moreover, there exists $A_1 >0$, such that    $\mu  \big ( x \in X \; :\; |Kf (x)| > \alpha \big) \le \frac{A_1}{\alpha} \| f\|_{L^1}$, $\alpha>0$, $f \in L^1 \cap L^2$.
 \end{theorem}

\subsection{The Strip $\Sc$ viewed as a Homogeneous space}

We eventually need the following topological result.  
Similar results can be found in \cite{bram:cerr:manf:96} and \cite{difr:poli:06}.
\begin{proposition}[]\label{PROP_DOUBLING}
Let $d$ be as in \eqref{DEF_QD} with $\alpha \in (0,2]$. Then $d$ is a quasi-distance on $\Sc$, i.e.  $d\big (t,\x),(s,\y) \big) =0$ if and only if $(t,\x)=(s,\y)$, $d$ is symmetric and moreover there exists  $C= C (\A{A})>0 $, s.t.  for all $(t,\x), (\sigma,\bxi), (s,\y)\in \Sc^3 $:
\begin{equation}
\label{QDT}
d\big((t,\x),(s,\y) \big)\le C\, \Big( d\big((t,\x),(\sigma,\bxi) \big)+d\big((\sigma,\bxi),(s,\y) \big)\Big).
\end{equation}
Also, $(\Sc,d, \mu)$ where $\mu$ is the Lebesgue measure
is a homogeneous space. 
\end{proposition}
\begin{proof} Let us first deal with the quasi-triangle inequality. Introduce for $\z\in \R^N$:
\begin{eqnarray*}
\rho_{{\rm Sp}}(\z)=\sum_{i=1}^n |\z_i|^{\frac{1}{1+\alpha(i-1)}},
\end{eqnarray*}
which corresponds to the \textit{spatial} contribution in the definition of $\rho $ in \eqref{rho_homo}.
For $(t,\x),  (\sigma,\bxi), (s,\y)\in \Sc^3 $ write:
\begin{eqnarray}
d\big( (t,\x),(s,\y)\big)
&\le & \frac 12 \Big( 2 |t-s|^{\frac 1\alpha}+\rho_{{\rm Sp}}\big( e^{(s-t)A}\x-e^{(s-\sigma)A}\bxi\big)+\rho_{{\rm Sp}}\big(e^{(s-\sigma)A}\bxi- \y\big)\nonumber\\
&& +\rho_{{\rm Sp}}\big( \x-e^{(t-\sigma)A}\bxi)+\rho_{{\rm Sp}}\big(e^{(t-\sigma)A}\bxi-e^{(t-s)A}\y\big) \Big).
\label{PREAL_QT}
\end{eqnarray} 
In the sequel consider  $r, u \in \R$, $\x, \bxi \in \R^N$. There exists $C =C(A) >0$ such that
\begin{equation} \label{ci11}
\rho_{{\rm Sp}}\big( e^{(r+u)A}\x-e^{r A }\bxi\big) \le C(\rho_{{\rm Sp}}(e^{ u A}\x-\bxi)+|r|^{\frac 1 \alpha}) = C\rho(r, e^{ u A}\x-\bxi).
\end{equation} 
Indeed, observe that (using the structure of the matrix $e^{tA} $, as established in  \eqref{FORM_RES}):
\begin{eqnarray*}
 \rho_{{\rm Sp}}\big( e^{(r+u)A}\x-e^{r A }\bxi\big) =
 \rho_{{\rm Sp}}\big( e^{rA}( e^{uA}\x- \bxi)\big)
=
\sum_{i=1}^n\big|\big(e^{r A} (e^{uA}\x-\bxi)\big)_i\big|^{\frac{1}{1+\alpha(i-1)}} \\ 
\le C \sum_{i=1}^n \Big(\sum_{j=1}^i |r|^{i-j}|(e^{uA}\x-\bxi)_j \big|\Big)^{\frac{1}{1+\alpha(i-1)}}
 \le C  \sum_{i=1}^n \sum_{j=1}^i |r|^{\frac{i-j}{\alpha(i-1)+1}}|(e^{ u A}\x-\bxi)_j|^{\frac{1}{1+\alpha(i-1)}}
\\
 \le  C \sum_{i=1}^n\Big(|(e^{uA}\x-\bxi)_i|^{\frac{1}{1+\alpha(i-1)}} +
\sum_{j=1}^{i-1} |r|^{\frac{i-j}{1+\alpha(i-1)}}|(e^{uA}\x-\bxi)_j|^{\frac{1}{1+\alpha(i-1)}}\Big),
\end{eqnarray*}
with the convention that $\sum_{j=1}^0=0 $. For each entry in the sum over $j$ we now use the Young inequality with exponents, $q_j=\frac{1+\alpha(i-1)}{1+\alpha(j-1)}>1, p_j=\Big(1-\frac{1+\alpha(j-1)}{1+\alpha(i-1)}\Big)^{-1}=(\alpha\frac{i-j}{1+\alpha(i-1)})^{-1}=\frac{1}{\alpha}\frac{1+\alpha(i-1)}{i-j}$. Hence:
\begin{eqnarray*}
 \rho_{{\rm Sp}}\big( e^{(r+u)A}\x-e^{r A }\bxi\big)
 \le  C \sum_{i=1}^n\Big(|(e^{uA}\x-\bxi)_i|^{\frac{1}{1+\alpha(i-1)}} +\sum_{j=1}^{i-1} \Big[\frac{|r|^{\frac 1\alpha }}{p_j} +\frac{|(e^{uA}\x-\bxi)_j|^{\frac{1}{1+\alpha(j-1)}}}{q_j}\Big]\Big)\nonumber\\
\le  C\Big( \sum_{i=1}^n|(e^{uA}\x-\bxi)_i|^{\frac{1}{1+\alpha(i-1)}} +|r|^{\frac 1\alpha}\Big)
=C(\rho_{{\rm Sp}}(e^{uA}\x-\bxi)+|r|^{\frac 1 \alpha}).\label{PREAL_INJ_QT}
\end{eqnarray*}
By \eqref{ci11} we derive $\rho_{{\rm Sp}}\big( e^{(s-t)A}\x-e^{(s-\sigma)A}\bxi\big) \le  C(\rho_{{\rm Sp}}(e^{(\sigma-t)A}\x-\bxi)+|s-\sigma|^{\frac 1 \alpha})$,  with $r = s- \sigma  $ and $u= \sigma -t$.
 
Similarly, one get:
$\rho_{{\rm Sp}}\big(e^{(t-\sigma)A}(\bxi-e^{(\sigma-s)A}\y)\big)\le C\Big(|t-\sigma|^{\frac 1\alpha}+ \rho_{{\rm Sp}}(\bxi-e^{(\sigma-s)A}\y)\Big).$ We obtain 
in \eqref{PREAL_QT}:
\begin{eqnarray*}
d\big( (t,\x),(s,\y)\big)&\le&  C \Big( |t-\sigma|^{\frac 1\alpha}+|s-\sigma|^{\frac 1\alpha}+\rho_{{\rm Sp}}\big( e^{(\sigma-t)A}\x-\bxi\big)+\rho_{{\rm Sp}}\big(e^{(s-\sigma)A}\bxi- \y\big)\nonumber\\
&& +\rho_{{\rm Sp}}\big( \x-e^{(t-\sigma)A}\bxi)+\rho_{{\rm Sp}}\big(\bxi-e^{(\sigma-s)A}\y\big) \Big)\le C\Big(d\big( (t,\x),(\sigma,\bxi)\big)+d\big( (\sigma,\bxi),(s,\y)\big)\Big).
\end{eqnarray*}
This proves \eqref{QDT}. 

Let us now define the $d$-balls and check the doubling property. Namely, for a point $(t,\x)\in \Sc $ and a given $\delta $, we introduce:
 $B\big((t,\x),\delta \big)$ $:=\{ (s,\y)\in \Sc: d\big((t,\x),(s,\y))< \delta \}.$
From the above definition, \eqref{DEF_QD} and the invariance of the Lebesgue measure $|\cdot| $ by translation, recalling as well that for all $r\in\R, \det(e^{rA})=1 $, we obtain that there exists $c:=c(\A{A})>0$ s.t. for all $\delta>0, (t,\x)\in \Sc $,
\begin{equation}
\label{FIRST_CTR_VOL}
|B((t,\x),\delta)|\le  c \delta^{\alpha+\sum_{i=1}^n d_i(1+(i-1)\alpha)}=c \delta^{\alpha+N+\sum_{i=1}^n d_i(i-1)\alpha}.
\end{equation}
Assume now that the following control holds:  there exists $\kappa:=\kappa( \A{A})>1$ s.t. for all $\big((t,\x), (s,\y)\big)\in \Sc^2 $
\begin{equation}
\label{EQUIV_FLOW}
\kappa^{-1}\rho_{}\big(s-t,  \x- e^{(t-s)A}\y\big)\le \rho_{}\big(s-t, e^{(s-t)A}\x- \y\big)\le \kappa \rho_{}\big(s-t, \x- e^{(t-s)A}\y\big).
\end{equation}
Equation \eqref{EQUIV_FLOW} means that we have equivalence  of the contributions associated with the forward and backward flows for the homogeneous pseudo-norm $\rho$.
 Then, introducing
\begin{eqnarray*}
\bar B\big((t,\x),\delta\big):=\Bigl\{(s,\y)\in \Sc:   \rho\big(s-t,e^{(s-t)A}\x-\y\big)< \frac{2\delta}{1+\kappa} \Bigr\},\\
\end{eqnarray*}
 we have that $\bar B\big((t,\x),\delta\big)\subset B\big((t,\x),\delta\big) $. Indeed, for all $(s,\y)\in \bar B\big((t,\x),\delta\big)$,
\begin{eqnarray*}
 d\big((t,\x),(s,\y)\big)&=&\frac 12 \Big( \rho(t-s,e^{(s-t)A}\x-\y)+\rho(t-s,\x-e^{(t-s)A}\y)\Big)\\
 &\le & \frac12 (1+\kappa)\rho(s-t,e^{(s-t)A}\x-\y)< \delta,
 \end{eqnarray*}
 using \eqref{EQUIV_FLOW} for the penultimate inequality.
 Since we also have, up to a modification of $c$ in \eqref{FIRST_CTR_VOL}, that for all $(t,\x)\in \Sc,\ \delta>0,\  |\bar B((t,\x),\delta)|\ge c^{-1}\delta^{\alpha+N+ \sum_{i=1}^n d_i(i-1)\alpha  
 } $, we finally get:
$$c^{-1}\delta^{\alpha+N+\sum_{i=1}^n d_i(i-1)\alpha}\le |B((t,\x),\delta)|\le c \delta^{\alpha+N+\sum_{i=1}^n d_i(i-1)\alpha},$$
which gives the doubling property of the $d$-balls in $\Sc$.

It  only remains to prove \eqref{EQUIV_FLOW}. It is  enough to prove that there exists $C>0$ s.t., for any $r \in \R$, $\x, \bxi \in \R^N,$ 
\begin{equation} \label{de11}
\rho_{}\big(r,  \x- e^{rA}\bxi\big)\le C \rho_{}\big(r, e^{-rA}\x- \bxi\big).
\end{equation} 
This follows easily from \eqref{ci11} with $u =-r$. 
 \end{proof}

 \section*{Acknowledgements}
For the first two authors, the article was prepared within the framework of a subsidy granted to the HSE by the Government of the Russian Federation for the implementation of the Global Competitiveness Program.

The last author would like to thank the Universit\'e d'Evry Val d'Essonne and the Laboratoire de Mod\'elisation Math\'ematique d'Evry (LaMME) for an invitation during which this work began.


\bibliographystyle{alpha}
 \bibliography{MyLibrary}

\end{document}